\documentclass[12pt]{amsart}
\usepackage{amsmath} % AMS-LaTeX core module
\usepackage{amsfonts} % more maths symbols
\usepackage{amssymb} % more maths symbols
\usepackage{amsthm} % extended theorem environment
\usepackage{amscd} % commutative diagrams
\usepackage{enumerate}
\usepackage{graphicx,xcolor}
\usepackage{amsopn}
\numberwithin{equation}{section}
\usepackage{wasysym} % more symbols (RIGHTarrow in proof environment)
\usepackage{yhmath} % wide parenthesis accent
\usepackage{hyperref}
\hypersetup{
colorlinks=true, 
linktocpage=true, 
linkcolor=teal, 
citecolor=brown
}
\theoremstyle{definition}

\newtheorem{definition}{Definition}[section] % DO NOT USE subsection!
\newtheorem*{definition*}{Definition} % no number is given
\theoremstyle{plain}

\newtheorem{lemma}{Lemma}[section]
\newtheorem*{lemma*}{Lemma} % no number is given

\newtheorem{proposition}{Proposition}[section]
\newtheorem*{proposition*}{Proposition} % no number is given

\newtheorem{theorem}{Theorem}[section]
\newtheorem*{theorem*}{Theorem} % no number is given
\newtheorem{corollary}{Corollary}[section]
\newtheorem*{corollary*}{Corollary} % no number is given

\newtheorem*{consequence*}{Consequence} % no number is given

\newtheorem*{conjecture*}{Conjecture} % no number is given

\theoremstyle{definition}

\newtheorem{remark}{Remark}[section]
\newtheorem*{remark*}{Remark} % (single remark) no number is given
\newtheorem*{example*}{Example} % (single example) no number is given
\newtheorem*{question*}{Question} % (single question) no number is given
\newtheorem*{exercise*}{Exercise} % (single exercise) no number is given
\newcommand{\disp}{\displaystyle}

\renewcommand{\.}{{}_{\!}} % gives a bigger space than \!
\renewcommand{\a}{\alpha}

\newcommand{\e}{\varepsilon}
\newcommand{\f}{\varphi}

\renewcommand{\l}{\lambda}

\newcommand{\m}{\mu}

\newcommand{\Om}{\Omega}

\newcommand{\s}{\sigma}

\renewcommand{\t}{\theta}

\newcommand{\cT}{\mathcal{T}}

\newcommand{\as}{\ \mbox{\raisebox{.085ex}{$:$}$\! =$} \ \;\!\.} % direct assignment (DO NOT CHANGE!)
\let\oldexists\exists
\renewcommand{\exists}{\oldexists \,} % existential quantifier
\let\oldforall\forall
\renewcommand{\forall}{\oldforall \,} % universal quantifier

\newcommand{\imp}{\ \Longrightarrow \ } % implication
\newcommand{\con}{\ \Longleftarrow \ } % converse
\newcommand{\st}{~|~} % such that
\newcommand{\Oset}{\mbox{\large $\varnothing$}} % empty set

\newcommand{\llist}[3]{#1_{#2} , \ldots , #1_{#3}} % line list = list of elements in a line
 
\newcommand{\inc}{\subseteq} % inclusion
\newcommand{\setmin}{\,\! \raisebox{0.40ex}{$\smallsetminus$} \;\!\!} % difference between sets
\newcommand{\cart}{\! \times \!} % Cartesian product

\newcommand{\1}[1]{\mathbf{1}_{\scriptscriptstyle \! #1}} % indicator function of the subset #1

\newcommand{\rest}[2]{#1_{\mathbf{|} #2}} % restriction of the correspondance/map #1 to the subset #2
\newcommand{\comp}{\. \circ \.} % composition of morphisms/correspondances/maps
\newcommand{\inv}[1]{#1^{- \! 1 \.}} % inverse of the morphism/correspondance/map #1

\renewcommand{\to}{\longrightarrow} % [f : X \to Y]
 
\newcommand{\act}{\! \cdot \!} % action operator

\newcommand{\seq}[3]{{\left( #1 \right)}_{\! #2 \geq #3}} 
\newcommand{\QQ}{\mathbf{Q}} % rational numbers
\newcommand{\RR}{\mathbf{R}} % real numbers
\newcommand{\Rn}[1]{\mathbf{R}^{\! #1 \!}} % #1-dimensional real Cartesian space
\newcommand{\intr}[1]{\overset{\ {}_{\circ}}{#1}} % interior
\newcommand{\ri}[1]{\mathrm{ri}{\left( #1 \right)}} % relative interior
\newcommand{\clos}[1]{\, \overline{\! #1}} % closure
\newcommand{\rc}[1]{\mathrm{rc}{\left( #1 \right)}} % relative closure
\newcommand{\bd}[1]{\partial #1} % boundary (closure less interior)
\newcommand{\rb}[1]{\mathrm{rb}{\left( #1 \right)}} % relative boundary

\renewcommand{\leq}{\leqslant} % less than or equal to
\renewcommand{\geq}{\geqslant} % greater than or equal to

\let\oldint\int
\renewcommand{\int}[4]{\oldint_{\! #1}^{#2}{\!\!\!\! #3 \mathrm{d} #4}} 
\newcommand{\cLL}[1]{\mathcal{L}^{\raisebox{.5ex}{$\scriptstyle #1$}} \!} 
\renewcommand{\bar}{\overline} % conjugate

\newcommand{\Vect}[1]{\mathrm{Vect}{\left( #1 \right)}} % spanned vector subspace

\let\olddet\det
\renewcommand{\det}[1]{{\olddet}{\left( #1 \right)}} % determinant of an endomorphism or a matrix

\newcommand{\Aff}[1]{\mathrm{Aff}{\left( #1 \right)}} % spanned affine subspace (affine hull)
\newcommand{\Conv}[1]{\mathrm{Conv}{\left( #1 \right)}} % convex hull
\newcommand{\norm}[1]{{\left\| #1 \right\|}} % norm
\newcommand{\onenorm}[1]{{\left\| #1 \right\|}_{\scriptscriptstyle \. 1}} % L1-norm
\newcommand{\twonorm}[1]{{\left\| #1 \right\|}_{\scriptscriptstyle \. 2}} % L2-norm
\newcommand{\scal}[2]{{\left\langle #1 , #2 \right\rangle}} % scalar product

\newcommand{\Cl}[1]{\mathrm{C}^{\raisebox{.5ex}{$\scriptstyle #1$}} \!} % class C^{#1}

\makeatletter

\renewcommand{\subsection}{
\@startsection{subsection}
{2} % Mettre 6 permet de se débarrasser de la numérotation
{0ex} % indentation : espace entre la marge de gauche et le titre
{5ex} % espace vertical au-dessus du titre
{2ex} % espacement derrière le titre 
{\bfseries} % style du texte du titre
}

\makeatother

\renewcommand{\setmin}{\raisebox{0.30ex}{\fontsize{6}{7}\selectfont $\diagdown$}} 
\newcommand{\Cone}[1]{\mathrm{Cone}\!\left(#1\right)} % cone

\newcommand{\Bcone}[1]{\mathrm{Cone_{\textsf b}}\!\left(#1\right)} % blunt conic hull

\begin{document}

\title[]{When is a Minkowski norm strictly sub-convex?}

\author{St\'{e}phane Simon}
\address{St\'{e}phane Simon, 
Universit\'{e} Savoie Mont Blanc, 
CNRS, 
LAMA, 
F--73000 Chamb\'{e}ry, 
France}
\email{Stephane.Simon@univ-smb.fr}

\author{Patrick Verovic}
\address{Patrick Verovic, 
Universit\'{e} Savoie Mont Blanc, 
CNRS, 
LAMA, 
F--73000 Chamb\'{e}ry, 
France}
\email{Patrick.Verovic@univ-smb.fr}

\date{\today}
\subjclass[2010]{Primary: 52A07, Secondary: 52A05}
\keywords{Minkowski norms, affine geometry, topological vector spaces, strict sub-convexity}

%%%%%%%%%%%%
% Abstract %
%%%%%%%%%%%%

\begin{abstract} 
The aim of this paper is to give two complete and simple characterizations of Minkowski norms $N \.$ 
on an \emph{arbitrary} topological real vector space such that 
the sublevel sets of $N \.$ are \emph{strictly convex}. 
We first show that this property is equivalent to the continuity of $N \.$ 
together with the fact that any open chord between two points of the boundary 
of the sublevel set ${\inv{N}([0 , 1) \. )}$ lies inside that set (geometric characterization). 
On the other hand, we prove that this is also the same as saying that $N \.$ is continuous 
and that for an arbitrary real number ${\a > 1}$ the function $N^{\a} \!$ is strictly convex 
(analytic characterization). 
\end{abstract}

\maketitle

\tableofcontents

\setlength{\parindent}{0pt}

\bigskip

\bigskip

\newpage

%%%%%%%%%%%%%%%%%%%%

\section*{Introduction}

In this paper, we shall be concerned with functions defined on a topological real vector space 
which are strictly sub-convex, that is, whose sublevel sets are strictly convex. 
The property of being strictly sub-convex is of course more general than that of being strictly convex. 
For example, the function ${\. f \. : \RR \to \RR}$ defined by 
${\. f \. (t) \. \as 0}$ for ${t \leq 0}$ and ${\. f \. (t) \. \as 1}$ for ${t > 0}$ 
is strictly sub-convex even though it is not strictly convex. 
This notion will be made more precise in Section~\ref{sec:About-sub-convex-functions}. 

\bigskip

We shall focus on a particular class of strictly sub-convex functions 
on an arbitrary topological real vector space which are positively homogeneous and that are called Minkowski norms (see Section~\ref{sec:Preliminaries}). 

\bigskip

Minkowski norms are the most natural generalization of the usual norms defined on real vector spaces. 
Indeed, they differ from norms in that they are not necessarily symmetric, 
and hence they give rise to pseudo-metrics instead of the classical distance functions. 
Moreover, Minkowski norms are the main mathematical objects behind Finsler metrics that are defined 
on differentiable manifolds and which are the smallest extensions of the well-known Riemannian metrics 
(see for example~\cite{BCS00} and~\cite{CheShe12}). 

\bigskip

The present work gives two complete and simple characterizations of Minkowski norms 
which are strictly sub-convex as we shall see with both Theorem~\ref{thm:Carothers-generalized} 
and Theorem~\ref{thm:Minkowski-norm-strict-sub-cvx} 
in Section~\ref{sec:Application-to-Minkowski-norms}. 

\bigskip

Indeed, in Theorem~\ref{thm:Carothers-generalized}, 
we prove that being strictly sub-convex for a Minkowski norm $N \.$ 
is the same as saying that $N \.$ is continuous and that it satisfies the following 
geometric property: any open chord between two points of the boundary 
of the sublevel set ${\inv{N}([0 , 1) \. )}$ lies inside that set. 

\bigskip

On the other hand, from the point of view of analysis, 
we show in Theorem~\ref{thm:Minkowski-norm-strict-sub-cvx} 
that the strict sub-convexity of a Minkowski norm $N$ 
is equivalent to the continuity of $N$ together with the strict convexity 
of the function $N^{\a}$ for an arbitrary real number $\a > 1$. 

\bigskip

As Minkowski norms belong to the class of non-negative functions which are positively homogeneous 
and sub-convex, Theorem~\ref{thm:Carothers-generalized} 
and Theorem~\ref{thm:Minkowski-norm-strict-sub-cvx} will be a consequence of a more general result 
given by Theorem~\ref{thm:strict-cvx-strict-qc-strict-subc} 
that we will state in Section~\ref{sec:About-sub-convex-functions} 
and which gives a relationship between strict sub-convexity 
of non-negative positively homogeneous functions and strict convexity of any of their powers 
whose exponent is greater than one. 

\bigskip

In order to build this bridge between geometric and analytic aspects of convexity, 
we shall have to give some useful properties about gauge functions 
in Section~\ref{sec:About-positive-homogeneity}---where their continuity is fully characterized---, 
convexity in Section~\ref{sec:About-convexity} 
and sub-additive functions in Section~\ref{sec:About-sub-additive-functions}. 

\bigskip

%%%%%%%%%%%%%%%%%%%%

\section{Preliminaries} \label{sec:Preliminaries}

Let us first give the definition of a Minkowski norm on a real vector space. 

\bigskip

\begin{definition} 
   Given a real vector space $V \!\!$, 
   a function ${N \. : V \! \to \RR}$ is said to be a \emph{Minkowski norm} on $V \!$ 
   whenever it satisfies the following properties: 
   
   \begin{enumerate}
      \item $N \.$ is non-negative. 
      
      \smallskip
      
      \item $N \. (\l x) = \l N \. (x)$ for any $x \in V \!$ and any real number $\l > 0$. 
      \quad (positive homogeneity) 
      
      \smallskip
      
      \item $N \. (x + y) \leq N \. (x) + N \. (y)$ for any $x , y \in V \!\.$. 
      \quad (sub-additivity) 
      
      \smallskip
      
      \item $N \. (x) \neq 0$ for any $x \in V \setmin \{ 0 \}$. 
      \quad (point-separating) 
   \end{enumerate} 
\end{definition}

\bigskip

A Minkowski norm $N \.$ on a real vector space $V \!$ gives rise to the norm $\norm{\cdot}$ on $V \!$ 
defined by ${\norm{x} \as (N \. (x) + N \. ({-x}) \. ) \! / \. 2}$ (the symmetric part of $N$) 
which coincides with $N \.$ in the case where this latter is symmetric 
(that is, if we have ${N \. ({-x}) = N \. (x)}$ for any ${x \in V}$). 

\bigskip

Here are now some basic facts about Minkowski norms that may be useful in the sequel. 

\bigskip

\begin{proposition} \label{prop:Minkowski} 
   Given a Minkowski norm $N \!$ on a real vector space $V \!\!$, we have the following properties: 
   
   \begin{enumerate}
      \item $N \. (0) = 0$. 
      
      \smallskip
      
      \item $N \. (x) \leq 2 \norm{x}$ for any $x \in V \!\!$. 
      
      \smallskip
      
      \item $|N \. (x) - N \. (y) \. | \leq \max{\! \{ N \. (x - y) \, , \, N \. (y - x) \. \}}$ 
      for any $x , y \in V \!\!$. 
      
      \smallskip
      
      \item $|N \. (x) - N \. (y) \. | \leq 2 \norm{x - y}$ for any $x , y \in V \!\!$. 
      \quad (Lipschitz continuity) 
   \end{enumerate} 
\end{proposition}

\bigskip

\begin{proof} 
The proof is easy and left to the reader. 
\end{proof}

\bigskip

Let now $\cT \.$ be the collection of all the subsets $U \!$ of $V \!$ such that for any ${a \in U \!}$ 
there exists a real number ${r > 0}$ which satisfies ${\inv{N}([0 , r]) + a \inc U \!}$. 

\bigskip

We can then easily check that $\cT \.$ is a topology on $V \!$ which contains $\inv{N}([0 , \e) \. )$ 
for any real number ${\e > 0}$, and for which every translation of $V \!$ is a homeomorphism. 
It is actually the coarsest topology on $V \!$ for which $N \.$ is continuous at the origin, 
and it is coarser than the topology on $V \!$ associated with the norm $\norm{\cdot}$. 

\bigskip

Moreover, the vector addition ${V \. \cart V \! \to V \!}$ and the scalar multiplication 
${[0 , {+\infty}) \cart V \! \to V \!}$ by non-negative real numbers are continuous functions 
when $V \!$ is endowed with the topology $\cT \!\.$. 

\bigskip

Nevertheless, the topology $\cT \.$ does not always give rise 
to a topological vector space structure on $V \!$ 
as does the topology on $V \!$ associated with the norm $\norm{\cdot}$. 

\bigskip

Indeed, in order for this to happen, it is necessary and sufficient 
that the antipodal map ${A : V \! \to V \!}$ defined by ${\! A(x) \. \as {-x}}$ 
be continuous at the origin with respect to the topology $\cT \!\!$, 
which is equivalent to having the following condition: 

\smallskip

\centerline{
There exists a real number ${C > 1}$ such that we have ${N \. ({-x}) \leq C N \. (x)}$ 
for any ${x \in V \!\!}$. \qquad ($\clubsuit$)
} 

\bigskip

It is to be noticed that this amounts to saying that the topology $\cT \.$ is nothing else 
than the topology associated with $\norm{\cdot}$. 

\bigskip

The condition~($\clubsuit$) is in particular satisfied when $V \!$ is finite dimensional. 

Indeed, in that case, the sphere ${S \as \{ x \in V \st \norm{x} = 1 \}}$ is compact 
(closed and bounded) for the norm $\norm{\cdot}$, and hence the continuity of $N \.$ 
with respect to $\norm{\cdot}$ (see Point~4 in Proposition~\ref{prop:Minkowski}) 
insures the existence of a real number ${C > 1}$ such that 
we have the inclusion ${N \. (S) \inc [2 \. / \:\!\! C , {+\infty})}$ 
owing to the point-separating property of $N \!$. 
Therefore, we get ${\norm{x} \leq (C \!\. / \. 2) N \. (x)}$ for any ${x \in V}$ 
% no \! at the end of the line
by the positive homogeneity of $N \!$, which finally implies the inequality 
${N \. ({-x}) \leq C N \. (x)}$ since we have ${N \. ({-x}) \leq 2 \norm{{-x}} = 2 \norm{x}}$ 
(see Point~2 in Proposition~\ref{prop:Minkowski}). 

\bigskip

Now, if $V \!$ is infinite dimensional, the above condition~($\clubsuit$) may be false 
as we can see with the following example. 

\smallskip

Let $V \!$ be the real vector space $\ell^{1 \!}(\RR)$ endowed with the one-norm $\onenorm{\cdot}$, 
and let $\f$ be the linear form on $V \!$ defined by 
${\disp \f(x) \. \as \!\!\! \sum_{n = 0}^{+\infty} \frac{n + 1}{n + 2} x_{n}}$. 

Then we have ${\f(x) < \onenorm{x} \.}$ for any ${x \in V \setmin \{ 0 \}}$ 
(this proves in particular that $\f$ is continuous for $\onenorm{\cdot}$), 
which implies that the function ${N \. \as \onenorm{\cdot} \. + \f}$ is a Minkowski norm on $V \!\.$. 

Therefore, if we assume that the above condition~($\clubsuit$) 
is satisfied for $N \!$, then one obtains 
${\f(x) \leq [ \. (C - 1) \:\!\! / \:\!\! (C + 1) \. ] \onenorm{x} \.}$ for any ${x \in V \!\!}$, 
which is equivalent to saying that there exists a constant ${K \in (0 , 1)}$ such that we have 
${\f(x) \leq K \onenorm{x} \.}$ for any ${x \in V \!\.}$. 

But an easy computation shows that we have 
${\sup{\! \{ \f(x) \st x \in V \! \ \mbox{and} \ \onenorm{x} \. = 1 \}} = 1}$, 
which leads to a contradiction. 

\bigskip

Among all the Minkowski norms on a real vector space, those which satisfy the conditions 
given by the following result are of particular interest for the purpose of the present paper. 

\bigskip

\begin{proposition} \label{prop:Minkowski-strict-sub-additivity} 
   Given a real vector space $V \!\.$ and a Minkowski norm $N \!$ on $V \!\!$, 
   the following properties are equivalent: 
   
   \begin{enumerate}
      \item For any non-collinear vectors $x$ and $y$ in $V \!\!$, 
      we have ${N \. (x + y) < N \. (x) + N \. (y)}$. 
      
      \smallskip
      
      \item For any two vectors ${x \neq y}$ in $V \!\.$ satisfying ${N \. (x) = N \. (y) = 1}$, 
      we have ${N \. ( \. (x + y) \! / \.2) < 1}$. 
      
      \smallskip
      
      \item For any two vectors ${x \neq y}$ in $V \!\.$ which satisfy ${N \. (x) = N \. (y) = 1}$, 
      there exists ${s \in (0 , 1)}$ such that we have ${N \. ( \. (1 - s) x + s y) < 1}$. 
      
      \smallskip
      
      \item For any two vectors ${x \neq y}$ in $V \!\.$ which satisfy ${N \. (x) = N \. (y) = 1}$ 
      and for every ${t \in (0 , 1)}$, we have ${N \. ( \. (1 - t) x + t y) < 1}$. 
   \end{enumerate} 
\end{proposition}

\bigskip

\begin{proof} \ 

\textsf{Point~1}$\iff$\textsf{Point~2.} 
This equivalence is an adaptation of~\cite[Theorem~11.1, page~110]{Car04}. 

\medskip

\textsf{Point~2}$\imp$\textsf{Point~3.} 
This is obvious by considering ${s \as 1 \. / \. 2}$. 

\medskip

\textsf{Point~3}$\imp$\textsf{Point~4.} 
Assume that Point~3 is satisfied, 
and let ${x \neq y}$ be two vectors in $V \!$ which satisfy ${N \. (x) = N \. (y) = 1}$. 

\smallskip

Therefore, there exists ${s \in (0 , 1)}$ 
such that the point ${z \as (1 - s) x + s y}$ satisfies ${N \. (z) < 1}$. 

\smallskip

Now, let us fix an arbitrary ${t \in (0 , 1)}$, and consider the point ${a \as (1 - t) x + t y}$. 

\smallskip

\textasteriskcentered \ 
If we have ${t \in (0 , s]}$, then one can write 
${a = (1 - \a) x + \a z}$ with ${\a \as t \. / \. s \in (0 , 1]}$, which yields 
${N \. (a) \leq (1 - \a) N \. (x) + \a N \. (z) = (1 - \a) + \a N \. (z) < (1 - \a) + \a = 1}$ 
since we have ${1 - \a \geq 0}$ and since $N \.$ is positively homogeneous and sub-additive. 

\smallskip

\textasteriskcentered \ 
If we have ${t \in [s , 1)}$, 
then the points ${x' \! \as y}$ and ${y' \! \as x}$ satisfy ${z = (1 - s') x' \! + s' y' \!}$ 
with ${s' \! \as 1 - s \in (0 , 1)}$, 
and hence we get ${a = (1 - t') x' \! + t' y' \!}$ 
with ${t' \! \as 1 - t \in (0 , s']}$, 
which yields ${N \. (a) < 1}$ according to the previous case. 

\medskip

\textsf{Point~4}$\imp$\textsf{Point~2.} 
This is obvious by considering ${t \as 1 \. / \. 2}$. 
\end{proof}

\bigskip

When $N \.$ is a norm on a real vector space $V \!$ 
(that is, a Minkowski norm on $V \!$ which is symmetric) that satisfies the four equivalent properties 
in Proposition~\ref{prop:Minkowski-strict-sub-additivity}, 
then the normed vector space $(V , N)$ is often called ``strictly convex'' 
in the literature as in~\cite[page~108]{Car04} and \cite[page~30]{JohLin01}, 
which is unfortunate (indeed, this expression may induce some confusion and make believe 
that it applies to $N \!$, whereas a norm cannot be strictly convex in the usual sense!). 

Therefore, some authors prefer to say that such a norm is ``rotund'' 
(see~\cite{Asp67} and~\cite{Phe89}), which is much better since 
the fourth property in Proposition~\ref{prop:Minkowski-strict-sub-additivity} 
exactly means that there is no non-trivial line segment in the unit sphere of $N \.$ 
(from an intuitive point of view, this unit sphere does not contain any ``flat piece''). 

We will see in Section~\ref{sec:Application-to-Minkowski-norms} that this can be expressed 
by a topological property of the unit ball of $N \!$, 
and we will call such a norm $N \.$ ``strictly sub-convex''. 

Moreover, we will generalize the notion of ``strict sub-convexity'' 
in Theorem~\ref{thm:Carothers-generalized} to any Minkow-ski % hyphenation
norm on an arbitrary topological real vector space. 

\bigbreak

%%%%%%%%%%%%%%%%%%%%

\section{About positive homogeneity} \label{sec:About-positive-homogeneity}

In this section, we introduce a couple of notions related to affine geometry 
(Subsection~\ref{subsec:Geometric aspects of positive homogeneity}), 
and then give some definitions and properties about positively homogeneous functions 
(Subsection~\ref{subsec:Positively homogeneous functions}) 
in order to characterize the continuity of a gauge function 
on an arbitrary topological real vector space (Subsection~\ref{subsec:Continuity of gauge functions}). 

%%%%%%%%%%%%%%%%%%%%

\subsection{Geometric aspects of positive homogeneity} 
\label{subsec:Geometric aspects of positive homogeneity}

Let us begin by recalling the definition of a cone 
(a subset of a real vector space which is closed under scalar multiplications by positive real numbers) 
and some related affine notions that will be needed in the sequel. 

\bigskip

\begin{definition} 
   A subset $C$ of a real vector space $V \!$ is said to be 
   
   \begin{enumerate}
      \item a \emph{ray} (with initial point at the origin) in $V \!$ 
      if there exists a non-zero vector ${x \in V \!}$ which satisfies 
      ${C = \{ t x \st t \geq 0 \}}$. 
      
      \smallskip
      
      \item a \emph{cone} (with apex at the origin) in $V \!$ 
      if it satisfies ${\l C \inc C}$ for any scalar ${\l > 0}$ 
      (it is said to be \emph{pointed} if it contains the origin, and \emph{blunt} otherwise).
   \end{enumerate} 
\end{definition}

\bigskip

\begin{remark} \label{rem:coneop} \ 

\begin{enumerate}[1)]
   \item A pointed cone is characterized by the fact that its intersection with any ray 
   reduces to that ray or to the origin. Therefore, pointed cones not reduced 
   to the origin coincide with arbitrary unions of rays. 
   
   \medskip
   
   \item In particular, the empty set is a blunt cone, 
   and any union or intersection of cones is also a cone. 
   Moreover, the complement of a cone is a cone, and a product of cones is a cone too. 
   
   \medskip
   
   \item In the case where $V \!$ is a topological real vector space, the interior, 
   the closure and the boundary of a cone in V are also cones in $V \!\.$. 
\end{enumerate} 
\end{remark}

\bigskip

\begin{definition} 
   A subset $S$ of a real vector space is said to be \emph{star-shaped} (about the origin) 
   if for any ${x \in S}$ and ${t \in [0 , 1]}$ we have ${t x \in S \.}$. 
\end{definition}

\bigskip

\begin{definition} 
   Given a subset $S$ of a real vector space $V \!\!$, 
   the \emph{star-shaped hull} $\widehat{S}$ of $S$ 
   is the smallest star-shaped subset of $V \!$ which contains $S \.$. 
\end{definition}

\bigskip

In other words, we have ${\widehat{S} = [0 , 1] S \.}$. 

\bigskip

\begin{definition} 
   The \emph{pointed conic hull} $\Cone{S}$ of a subset $S$ of a real vector space $V \!$ 
   is the smallest cone in $V \!$ which contains ${S \cup \{ 0 \}}$. 
   The \emph{blunt conic hull} $\Bcone{S}$ of $S$ is the smallest cone 
   which contains ${S \,\. \setmin \{ 0 \}}$. In other words, we have 
   ${\Cone{S} = \{ \l x \st \l > 0 \ \ \mbox{and} \ \ x \in S \cup \{ 0 \} \. \}}$ 
   and ${\Bcone{S} = \Cone{S} \. \setmin \{ 0 \}}$. 
\end{definition}

\bigskip

According to this definition, one has 
${S \cup \{ 0 \} \inc \Cone{S}}$ and ${S \,\. \setmin \{ 0 \} \inc \Bcone{S}}$. 

\bigskip

It is to be noticed that for any ${x \in V \setmin \{ 0 \}}$ the pointed conic hull of $\{ x \}$ 
is nothing else than the ray in $V \!$ passing through $x$. 

\bigbreak

\begin{remark} \label{rem:cone} \ 

\begin{enumerate}[1)]
   \item For any subset $S$ of a real vector, we of course have 
   
   %\smallskip
   
   \centerline{
   $\Cone{S \,\. \setmin \{ 0 \} \.} \ = \ \Cone{S} \ = \ \Cone{S \cup \{ 0 \} \.}$~,
   } 
   
   %\smallskip
   
   \centerline{
   $\Bcone{S \,\. \setmin \{ 0 \} \.} \ = \ \Bcone{S} \ = \ \Bcone{S \cup \{ 0 \} \.}$ 
   \qquad and \qquad 
   $\widehat{S} \ \inc \ \Cone{S}$~.
   } 
   
   \medskip
   
   \item For any subsets $A$ and $B$ of a real vector space which satisfy ${A \inc B}$, the inclusions 
   
   %\smallskip
   
   \centerline{
   $\Cone{A} \ \inc \ \Cone{B}$ 
   \qquad and \qquad 
   $\Bcone{A} \ \inc \ \Bcone{B}$
   } 
   
   %\smallskip
   
   \noindent are straightforward. 
   
   \medskip
   
   \item For any subset $S$ of a real vector, we have 
   ${\Cone{\widehat{S} \,\. \setmin \{ 0 \} \.} = \Cone{S}}$ 
   by Points~1 and~2 above. 
   
   \medskip
   
   \item For any subsets $A$ and $C$ of a real vector space, if $C$ is a cone, then we have 
   
   %\smallskip
   
   \centerline{
   $\Cone{A \cap C} \ = \ \Cone{A} \cap (C \cup \{ 0 \} \. )$ 
   \qquad and \qquad 
   $\Bcone{A \cap C} \ = \ \Bcone{A} \cap C \.$~.
   } 
\end{enumerate} 
\end{remark}

\bigskip

\begin{definition} \label{def:vect} 
   Given a subset $S$ of a real vector space $V \!\!$, 
   the vector subspace $\Vect{S}$ of $V$ % no \! at the end of the line
   \emph{spanned} by $S$ is the smallest linear subspace of $V \!$ which contains $S \.$. 
\end{definition}

\bigskip

Therefore, one has ${\Vect{S \,\. \setmin \{ 0 \} \.} = \Vect{S}}$. 
Moreover, for any subsets $A$ and $B$ of $V \!$ which satisfy ${A \inc B}$, 
we obviously have ${\Vect{A} \inc \Vect{B}}$. 

\bigskip

\begin{definition} \label{def:aff-hull} 
   Given a subset $S$ of a real vector space $V \!\!$, the \emph{affine hull} $\Aff{S}$ of $S$ 
   is the smallest affine subspace of $V \!$ which contains $S \.$. 
   In other words, the affine hull of $S$ is equal to the set of points ${x \in V \!}$ 
   which write ${\disp x = \! \sum_{i = 1}^{n} \l_{i} x_{i}}$ 
   for some integer ${n \geq 1}$, some points ${\llist{x}{1}{n} \in S}$ \linebreak
   
   \vspace{-8pt}
   
   and some real numbers ${\llist{\l}{1}{n}}$ 
   which satisfy ${\disp \,\! \sum_{i = 1}^{n} \l_{i} = 1}$. 
\end{definition}

\medskip %\bigskip

Therefore, for any subset $S$ of $V \!\!$, we have ${\Aff{S} \inc \Vect{S}}$. 

\bigskip

Moreover, for any subsets $A$ and $B$ of $V \!$ which satisfy ${A \inc B}$, 
we of course have ${\Aff{A} \inc \Aff{B}}$. 

\bigskip

It is to be noticed that for any subset $S$ of $V \!$ 
the union of all the lines passing through two distinct points of $S$ is exactly $\Aff{S}$. 

\bigskip

Once all these definitions have been recalled, let us now give some useful relationships 
between the affine operations $\mathrm{Cone}$, $\mathrm{Aff}$ and $\mathrm{Vect}$. 

\bigskip

\begin{proposition} \label{prop:vect-vs-aff-bis} 
   For any subset $S \.$ of a real vector space $V \!\!$, 
   the following three properties are equivalent: 
   
   \begin{enumerate}
      \item $0 \in \Aff{S}$. 
      
      \smallskip
      
      \item $\Aff{S} = \Aff{S \cup \{ 0 \} \.}$. 
      
      \smallskip
      
      \item $\Aff{S} = \Vect{S}$. 
   \end{enumerate}
   
   Moreover, we have $\Vect{{\Cone{S}} \.} = \Aff{S \cup \{ 0 \} \.} = \Vect{S}$. 
\end{proposition}

\bigskip

\begin{proof} \ 

\textsf{Point~1}$\imp$\textsf{Point~2.} 
Assume that we have $0 \in \Aff{S}$. 

\smallskip

This implies $\Aff{S} \cup \{ 0 \} \inc \Aff{S}$, 
and hence $\Aff{\Aff{S} \cup \{ 0 \} \.} \inc \Aff{S}$. 

\smallskip

But, on the other hand, 
we obviously have $S \inc S \cup \{ 0 \} \inc \Aff{S} \cup \{ 0 \}$, which yields 

\smallskip

\centerline{
$\Aff{S} \ \inc \ \Aff{S \cup \{ 0 \} \.} \ \inc \ \Aff{\Aff{S} \cup \{ 0 \} \.}$.
} 

\smallskip

Therefore, we get $\Aff{S} = \Aff{S \cup \{ 0 \} \.}$. 

\medskip

\textsf{Point~2}$\imp$\textsf{Point~3.} 
Assume that we have ${\Aff{S} = \Aff{S \cup \{ 0 \} \.}}$, and pick ${x \in \. \Vect{S}}$ which writes 
${x = \l_{1} x_{1} + \cdots + \l_{n} x_{n}}$ for some integer ${n \geq 1}$, 
some points ${\llist{x}{1}{n} \in S}$ and some real numbers $\llist{\l}{1}{n}$. 

\smallskip

Then ${\l \as \l_{1} + \cdots + \l_{n}}$ satisfies 
${x = \l_{1} x_{1} + \cdots + \l_{n} x_{n} + (1 - \l) 0}$, 
which shows that $x$ is in $\Aff{S \cup \{ 0 \} \.}$ 
since we have ${\l_{1} + \cdots + \l_{n} + (1 - \l) = 1}$. 

\smallskip

So, we proved the inclusion ${\Vect{S} \inc \Aff{S}}$, and hence we get ${\Vect{S} = \Aff{S}}$ 
since we always have ${\Aff{S} \inc \Vect{S}}$. 

\medskip

\textsf{Point~3}$\imp$\textsf{Point~1.} 
This implication is straigthforward since we have ${0 \in \. \Vect{S}}$. 

\medskip

Let us now prove the last point in Proposition~\ref{prop:vect-vs-aff-bis}. 

\smallskip

First of all, since we have ${0 \in S \cup \{ 0 \} \inc \Aff{S \cup \{ 0 \} \.}}$, 
we get ${\Aff{S \cup \{ 0 \} \.} = \Vect{S \cup \{ 0 \} \.}}$ 
by the previous implications Point~1$\imp$Point~2$\imp$Point~3, 
were $S$ is replaced by ${S \cup \{ 0 \}}$. 

\smallskip

Therefore, this yields $\Aff{S \cup \{ 0 \} \.} = \Vect{S}$ 
since we always have $\Vect{S \cup \{ 0 \} \.} = \Vect{S}$. 

\smallskip

Moreover, we have ${S \inc \Cone{S} \inc \Vect{S}}$, and hence 
${\Vect{S} \inc \Vect{{\Cone{S}} \.} \inc \Vect{S}}$, which writes ${\Vect{S} = \Vect{{\Cone{S}} \.}}$. 
\end{proof}

\bigskip

\begin{proposition} \label{prop:vect-vs-aff} 
   Given a real vector space $V \!\!$, the following properties hold: 
   
   \begin{enumerate}
      \item For any $x \in V \!\!$, we have $\Aff{\. (0 , 1) x} = \Aff{[0 , {+\infty}) x} = \RR x$. 
      
      \smallskip
      
      \item For any subset $S \.$ of $V \!\.$ which satisfies $S \,\. \setmin \{ 0 \} \neq \Oset \.$, 
      we have 
      
      %\smallskip
      
      \centerline{
      $\Aff{\widehat{S}} \ = \ \Aff{\. (0 , 1)S} \ = \ \Aff{[ \. (0 , 1) S] \,\. \setmin \{ 0 \} \.} 
      \ = \ 
      \Aff{{\Bcone{S}} \.} \ = \ \Aff{{\Cone{S}} \.}$~.
      } 
   \end{enumerate} 
\end{proposition}

%\bigskip
%\pagebreak

\begin{proof} \ 

\textsf{Point~1.} 
Let us fix a vector $x \in V \!\.$. 

\smallskip

First of all, since one has ${(0 , 1) x \inc [0 , {+\infty}) x \inc \RR x}$, we get 

\smallskip

\centerline{
$\Aff{(0 , 1) x} \ \inc \ \Aff{[0 , {+\infty}) x} \ \inc \ \Aff{\RR x} \ = \ \RR x$~.
} 

\smallskip

On the other hand, given any ${\l \in \RR}$, we have 
${\l x = (2 - 3 \l) \act (1 \. / \. 3) x + (3 \l - 1) \act (2 \. / \. 3) x}$, 
which shows that $\l x$ lies in $\Aff{\. (0 , 1) x}$ since one has ${(2 - 3 \l) + (3 \l - 1) = 1}$ 
and since $(1 \. / \. 3) x$ and $(2 \. / \. 3) x$ belong to $(0 , 1) x$. 

\smallskip

This proves the inclusion ${\RR x \inc \Aff{\. (0 , 1) x}}$, and hence the equalities 

\smallskip

\centerline{
$\Aff{\. (0 , 1) x} \ = \ \Aff{[0 , {+\infty}) x} \ = \ \RR x$~.
} 

\medskip

\textsf{Point~2.} 
Let us fix a subset $S$ of $V \!$ which satisfies $S \,\. \setmin \{ 0 \} \neq \Oset \.$. 

\smallskip

\textasteriskcentered \ 
First of all, for any ${x \in S \.}$, one has 

\smallskip

\centerline{
$[0 , {+\infty}) x \ \inc \ \Aff{[0 , {+\infty}) x} 
\ = \ 
\Aff{\. (0 , 1) x} \ \inc \ \Aff{\. (0 , 1) S}$
} 

\smallskip

by Point~1 above, and hence 
${\disp \Cone{S} = \! \bigcup_{x \in S} \! [0 , {+\infty}) x \inc \Aff{\. (0 , 1) S}}$ 
since we have ${S \neq \Oset \.}$, 
which yields the inclusion ${\Aff{{\Cone{S}} \.} \inc \Aff{\. (0 , 1) S}}$. 

\smallskip

\textasteriskcentered \ 
On the other hand, since one has ${(0 , 1) S \inc \widehat{S} \inc \Cone{S}}$ 
by Point~1 in Remark~\ref{rem:cone}, we get 
${\Aff{\. (0 , 1) S} \inc \Aff{\widehat{S}} \inc \Aff{{\Cone{S}} \.}}$, and hence 

\smallskip

\centerline{
$\Aff{\. (0 , 1) S} \ = \ \Aff{\widehat{S}} \ = \ \Aff{{\Cone{S}} \.}$
} 

\smallskip

owing to the previous step. 

%\smallskip
\pagebreak

\textasteriskcentered \ 
Let us now notice that we obviously have 
${[ \. (0 , 1) S] \,\. \setmin \{ 0 \} = (0 , 1) [S \,\. \setmin \{ 0 \} \. ]}$, 
which implies 

\smallskip

\centerline{
$\Aff{[ \. (0 , 1) S] \,\. \setmin \{ 0 \} \.} \ = \ \Aff{\. (0 , 1) [S \,\. \setmin \{ 0 \} \. ]} 
\ = \ 
\Aff{{\Cone{S \,\. \setmin \{ 0 \} \.}} \.} \ = \ \Aff{{\Cone{S}} \.}$
} 

\smallskip

owing to Point~1 in Remark~\ref{rem:cone} and by replacing $S$ by ${S \,\. \setmin \{ 0 \} \neq \Oset}$ 
in the equality 

\smallskip

\centerline{
$\Aff{\. (0 , 1) S} \ = \ \Aff{{\Cone{S}} \.}$
} 

%\smallskip

obtained in the previous step. 

\smallskip

\textasteriskcentered \ 
Finally, since one has the inclusions 
${(0 , 1) [S \,\. \setmin \{ 0 \} \. ] \inc \Bcone{S} \inc \Cone{S}}$, we get 

\smallskip

\centerline{
$\Aff{[ \. (0 , 1) S] \,\. \setmin \{ 0 \} \.} 
\ \inc \ \Aff{{\Bcone{S}} \.} \ \inc \ \Aff{{\Cone{S}} \.}$~,
} 

\smallskip

and hence ${\Aff{{\Bcone{S}} \.} = \Aff{{\Cone{S}} \.}}$ owing to the equality 

\smallskip

\centerline{
$\Aff{[(0 , 1) S] \,\. \setmin \{ 0 \}} \ = \ \Aff{{\Cone{S}} \.}$
} 

%\smallskip

obtained in the previous step. 
\end{proof}

\bigskip

\begin{definition} 
   Given a point $x$ in a real vector space $V \!\!$, 
   a subset $S$ of $V \!$ is said to \emph{absorb} $x$ if there exists a real number ${\l > 0}$ 
   such that we have ${x \in \l S \.}$. 
   
   We will say that $S$ is \emph{absorbing} if $S$ absorbs any ${x \in V \setmin \{ 0 \}}$. 
\end{definition}

\bigskip

\begin{remark} \label{rem:abs-cone} 
In other words, $S$ absorbs ${x \in V \setmin\{ 0 \}}$ if and only if we have ${x \in \Bcone{S}}$. 
Moreover, $S$ absorbs $0$ if and only if $S$ contains $0$. 
\end{remark}

\bigskip

\begin{proposition} \label{prop:abs-cone} 
   Given a subset $S \.$ of a real vector space $V \!\!$, the following properties hold: 
   
   \begin{enumerate}
   \item We have $\Aff{{\Cone{S}} \.} = \Vect{S}$. 
   
   \smallskip
   
   \item If we have $S \,\. \setmin \{ 0 \} \neq \Oset \.$, 
   then we get $\Aff{{\Bcone{S}} \.} = \Vect{S}$. 
   
   \smallskip
   
   \item If there exists $x \in V \!\.$ such that $S \.$ absorbs both $x$ and ${-x}$, 
   then we have $\Aff{S} = \Vect{S}$. 
   
   \smallskip
   
   \item We have the equivalence ($S \.$ is absorbing) $\iff$ $\Cone{S} = V \!\!$. 
   
   \noindent Moreover, having both $\Cone{S} = V \!\.$ and $S \neq \Oset$ 
   implies ${\Aff{S} = \Vect{S} = V}$. 
\end{enumerate} 
\end{proposition}

\bigskip

\begin{proof} \ 

\textsf{Point~1.} 
Since $\Cone{S}$ contains the origin, we have ${\Cone{S} = \Cone{S} \cup \{ 0 \}}$, 
and hence $\Aff{{\Cone{S}} \.} = \Aff{{\Cone{S} \cup \{ 0 \}} \.}$. 

\smallskip

Now, according to Proposition~\ref{prop:vect-vs-aff-bis} with $\Cone{S}$ instead of $S \.$, one has 

\smallskip

\centerline{
$\Aff{{\Cone{S} \cup \{ 0 \}} \.} \ = \ \Vect{{\Cone{{\Cone{S}} \.}} \.} \ = \ \Vect{{\Cone{S}} \.}$~,
} 

%\smallskip

which yields 

%\smallskip

\centerline{
$\Aff{{\Cone{S}} \.} = \Aff{{\Cone{S} \cup \{ 0 \}} \.} = \Vect{S}$
} 

\smallskip

since we have $\Vect{{\Cone{S}} \.} = \Vect{S}$ 
owing to Proposition~\ref{prop:vect-vs-aff-bis} once again. 

\medskip

\textsf{Point~2.} 
Assume that we have $S \,\. \setmin\{ 0 \} \neq \Oset \.$. 

\smallskip

Then one obtains ${\Aff{{\Bcone{S}} \.} = \Aff{{\Cone{S}} \.}}$ by Proposition~\ref{prop:vect-vs-aff}, 
which yields the equality $\Aff{{\Bcone{S}} \.} = \Vect{S}$ owing to Point~1 above. 

\medskip

\textsf{Point~3.} 
Let $x \in V \!$ such that $S$ absorbs both $x$ and ${-x}$, 
which means that there exist ${\l > 0}$ and ${\m > 0}$ satisfying $\l x \in S$ and ${-\m} x \in S \.$. 

\smallskip

Since we have 

%\smallskip

\centerline{
$\disp \l + \m \ > \ 0 
\qquad \mbox{and} \qquad 
0 \ = \ \frac{\m}{\l + \m}(\l x) + \frac{\l}{\l + \m}({-\m} x)$~,
} 

\smallskip

one gets $0 \in \Aff{S}$, which yields $\Aff{S} = \Vect{S}$ 
by Point~3 in Proposition~\ref{prop:vect-vs-aff-bis}.

%\medskip
\pagebreak

\textsf{Point~4.} 

\textasteriskcentered \ 
The equivalence is given by Remark~\ref{rem:abs-cone}. 

\smallskip

\textasteriskcentered \ 
Assume that we have $\Cone{S} = V \!$ and $S \neq \Oset \.$. 

\smallskip

If $S$ contains the origin, then $S$ absorbs $0$ by Remark~\ref{rem:abs-cone}, 
and hence we get ${\Aff{S} = \Vect{S}}$ by Point~3 above with $x \as 0$. 

\smallskip

Otherwise, $S$ contains a non-zero vector $x \in V \!\!$, 
and hence it absorbs $x$ and ${-x}$ according to the equivalence previously established, 
which yields $\Aff{S} = \Vect{S}$ by Point~3 above. 

\smallskip

Moreover, since we always have ${\Cone{S} \inc \Vect{S}}$, 
the equality ${\Cone{S} = V \!}$ implies the inclusion $V \inc \Vect{S}$, 
that is, $\Vect{S} = V \!\.$. 
\end{proof}

%%%%%%%%%%%%%%%%%%%%

\subsection{Positively homogeneous functions} \label{subsec:Positively homogeneous functions}

We shall now deal with positively homogeneous functions 
and give some useful properties of their sublevel sets. 

\bigskip

\begin{definition} 
Let $V \!\!$, $W \!$ be real vector spaces and $C$ a cone in $V \!\.$. 
Given a real number ${\a > 0}$, a map ${\. f \. : C \to W}$ is said to be \emph{positively homogeneous 
of degree $\a$} if ${\. f \. (\l x) = \l^{\. \a} \. f \. (x)}$ holds for any $x \in C$ 
and any real number $\l > 0$. 

In the particular case where one has $\a = 1$, 
we merely say that $\. f \.$ is \emph{positively homogeneous}. 
\end{definition}

\bigskip

If we have $W \! \as \RR$, then the word function is preferred to that of map. 

\bigskip

\begin{remark} \label{rem:ph-z} 
For any real vector spaces $V \!$ and $W \!\!$, it is clear that a positively homogeneous map 
$\. f \. : C \to W \!$ of degree $\a  >0$ defined on a cone $C$ in $V \!$ 
satisfies $C \cap \{ 0 \} \inc \inv{f}(0)$. 
\end{remark}

\bigskip

\begin{definition} 
Given a set $X \.$, a function $\. f \. : X \! \to \RR$ and a number $r \in \RR$, 
the \emph{sublevel set} of $\. f \.$ associated with $r$ is defined by 

\smallskip

\centerline{
$S_{r}(f) \. \as \{ x \in \. X \. \st f \. (x) \leq r \} \.$~.
} 
\end{definition}

\bigskip

\begin{remark} \label{rem:sublevel-basic} \ 

\begin{enumerate}[1)]
   \item It is straightforward that the family ${{(S_{r}(f) \. )}_{\! r \in \RR \.}}$ 
   is non-decreasing: for any ${r , r' \. \in \RR}$ which satisfy $r' \. \leq r$, 
   we have the inclusion $S_{r'}(f) \inc S_{r}(f)$. 
   
   \medskip
   
   \item On the other hand, it is useful to notice that for any $r \in \RR$ 
   we have $\disp S_{r}(f) = \:\!\! \bigcap_{a > r} \! S_{a}(f)$. 
\end{enumerate} 
\end{remark}

\bigskip

From now on, we will focus on functions $\. f \. : S \to \RR$ 
defined on a subset $S$ of a real vector space. 

\bigskip

\begin{proposition} \label{prop:ph-sl} 
   Let $C \.$ be a cone in a real vector space and ${\. f \. : C \to \RR}$ 
   a positively homogeneous function of degree ${\a > 0}$. 
   Then for any real number ${r > 0}$, we have the following properties: 
   
   \begin{enumerate}
      \item $S_{r}(f) = r^{1 \. / \. \a} S_{1 \.}(f)$. 
      
      \smallskip
      
      \item $S_{r}(f) \cup \{ 0 \}$ is star-shaped. 
      
      \smallskip
      
      \item If $C \.$ is not empty, then neither is $S_{r}(f)$, 
      and hence we have $\widehat{S_{r}(f)} = S_{r}(f) \cup \{ 0 \}$. 
   \end{enumerate} 
\end{proposition}

\bigskip

\begin{proof} 
Let $r$ be a positive real number.

\medskip

\textsf{Point~1.} 
For any $x \in C \.$, we have the equivalences 
%%%%%%%%%%
\begin{eqnarray*} 
   x \in S_{r}(f) 
   & \iff & \. 
   f \. (x) \ \leq \ r \ = \ (r^{1 \. / \. \a})^{\. \a} \\ 
   & \iff & \. 
   f \. (x) \:\!\! / \:\!\! (r^{1 \. / \. \a})^{\. \a} \. \ \leq \ 1 \\ 
   & \iff & \. 
   f \. (x \. / \. r^{1 \. / \. \a}) \ \leq \ 1 \\ 
   & \iff & 
   x \in r^{1 \. / \. \a} S_{1 \.}(f)~. 
\end{eqnarray*} 
%%%%%%%%%%

\medskip

\textsf{Point~2.} 
Let us consider the set $S \as S_{r}(f) \cup \{ 0 \}$. 

\smallskip

For any $t \in (0 , 1]$, we then have 

\smallskip

\centerline{
$t S \ = \ t S_{r}(f) \cup \{ 0 \} 
\ = \ 
t r^{1 \. / \. \a} S_{1 \.}(f) \cup \{ 0 \} 
\ = \ 
(t^{\a} r)^{\. 1 \. / \. \a} S_{1 \.}(f) \cup \{ 0 \} 
\ = \ 
S_{t^{\a} r}(f) \cup \{ 0 \} \ \inc \ S$
} 

\smallskip

by Point~1 and Point~1 in Remark~\ref{rem:sublevel-basic} with $r' \! \as t^{\a} r \leq r$. 

\smallskip

Moreover, since one has $0 S \inc S \.$, we obtain that $S$ is star-shaped. 

\medskip

\textsf{Point~3.} 
Assume that $C$ is not empty, pick $x$ in $C \.$, and consider $\l \as |f \. (x) \. | + 1 > 0$. 

\smallskip

Then we obtain $S_{\. f \. (x)}(f) \inc S_{\l}(f)$ by Point~1 in Remark~\ref{rem:sublevel-basic} 
since we have $\. f \. (x) \leq \l$. 

\smallskip

Now, we can write ${S_{\l}(f) = (\l \. / \. r)^{\. 1 \. / \. \a} S_{r}(f)}$ by Point~1 above, 
which proves that $S_{r}(f)$ is not empty since $S_{\. f \. (x)}(f)$ contains $x$. 

\smallskip

Therefore, the non-emptyness of $S_{r}(f)$ implies that $\widehat{S_{r}(f)}$ contains $\{ 0 \} \.$, 
and hence contains $S_{r}(f) \cup \{ 0 \}$ since we obviously have $S_{r}(f) \inc \widehat{S_{r}(f)}$. 

\smallskip

On the other hand, the obvious inclusion ${S_{r}(f) \inc S_{r}(f) \cup \{ 0 \}}$ yields 
${\widehat{S_{r}(f)} \inc S_{r}(f) \cup \{ 0 \}}$ 
since ${S_{r}(f) \cup \{ 0 \}}$ is star-shaped by Point~2. 
\end{proof}

\bigskip

\begin{proposition} \label{prop:ap} 
   Let $C \.$ be a cone in a real vector space, ${\. f \. : C \to \RR}$ a non-negative function 
   and ${\a > 0}$ a real number. Then the following properties are equivalent: 
   
   \begin{enumerate}
      \item The function $\. f \.$ is positively homogeneous of degree $\a$. 
      
      \smallskip
      
      \item The function $f^{1 \. / \. \a} \!$ is positively homogeneous. 
      
      \smallskip
      
      \item We have $S_{r}(f) = r^{1 \. / \. \a} S_{1 \.}(f)$ for any real number $r > 0$. 
   \end{enumerate} 
\end{proposition}

\bigskip

\begin{proof} \ 

\textsf{Point~1}$\iff$\textsf{Point~2.} 
This is obvious. 

\medskip

\textsf{Point~1}$\imp$\textsf{Point~3.} 
This is Point~1 in Proposition~\ref{prop:ph-sl}. 

\medskip

\textsf{Point~3}$\imp$\textsf{Point~1.} 
Assume that Point~3 is satisfied, and fix $x \in C$ and a real number $\l > 0$. 

\smallskip

Then, for any real number $r > 0$, we have the equivalences 

\smallskip

\centerline{
$\disp \l^{\. \a} \. f \. (x) \ \leq \ r 
\iff 
x \in S_{\. \frac{r}{\l^{\. \a}}}(f) 
\ = \ 
\left( \. \frac{r}{\l^{\. \a}} \. \right)^{\!\. 1 \. / \. \a} \!\!\!\!\! S_{1 \.}(f) 
\ = \ 
\frac{1}{\l} \! \left[ r^{1 \. / \. \a} S_{1 \.}(f) \. \right] 
\ = \ 
\frac{1}{\l}S_{r}(f) 
\iff \. 
f \. (\l x) \ \leq \ r$~.
} 

\smallskip

Now, given any real number $\e > 0$, 
we obtain $\l^{\. \a} \. f \. (x) \leq f \. (\l x) + \e$ 
and $\. f \. (\l x) \leq \l^{\. \a} \. f \. (x) + \e$ 
by choosing $r \as \. f \. (\l x) + \e > 0$ and $r \as \l^{\. \a} \. f \. (x) + \e > 0$, respectively. 

Conclusion: since these two inequalities are true for any $\e > 0$, 
we have $\. f \. (\l x) = \l^{\. \a} \. f \. (x)$. 
\end{proof}

\bigskip

\begin{proposition} \label{prop:1-sublevel-f-span-V} 
   Given a real vector space $V \!\.$ and a function ${\. f \. : V \! \to \RR}$ 
   which is positively homogeneous, the sublevel set $S_{1 \.}(f)$ is absorbing, and hence satisfies 
   
   \smallskip
   
   \centerline{
   $\Cone{S_{1 \.}(f) \.} \ = \ \Aff{S_{1 \.}(f) \.} \ = \ \Vect{S_{1 \.}(f) \.} \ = \ V \!\!$~.
   } 
\end{proposition}

\bigskip

\begin{proof} 
Given ${x \in V \setmin \{ 0 \}}$, we can write ${x = \l y}$ with ${\l \as |f \. (x) \. | + 1 > 0}$ 
and ${y \as x \. / \. \l \in S_{1 \.}(f)}$ (since $\. f \.$ is positively homogeneous). 

\smallskip

This proves that $S_{1 \.}(f)$ is absorbing. 

\smallskip

Now, since ${S \as S_{1 \.}(f)}$ contains the origin by Remark~\ref{rem:ph-z}, it is not empty, 
and hence Point~4 in Proposition~\ref{prop:abs-cone} gives the equalities to be proved. 
\end{proof}

\bigskip

\begin{corollary} \label{cor:1-sublevel-f-span-C} 
   Given a cone $C \.$ in a real vector space $V \!\.$ 
   and ${\. f \. : C \to \RR}$ a positively homogeneous function, the following properties hold: 
   
   \begin{enumerate}
      \item $\Cone{S_{1 \.}(f) \.} = C \cup \{ 0 \} \.$. 
      
      \smallskip
      
      \item $\Vect{S_{1 \.}(f) \.} = \Vect{C}$. 
      
      \smallskip
      
      \item If $C \.$ is not empty, then we have $\Aff{S_{1 \.}(f) \.} = \Vect{C}$. 
   \end{enumerate} 
\end{corollary}

\bigskip

\begin{proof} \ 

\textsf{Point~1.} 
Let us consider the function ${g : V \! \to \RR}$ defined by 
${g(x) \. \as \. f \. (x)}$ for ${x \in C}$ and ${g(x) \. \as 0}$ for ${x \not \in C \.}$. 

\smallskip

Since $g$ is positively homogeneous, 
we can write $\Cone{S_{1 \.}(g) \.} = V \!$ by Proposition~\ref{prop:1-sublevel-f-span-V}, and hence 

\smallskip

\centerline{
$\Cone{S_{1 \.}(g) \.} \cap (C \cup \{ 0 \} \. ) 
\ = \ 
V \cap (C \cup \{ 0 \} \. ) 
\ = \ 
C \cup \{ 0 \} \.$~,
} 

\smallskip

which yields ${\Cone{S_{1 \.}(f) \.} = C \cup \{ 0 \}}$ owing to Point~4 in Remark~\ref{rem:cone} 
with ${A \as S_{1 \.}(g)}$ and since we obviously have ${S_{1 \.}(g) \cap C = S_{1 \.}(f)}$. 

\medskip

\textsf{Point~2.} 
From Point~1 above, we immediately obtain 
${\Aff{{\Cone{S_{1 \.}(f) \.}} \.} = \Aff{C \cup \{ 0 \} \.}}$, 
and this yields ${\Vect{S_{1 \.}(f) \.} = \Vect{C}}$ 
since one has ${\Aff{C \cup \{ 0 \} \.} = \Vect{C}}$ 
by Proposition~\ref{prop:vect-vs-aff-bis} 
and ${\Aff{{\Cone{S_{1 \.}(f) \.}} \.} = \Vect{S_{1 \.}(f) \.}}$ 
by Point~1 in Proposition~\ref{prop:abs-cone}. 

\medskip

\textsf{Point~3.} 
Assume that $C$ is not empty, 
and let us prove that the set ${S \as S_{1 \.}(f)}$ satisfies ${\Aff{S} = \Vect{S}}$ 
by considering two cases. 

\smallskip

\textasteriskcentered \ 
First case: $S \,\. \setmin \{ 0 \} = \Oset \.$. 

\smallskip

Then $S$ contains the origin since it is not empty owing to Point~3 in Proposition~\ref{prop:ph-sl}, 
and hence we get ${0 \in \Aff{S}}$, which yields ${\Aff{S} = \Vect{S}}$ 
by Point~3 in Proposition~\ref{prop:vect-vs-aff-bis}. 

\smallskip

\textasteriskcentered \ 
Second case: $S \,\. \setmin \{ 0 \} \neq \Oset \.$. 

\smallskip

Since $S \cup \{ 0 \}$ is star-shapped by Point~2 in Proposition~\ref{prop:ph-sl}, we have 

\smallskip

\centerline{
$(0 , 1) S \cup \{ 0 \} 
\ = \ 
(0 , 1) (S \cup \{ 0 \} \. ) \ \inc \ S \cup \{ 0 \} \.$~,
} 

%\smallskip

which implies 

%\smallskip

\centerline{
$[ \. (0 , 1) S] \,\. \setmin \{ 0 \} 
\ \inc \ 
S \,\. \setmin \{ 0 \} 
\ \inc \ 
S \ \inc \ \widehat{S}$~.
} 

\smallskip

Therefore, the condition $S \,\. \setmin \{ 0 \} \neq \Oset$ yields 

\smallskip

\centerline{
$\Aff{\widehat{S}} 
\ = \ 
\Aff{[ \. (0 , 1) S] \,\. \setmin\{ 0 \} \.} \ \inc \ \Aff{S} 
\ \inc \ 
\Aff{\widehat{S}} \ = \ \Aff{{\Cone{S}} \.}$
} 

\smallskip

according to Point~2 in Proposition~\ref{prop:vect-vs-aff}, and hence we obtain 
${\Aff{S} = \Aff{\widehat{S}} = \Aff{{\Cone{S}} \.}}$. 

\smallskip

But, on the other hand, we have ${\Aff{{\Cone{S}} \.} = \Vect{S}}$ 
by Point~1 in Proposition~\ref{prop:abs-cone}, which implies ${\Aff{S} = \Vect{S}}$, 
and finally we get ${\Aff{S} = \Vect{C}}$ by Point~2 above. 
\end{proof}

\bigskip

A particular important class of positively homogeneous functions is given by gauge functions. 

\bigskip

\begin{definition} 
   Given a subset $S$ of a real vector space $V \!\!$, 
   the \emph{gauge} function $p_{S} : V \! \to \clos{\RR}$ of $S$ is defined by 
   
   %\smallskip
   
   \centerline{
   $p_{S}(x) \. \as \! \inf{\! \{ \l \geq 0 \st x \in \l S \}} \in [0 , {+\infty}]$~.
   } 
\end{definition}

\bigskip

\begin{remark} \label{rem:0-gauge} \ 

\begin{enumerate}[1)]
   \item If $S$ is void, then we have $p_{S} = {+\infty}$ 
   since the empty set does not absorb any vector in $V \!\.$. 
   
   \medskip
   
   \item If $S$ is not empty, then we obviously have $p_{S}(0) = 0$. 
   
   \medskip
   
   \item For any $x \in V \!$ which satisfies $(0 , {+\infty}) x \inc S \.$, 
   we have of course $p_{S}(x) = 0$. 
   
   \medskip
   
   \item Moreover, given any subsets $A$ and $B$ of $V \!$ which satisfy $A \inc B$, 
   we have $p_{B} \leq p_{A}$. 
\end{enumerate} 
\end{remark}

\bigskip

\begin{proposition} \label{prop:gauge-vect} 
   The gauge function $p_{S}$ of a subset $S \.$ of a real vector space $V \!\.$ 
   satisfies the following properties: 
   
   \begin{enumerate}
      \item For any $x \in V \setmin \{ 0 \}$, we have the equivalences 
      
      \smallskip
      
      \centerline{
      \hspace{30pt} $S \. \ \mbox{absorbs} \ x 
      \iff 
      x \in \Bcone{S}
      \iff 
      S \cap (0 , {+\infty}) x \ \neq \ \Oset 
      \iff 
      p_{S}(x) \in [0 , {+\infty})$~.
      } 
      
      \smallskip
      
      \noindent As a consequence, one has $\Cone{S} = \inv{p_{S}}(\RR)$ 
      in case when $S \.$ is not void. 
      
      %\smallskip
      
      \item For any ${x \in V \setmin \{ 0 \}}$, we have 
      ${\disp p_{S}(x) \: = \: \. \inf{\! \{ \l > 0 \st x \in \l S \}} 
      \: = \: \frac{1}{\sup{\! \{ \m > 0 \st \m x \in S \}}}}$ \linebreak
      
      \vspace{-14pt}
      
      \noindent (with the conventions $1 \. / \. 0 \as {+\infty}$ 
      and $1 \. / \! ({+\infty}) \. \as 0$). 
      
      \smallskip
      
      \item The gauge $p_{S}$ is positively homogeneous 
      (with the convention $\l \! \times \! ({+\infty}) \. \as {+\infty}$ 
      for any real number $\l > 0$). 
      
      \smallskip
      
      \item We have $p_{S} = p_{\widehat{S}}$ 
      (the gauge function of the star-shaped hull $\widehat{S}$ of $S \.$). 
      
      \smallskip
      
      \item In case when $S \.$ is star-shaped, we have 
      $\inv{p_{S}}([0 , 1) \. ) \inc S \inc \inv{p_{S}}([0 , 1])$. 
   \end{enumerate} 
\end{proposition}

\bigskip

\begin{proof} \ 

\textsf{Point~1.} 
Given $x \in V \setmin \{ 0 \}$, we have the equivalences 
%%%%%%%%%%
\begin{eqnarray*} % ===> HAL ne gère pas \\ avec l'environnement align !!!
   S \. \ \mbox{absorbs} \ x 
   & \iff & 
   x \in \Bcone{S} \quad \mbox{(see Remark~\ref{rem:abs-cone})} \\ 
   & \iff & 
   S \cap (0 , {+\infty}) x \ \neq \ \Oset \\ 
   & \iff & 
   \{ \l \geq 0 \st x \in \l S \} \ \neq \ \Oset \\ 
   & \iff & 
   p_{S}(x) \in [0 , {+\infty})~. 
\end{eqnarray*} 
%%%%%%%%%%
This proves that we have $\Bcone{S} = \inv{p_{S}}(\RR) \:\! \setmin \{ 0 \}$. 

\smallskip

Moreover, in case when $S$ is not void, 
we have ${0 \in \inv{p_{S}}(\RR)}$ by Point~2 in Remark~\ref{rem:0-gauge}, and hence we get 

\vspace{-20pt}

%%%%%%%%%%
\begin{eqnarray*} 
   \Cone{S} \ = \ \Bcone{S} \cup \{ 0 \} 
   & = & 
   \big[ \inv{p_{S}}(\RR) \cap (V \setmin \{ 0 \} \. ) \. \big] \cup \{ 0 \} \\ 
   & = & 
   \big[ \inv{p_{S}}(\RR) \cup \{ 0 \} \. \big] 
   \cap \big[ \. (V \setmin \{ 0 \} \. ) \cup \{ 0 \} \. \big] \\ 
   & = & 
   \inv{p_{S}}(\RR) \cup \{ 0 \} \ = \ \inv{p_{S}}(\RR)~. 
\end{eqnarray*} 
%%%%%%%%%%

\medskip

\textsf{Point~2.} 
For any $x \in V \setmin \{ 0 \}$, we have $x \not \in 0 \,\. S \.$, 
and hence we get the first equality. 

\smallskip

On the other hand, the second equality is a mere consequence of classical properties 
about the infimum and the supremum in $\RR$ with $\m \as 1 \. / \. \l$. 

\medskip

\textsf{Point~3.} 
Given $x \in V \!$ and a real number $t > 0$, we have 
%%%%%%%%%%
\begin{eqnarray*} 
   p_{S}(t x) \! 
   & = & \! \. 
   \inf{\! \{ \l \geq 0 \st t x \in \l S \}} \\ 
   & = & \! \. 
   \inf{\! \{ \l \geq 0 \st x \in (\l \. / \. t) S \}} \\ 
   & = & \! \. 
   \inf{\! \{ t (\l \. / \. t) \st \l \geq 0 \ \ \mbox{and} \ \ x \in (\l \. / \. t) S \}} 
   \ = \ \. 
   \inf{\! \{ t \m \st \m \geq 0 \ \ \mbox{and} \ \ x \in \m S \}} \ = \ t p_{S}(x)~. 
\end{eqnarray*} 
%%%%%%%%%%

%\medskip
\pagebreak

\textsf{Point~4.} 
First of all, the inclusion $S \inc \widehat{S}$ implies $p_{\widehat{S}} \leq p_{S}$ 
by Point~4 in Remark~\ref{rem:0-gauge}. 

\smallskip

On the other hand, given ${x \in V \!}$ and a real number ${\l \geq 0}$ 
such that we have ${x \in \l \widehat{S} \.}$, 
there exists ${t \in [0 , 1]}$ which satisfies ${x \in \l t S \.}$. 

\smallskip

This yields $p_{S}(x) \leq \l t$, and hence we get $p_{S}(x) \leq \l$ since one has $t \leq 1$. 

\smallskip

The inequality $p_{S}(x) \leq p_{\widehat{S}}(x)$ then follows. 

\medskip

\textsf{Point~5.} 

\textasteriskcentered \ 
For any $x \in S \.$, one can write $1 \in \{ \l \geq 0 \st x \in \l S\}$, 
which yields $p_{S}(x) \leq 1$. 

\smallskip

This proves the inclusion $S \inc \inv{p_{S}}([0 , 1])$. 

\smallskip

\textasteriskcentered \ 
For any ${x \in V \!}$ which satisfies ${p_{S}(x) < 1}$, 
there exists ${\l \in (0 , 1)}$ such that we have ${x \in \l S \.}$, 
and this yields ${x \in S}$ since $S$ is star-shaped. 

\smallskip

This proves the inclusion $\inv{p_{S}}([0 , 1) \. ) \inc S \.$. 
\end{proof}

\medskip %\bigskip

\begin{remark*} 
It is to be mentioned that owing to Point~4 in Proposition~\ref{prop:gauge-vect} 
it is enough to deal with \emph{star-shaped} subsets when considering gauge functions. 
\end{remark*}

\bigskip

It is now time to give the relationship between non-negative positively homogeneous functions 
and gauge functions. For this purpose, let us denote by ${j : \RR \to \bar{\RR}}$ 
the canonical inclusion of $\RR$ into $\bar{\RR}$. 

\medskip %\bigskip

\begin{proposition} \label{prop:f=p} 
   Let $C \.$ be a pointed cone in a real vector space, $S \.$ a star-shaped subset of $C \.$ 
   and $p_{S}$ the gauge function of $S \.$. 
   Then any non-negative function ${\. f \. : C \to \RR}$ which is positively homogeneous 
   satisfies the following equivalence: 
   \[
   \rest{(p_{S}){\.}}{C} \ = \ j \comp \:\!\! f \. 
   \qquad \iff \qquad \. 
   \inv{f}([0 , 1) \. ) \ \inc \ S \ \inc \ \inv{f}([0 , 1]) \ = \ S_{1 \.}(f)~.
   \] 
\end{proposition}

\smallskip %\bigskip

\begin{proof} \ 

$(\imp)$. 
Assume that we have $\rest{(p_{S}){\.}}{C} = j \comp \:\!\! f \!$. 

\smallskip

Since we have $\inv{p_{S}}([0 , 1) \. ) \inc S \inc \inv{p_{S}}([0 , 1])$ 
by Point~5 in Proposition~\ref{prop:gauge-vect}, we get 
\[
\inv{( \. \rest{(p_{S}){\.}}{C}) \.}([0 , 1) \. ) 
\ = \ 
\inv{p_{S}}([0 , 1) \. ) \cap C 
\ \inc \ 
S \cap C 
\ \inc \ 
\inv{p_{S}}([0 , 1]) \cap C 
\ = \ 
\inv{( \. \rest{(p_{S}){\.}}{C}) \.}([0 , 1])~,
\] 
which writes 
$\. \inv{f}([0 , 1) \. ) \inc S \inc \inv{f}([0 , 1])$ by using $S \cap C = S \.$. 

\medskip

$(\con)$. 
Assume that we have $\. \inv{f}([0 , 1) \. ) \inc S \inc \inv{f}([0 , 1])$. 

\smallskip

Then Point~4 in Remark~\ref{rem:0-gauge} 
implies $p_{\. \inv{f}([0 , 1])} \leq p_{S} \leq p_{\. \inv{f}([0 , 1) \. )}$. 

\smallskip

On the other hand, given $x \in C$ and a real number $\l > 0$, we can write the equivalences 
%%%%%%%%%%
\begin{eqnarray*} 
   x \in \l \inv{f}([0 , 1]) 
   & \iff & 
   x \. / \. \l \in \:\!\! \inv{f}([0 , 1]) \\ 
   & \iff & \. 
   f \. (x \. / \. \l) \leq 1 \\ 
   & \iff & \. 
   f \. (x) \leq \l 
   \quad \mbox{(since $\. f \.$ is positively homogeneous)}~. 
\end{eqnarray*} 
%%%%%%%%%%
If we have $x \neq 0$, then Point~2 in Proposition~\ref{prop:gauge-vect} yields 

\smallskip

\centerline{
$p_{\. \inv{f}([0 , 1])}(x) 
\ = \ 
\inf{\! \{ \l > 0 \st x \in \l \inv{f}([0 , 1]) \. \}} \. 
\ = \ 
\inf{\! \{ \l > 0 \st \. f \. (x) \leq \l \}} \ = \ \:\!\! f \. (x)$~.
} 

\smallskip

If we have ${x = 0 \in C \.}$, then $\. \inv{f}([0 , 1])$ contains the origin since $\. f \.$ 
satisfies $\. f \. (0) = 0$ by Remark~\ref{rem:ph-z}. 

\smallskip

Therefore, since $\. \inv{f}([0 , 1])$ is not empty, we get ${p_{\. \inv{f}([0 , 1])}(0) = 0}$ 
by Point~2 in Remark~\ref{rem:0-gauge}, 
and hence we can write $p_{\. \inv{f}([0 , 1])}(x) = f \. (x) = 0$. 

\smallskip

This proves $\rest{(p_{\. \inv{f}([0 , 1])}){\.}}{C} = j \comp \:\!\! f \!$. 

\smallskip

With the same reasoning, we also obtain 
$\rest{(p_{\. \inv{f}([0 , 1) \. )}){\.}}{C} = j \comp \:\!\! f \!$. 

\smallskip

Conclusion: summing up, we have proved 
$j \comp \:\!\! f \leq \rest{(p_{S}){\.}}{C} \leq j \comp \:\!\! f \!$. 
\end{proof}

\bigskip
%\pagebreak

It is to be noticed that Proposition~\ref{prop:f=p} is a generalization 
of the result~\cite[Lemma~5.50, Point~1, page~192]{AliBor06} 
in the situation where $C$ is not reduced to the whole space $V \!\.$. 

\bigskip

Actually, the following result shows that Proposition~\ref{prop:f=p} extends to blunt cones. 

\smallskip %\bigskip

\begin{corollary} \label{cor:f=p} 
   Let $C \.$ be a cone in a real vector space, $S \.$ a star-shaped subset of $C \cup \{ 0 \}$ 
   and $p_{S}$ the gauge function of $S \.$. 
   Then any non-negative function ${\. f \. : C \to \RR}$ which is positively homogeneous 
   satisfies the following equivalence: 
   \[
   \rest{(p_{S}){\.}}{C} \ = \ j \comp \:\!\! f \. 
   \qquad \iff \qquad \. 
   \inv{f}([0 , 1) \. ) \ \inc \ S \ \inc \ \inv{f}([0 , 1]) \cup \{ 0 \} 
   \ = \ 
   S_{1 \.}(f) \cup \{ 0 \}~.
   \] 
\end{corollary}

\smallskip %\bigskip

\begin{proof} 
We may assume that $C$ is not empty since the equivalence to be proved is obvious otherwise. 

\smallskip

Let us consider the pointed cone ${D \as C \cup \{ 0 \}}$ 
and the function ${g : D \to \RR}$ defined by 
${g(0) \. \as 0}$ and ${g(x) \. \as \:\!\! f \. (x)}$ for ${x \in C}$ 
(in case when $C$ is pointed, this makes sense since we then have ${\. f \. (0) = 0}$ 
by Remark~\ref{rem:ph-z}). 

\smallskip

Since $g$ is non-negative and positively homogeneous, 
Proposition~\ref{prop:f=p} yields the equivalence 

\smallskip

\centerline{
$\rest{(p_{S}){\.}}{D} \ = \ j \comp g 
\qquad \iff \qquad 
\inv{g}([0 , 1) \. ) \ \inc \ S \ \inc \ \inv{g}([0 , 1]) \ = \ S_{1 \.}(g)$~. \quad ($\star$)
} 

\medskip

$(\imp)$. 
Assume now that we have $\rest{(p_{S}){\.}}{C} = j \comp \:\!\! f \!$. 

\smallskip

This first implies that $S$ is not empty by Point~1 in Remark~\ref{rem:0-gauge}. 

\smallskip

Therefore, Point~2 in Remark~\ref{rem:0-gauge} 
yields $\rest{(p_{S}){\.}}{D}(0) = p_{S}(0) = 0 = g(0) = (j \comp g) \. (0)$. 

\smallskip

Since one has ${C = D \;\! \setmin \{ 0 \}}$, this proves ${\rest{(p_{S}){\.}}{D} = j \comp g}$, 
and hence we can deduce the inclusions 
${\inv{g}([0 , 1) \. ) \inc S \inc \inv{g}([0 , 1])}$ from the equivalence ($\star$). 

\smallskip

Finally, we get 

\smallskip

\centerline{
$\inv{f}([0 , 1) \. ) \ = \ \inv{g}([0 , 1) \. ) \cap \{ 0 \} \ \inc \ S \cap \{ 0 \} \ \inc \ S$ 
\ \ and \ \ % \ \ instead of \qquad since the line is too long
$S \ \inc \ \inv{g}([0 , 1]) \ = \ \:\!\! \inv{f}([0 , 1]) \cup \{ 0 \}$~.
} 

\medskip

$(\con)$. 
Conversely, assume that we have $\. \inv{f}([0 , 1) \. ) \inc S \inc \inv{f}([0 , 1]) \cup \{ 0 \}$. 

\smallskip

Since $C$ is not empty, the same is true for $\. \inv{f}([0 , 1) \. )$, 
which yields $S \neq \Oset \.$. 

\smallskip

This implies that $S$ contains the origin since it is star-shaped, and hence one obtains 

\smallskip

\centerline{
$\inv{g}([0 , 1) \. ) 
\ = \ \:\!\! 
\inv{f}([0 , 1) \. ) \cup \{ 0 \} \ \inc \ S \cup \{ 0 \} 
\ = \ 
S \ \inc \ \inv{f}([0 , 1]) \cup \{ 0 \} \ = \ \inv{g}([0 , 1])$~.
} 

\smallskip

Finally, this yields $\rest{(p_{S}){\.}}{D} = j \comp g$ owing to the equivalence ($\star$), 
which gives 

\smallskip

\centerline{
$\rest{(p_{S}){\.}}{C} 
\ = \ 
\rest{(\rest{(p_{S}){\.}}{D}){\.}}{C} 
\ = \ 
\rest{(j \comp g){\.}}{C} 
\ = \ 
j \comp (\rest{g}{C}) \ = \ j \comp \:\!\! f \!$~.
} 
\end{proof}

%\bigskip

\begin{remark*} 
For any non-negative function ${g : C \to \RR}$ which is positively homogeneous of degree ${\a > 0}$, 
Corollary~\ref{cor:f=p} obviously applies to $\. f \:\!\! \as g^{1 \. / \. \a} \!$. 
\end{remark*}

%%%%%%%%%%%%%%%%%%%%

\subsection{Continuity of gauge functions} \label{subsec:Continuity of gauge functions}

The aim of this subsection is to give a characterization of the continuity of the gauge function 
of an \emph{arbitrary} star-shaped subset of a \emph{general} topological real vector space. 

\medskip

As we will need in the sequel some elementary but useful properties 
about topological real vector spaces, let us begin with the following remark. 

\smallskip %\bigskip

\begin{remark} \label{rem:0-tvs} \ 

\begin{enumerate}[1)]
   \item Given a neighborhood $U \!$ of the origin in a topological real vector space $V \!\!$, 
   the following easy-to-prove properties hold: 
   
   \smallskip
   
   \begin{enumerate}[a)]
      \item For any vector ${x \in V \!\!}$, there exists a real number ${\e > 0}$ 
      such that we have ${[{-\e} , \e] x \inc U \!}$. 
      
      \smallskip
      
      \item The subset $U \!$ of $V \!$ is absorbing by Point~1.a, 
      and therefore satisfies ${\Vect{U} = V \!}$ by Point~4 in Proposition~\ref{prop:abs-cone}. 
   \end{enumerate} 
   
   %\medskip
   \pagebreak
   
   \item A cone $C$ in a topological real vector space $V \!$ 
   whose interior $\intr{C}$ contains the origin is equal to $V \!\.$. 
   Indeed, $C$ is then a neighborhood of $0$ in $V \!\!$, 
   and hence is absorbing by Point~1.b, 
   which yields ${C = V \!}$ by Point~4 in Proposition~\ref{prop:abs-cone}. 
   
   \medskip
   
   \item Given a finite-dimensional real vector space $W \!\!$, 
   there exists a \emph{unique} topological real vector space structure on $W \!$ which is Hausdorff. 
   Endowed with this structure, $W \!$ is then isomorphic 
   to the canonical topological real vector space $\Rn{n}$, where $n$ denotes the dimension of $W \!$ 
   (see~\cite[Chapitre~I, Th\'{e}or\`{e}me~2, page~14]{BouEVT81}). 
   
   \medskip
   
   \item That said, it is to be mentioned that any finite-dimensional topological real vector space 
   is isomorphic to the Cartesian product ${\Rn{k} \cart \Rn{n - k}}$ 
   for some integers ${0 \leq k \leq n}$, 
   where the first factor is equipped with the usual topology 
   and the second one with the trivial topology 
   (see for example~\cite[Chapter~2, Section~7, Problem~A, page~64]{KelNam76}). 
\end{enumerate} 
\end{remark}

\bigskip

\begin{lemma}~\label{lem:gauge-top} 
   The gauge function $p_{S}$ of a star-shaped subset $S \.$ 
   of a topological real vector space $V \!\.$ satisfies the following properties: 
   
   \smallskip % some space is needed here
   
   \begin{enumerate}
      \item $\inv{p_{S}}([0 , 1]) \inc \clos{\inv{p_{S}}([0 , 1) \. )}$. 
      
      \smallskip
      
      \item $\intr{\wideparen{\inv{p_{S}}([0 , 1])}} \inc \inv{p_{S}}([0 , 1) \. )$. 
   \end{enumerate} 
\end{lemma}

\bigskip

\begin{proof} \ 

\textsf{Point~1.} 
Let $x \in V \!$ such that $p_{S}(x) \leq 1$ holds, 
and let $U \!$ be a neighborhood of $x$ in $V \!\.$. 

\smallskip

Since the map ${t \longmapsto t x}$ from $\RR$ to $V \!$ is continuous at ${t \as \;\!\! 1}$, 
there exists ${\a \in (0 , 1)}$ 
satisfying ${[\a , 1] x \inc U \!}$, which in particular yields ${\a x \in U \!}$. 

\smallskip

But we have $p_{S}(\a x) = \a p_{S}(x) < 1$, which implies $\a x \in \inv{p_{S}}([0 , 1) \. )$. 

\smallskip

This proves the inclusion $\inv{p_{S}}([0 , 1]) \inc \clos{\inv{p_{S}}([0 , 1) \. )}$. 

\medskip

\textsf{Point~2.} 
Given ${x \in \intr{\wideparen{\inv{p_{S}}([0 , 1])}}}$, the continuity at ${t \as \;\!\! 1}$ 
of the map ${t \longmapsto t x}$ from $\RR$ to $V \!$ insures the existence of a real number ${\a > 1}$ 
such that we have ${[1 , \a] x \inc \inv{p_{S}}([0 , 1])}$. 

\smallskip

In particular, we get ${\a x \in \inv{p_{S}}([0 , 1])}$, which implies ${p_{S}(x) \leq 1 \. / \. \a}$ 
by the positive homogeneity of $p_{S}$ (see Point~3 in Proposition~\ref{prop:gauge-vect}), 
and hence ${p_{S}(x) < 1}$. 

\smallskip

This proves the inclusion $\intr{\wideparen{\inv{p_{S}}([0 , 1])}} \inc \inv{p_{S}}([0 , 1) \. )$. 
\end{proof}

\bigskip

We shall now use Lemma~\ref{lem:gauge-top} to establish a result which gives 
a complete and simple characterization of the continuity of a gauge function (see Point~2.c below). 

\bigskip

\begin{proposition} \label{prop:gauge-top} 
   The gauge function $p_{S}$ of a star-shaped subset $S \.$ 
   of a topological real vector space $V \!\.$ satisfies the following properties: 
   
   \smallskip % some space is needed here
   
   \begin{enumerate}
      \item We have the equalities 
      
      \begin{enumerate}[(a)]
         \item $\intr{\wideparen{\inv{p_{S}}([0 , 1) \. )}} = \intr{S} 
         = \intr{\wideparen{\inv{p_{S}}([0 , 1])}}$, and 
         
         \medskip % \smallskip
         
         \item $\clos{\inv{p_{S}}([0 , 1) \. )} = \clos{S} = \clos{\inv{p_{S}}([0 , 1])}$. 
      \end{enumerate} 
      
      \bigskip % \medskip
      %\pagebreak
      
      \item We have the equivalences 
      
      \smallskip
      
      \begin{enumerate}[(a)]
         \item $p_{S}$ is lower semi-continuous 
         $\iff$ $\inv{p_{S}}([0 , 1])$ is closed in $V \!\.$ 
         $\iff$ $\inv{p_{S}}([0 , 1]) = \clos{S} \.$, 
         
         \smallskip
         
         \item $p_{S}$ is upper semi-continuous 
         $\iff$ $\inv{p_{S}}([0 , 1) \. )$ is open in $V \!\.$ 
         $\iff$ $\inv{p_{S}}([0 , 1) \. ) = \intr{S} \.$, and 
         
         \smallskip
         
         \item $p_{S}$ is continuous $\iff$ $\inv{p_{S}}(1) = \bd{S} \.$. 
      \end{enumerate} 
   \end{enumerate} 
\end{proposition}

\bigskip

\begin{proof} \ 

\textsf{Point~1.a.} 
As we have ${\inv{p_{S}}([0 , 1) \. ) \inc S \inc \inv{p_{S}}([0 , 1])}$ 
by Point~5 in Proposition~\ref{prop:gauge-vect}, one first obtains 
${\intr{\wideparen{\inv{p_{S}}([0 , 1) \. )}} \inc \intr{S} 
\inc \intr{\wideparen{\inv{p_{S}}([0 , 1])}}}$. 

\smallskip

Then, combining these two inclusions with Point~2 in Lemma~\ref{lem:gauge-top} 
and taking the interior, we get 
${\intr{\wideparen{\inv{p_{S}}([0 , 1) \. )}} = \intr{S} = \intr{\wideparen{\inv{p_{S}}([0 , 1])}}}$. 

\medskip

\textsf{Point~1.b.} 
As we have ${\inv{p_{S}}([0 , 1) \. ) \inc S \inc \inv{p_{S}}([0 , 1])}$ 
by Point~5 in Proposition~\ref{prop:gauge-vect}, 
one first obtains ${\clos{\inv{p_{S}}([0 , 1) \. )} \inc \clos{S} \inc \clos{\inv{p_{S}}([0 , 1])}}$. 

\smallskip

Then, combining these two inclusions with Point~1 in Lemma~\ref{lem:gauge-top} 
and taking the closure, we get 
${\clos{\inv{p_{S}}([0 , 1) \. )} = \clos{S} = \clos{\inv{p_{S}}([0 , 1])}}$. 

\medskip

\textsf{Point~2.a.} 

\textasteriskcentered \ 
The first implication $\imp$ is straightforward since $p_{S}$ is non-negative. 

\smallskip

\textasteriskcentered \ 
The second implication $\imp$ is a consequence of the second equality in Point~1.b. 

\smallskip

\textasteriskcentered \ 
Now, assume that the equality $\inv{p_{S}}([0 , 1]) = \clos{S}$ holds, and pick any $\a \in \RR$. 
Then we get 
\[
\inv{p_{S}}([{-\infty} , \a]) \, = \, % instead of \ = \ since the line is too long
\begin{cases} 
   \Oset & \! \mbox{if one has} \ \a < 0 \ \mbox{(since $p_{S}$ is non-negative)} \\ 
   & \\ 
   \inv{p_{S}}([0 , \a]) = \a \inv{p_{S}}([0 , 1]) 
   = \a \clos{S} & \! \mbox{if one has} \ \a > 0 \ \mbox{(since $p_{S}$ is non-negative} \\ 
   & \! \mbox{and owing to Point~3 in Proposition~\ref{prop:gauge-vect})} \\ 
   & \\ 
   \disp \bigcap_{t > 0} \inv{p_{S}}([{-\infty} , t]) 
   = \bigcap_{t > 0} t \clos{S} & \! \mbox{if one has} \ \a = 0 \ 
   \mbox{(by the previous line)} 
\end{cases},
\] 
which shows that $\inv{p_{S}}([{-\infty} , \a])$ is closed in $V \!\.$. 

\smallskip

This proves that $p_{S}$ is lower semi-continuous. 

\medskip

\textsf{Point~2.b.} 

\textasteriskcentered \ 
The first implication $\imp$ is straightforward since $p_{S}$ is non-negative. 

\smallskip

\textasteriskcentered \ 
The second implication $\imp$ is a consequence of the first equality in Point~1.a. 

\smallskip

\textasteriskcentered \ 
Now, assume that the equality $\inv{p_{S}}([0 , 1) \. ) = \intr{S}$ holds, and pick any $\a \in \RR$. 
Then we get 
\[
\inv{p_{S}}([{-\infty} , \a) \. ) \, = \, % instead of \ = \ since the line is too long
\begin{cases} 
   \Oset & \! \mbox{if one has} \ \a \leq 0 \ \mbox{(since $p_{S}$ is non-negative)} \\ 
   & \\ 
   \inv{p_{S}}([0 , \a) \. ) = \a \inv{p_{S}}([0 , 1) \. ) 
   = \a \intr{S} & \! \mbox{if one has} \ \a > 0 \ \mbox{(since $p_{S}$ is non-negative} \\ 
   & \! \mbox{and owing to Point~3 in Proposition~\ref{prop:gauge-vect})} 
\end{cases},
\] 
which shows that $\inv{p_{S}}([{-\infty} , \a) \. )$ is open in $V \!\.$. 

\smallskip

This proves that $p_{S}$ is upper semi-continuous. 

\medskip

\textsf{Point~2.c.} 

\textasteriskcentered \ $(\imp)$. 
Assume that $p_{S}$ is continuous. 

\smallskip

Then we have ${\inv{p_{S}}([0 , 1]) = \clos{S}}$ by Point~2.a 
and ${\inv{p_{S}}([0 , 1) \. ) = \intr{S}}$ by Point~2.b, which yields the equalities 
${\inv{p_{S}}(1) = \inv{p_{S}}([0 , 1]) \setmin \inv{p_{S}}([0 , 1) \. ) 
= \clos{S} \setmin \intr{S} = \bd{S} \.}$. 

\smallskip

\textasteriskcentered \ $(\con)$. 
Conversely, assume that we have $\inv{p_{S}}(1) = \bd{S} \.$. 

\smallskip

Then we get 

\vspace{-6pt} %\smallskip

\centerline{
$\inv{p_{S}}([0 , 1) \. ) 
\ = \ 
\inv{p_{S}}([0 , 1]) \setmin \inv{p_{S}}(1) 
\ \inc \ 
\clos{\inv{p_{S}}([0 , 1])} \setmin \inv{p_{S}}(1) 
\ = \ 
\clos{S} \setmin \bd{S} \ = \ \intr{S}$
} 

\smallskip

since we proved $\clos{\inv{p_{S}}([0 , 1])} = \clos{S}$ in Point~1.b. 

%\smallskip
\pagebreak

On the other hand, we have 
$\intr{S} = \intr{\wideparen{\inv{p_{S}}([0 , 1) \. )}} \inc \inv{p_{S}}([0 , 1) \. )$ by Point~1.a. 

\smallskip

So we get ${\inv{p_{S}}([0 , 1) \. ) = \intr{S} \.}$, 
which proves that $p_{S}$ is upper semi-continuous according to Point~2.b. 

\smallskip

Finally, using again the last equality, we obtain 
${\inv{p_{S}}([0 , 1]) = \inv{p_{S}}([0 , 1) \. ) \cup \inv{p_{S}}(1) 
= \intr{S} \cup \bd{S} = \clos{S} \.}$, 
which proves that $p_{S}$ is lower semi-continuous according to Point~2.a. 

\smallskip

Conclusion: the function $p_{S}$ is continuous. 
\end{proof}

\bigskip

\begin{remark} \label{rem:gauge-top} \ 

\begin{enumerate}[1)]
   \item Each of the inclusions 
   ${\intr{S} \inc \inv{p_{S}}([0 , 1) \. ) \inc S \inc \inv{p_{S}}([0 , 1]) \inc \clos{S}}$ 
   contained in the proof of Point~1 in Proposition~\ref{prop:gauge-top} 
   are strict in general as we can check with the star-shaped subset $S$ of $\Rn{2}$ defined by 
   ${S \as D \cup \{ (r \cos{\. \t} \, , \, r \sin{\. \t}) 
   \st r \in [0 , 2] \ \mbox{and} \ \t \in \QQ \}}$, where $D$ denotes the open unit disk in $\Rn{2}$. 
   
   \medskip
   
   \item As a consequence of Point~2.a in Proposition~\ref{prop:gauge-top}, 
   the lower semi-continuity of $p_{S}$ yields the relation 
   ${\disp \inv{p_{S}}(0) = \bigcap_{t > 0} t \clos{S} \.}$. 
   
   \smallskip %\medskip
   
   \item Since we have ${S \inc \inv{p_{S}}([0 , 1])}$ by Point~5 in Proposition~\ref{prop:gauge-vect} 
   and ${\inv{p_{S}}([0 , 1]) \inc \clos{S}}$ by the second equality in Point~1.b 
   in Proposition~\ref{prop:gauge-top}, the relation ${\inv{p_{S}}([0 , 1]) = \clos{S}}$ occurs 
   whenever $S$ is closed in $V \!\!$, and this then implies the lower semi-continuity of $p_{S}$ 
   owing to Point~2.a in Proposition~\ref{prop:gauge-top}. 
   
   \medskip
   
   \item On the other hand, as regards the upper semi-continuity, 
   let us notice that if $S$ is not empty 
   and if $p_{S}$ is upper semi-continuous, then $S$ must be a neighborhood of the origin in $V \!$ 
   by Point~2.b in Proposition~\ref{prop:gauge-top}. 
   
   \noindent Nevertheless, the converse is not true as we can check 
   by considering the star-shaped subset $S$ 
   of $\Rn{2}$ defined by ${S \as D \cup ([0 , {+\infty}) \cart \{ 0 \})}$, 
   where $D$ denotes the open unit disk in $\Rn{2}$. Indeed, in that case, the point ${x \as (2 , 0)}$ 
   satisfies ${p_{S}(x) = 0 \in [0 , 1)}$ by Point~3 in Remark~\ref{rem:0-gauge} 
   but does not belong to ${\intr{S} = D \.}$, 
   which shows that the condition ${\inv{p_{S}}([0 , 1) \. ) = \intr{S}}$ 
   in Point~2.b in Proposition~\ref{prop:gauge-top} does not hold. 
   
   \noindent We shall see in Proposition~\ref{prop:cvx-gauge-cont} which extra condition is needed 
   to fill this gap. 
\end{enumerate} 
\end{remark}

\bigskip

%%%%%%%%%%%%%%%%%%%%

\section{About convexity} \label{sec:About-convexity}

In this section, we give some definitions and properties about convex sets and convex functions 
before moving towards the notion of strict convexity for subsets of an arbitrary topological 
real vector space. 

%%%%%%%%%%%%%%%%%%%%

\subsection{Geometric aspects of convexity}

We shall first focus on some useful properties of convexity for both sets and functions. 

\bigskip

\begin{definition} 
   Given points $x$ and $y$ in a real vector space $V \!\!$, the set 
   
   \smallskip
   
   \centerline{
   $[x , y] \as \{ \. (1 - t) x + t y \st t \in [0 , 1] \}$
   } 
   
   \smallskip
   
   is called the \emph{(closed) line segment} between $x$ and $y$, whereas the set 
   
   \smallskip
   
   \centerline{
   $]x , y[ \as [x , y] \,\. \setmin \{ x , y \}$
   } 
   
   \smallskip
   
   is called the \emph{open line segment} between $x$ and $y$ 
   (the latter set is therefore empty in case when one has $x = y$). 
   
   \medskip
   
   \noindent A subset $C$ of $V \!$ is said to be \emph{convex} 
   if we have $[x , y] \inc C$ for all $x , y \in C \.$. 
\end{definition}

%\bigskip
%\pagebreak

\begin{remark} \label{rem:cvx-op} 
For any collection of convex subsets of $V \!$ 
which is an upward directed set for the inclusion, its union is itself a convex subset of $V \!\.$. 
\end{remark}

\bigskip

\begin{definition} 
   The \emph{convex hull} $\Conv{S}$ of a subset $S$ of a real vector space $V \!$ 
   is the smallest convex subset of $V \!$ which contains $S \.$. 
   In other words, the convex hull of $S$ is equal to the set of points ${x \in V \!}$ 
   which write ${\disp x = \! \sum_{i = 1}^{n} \l_{i} x_{i}}$ 
   for some integer ${n \geq 1}$, some points ${\llist{x}{1}{n} \in S}$ and \linebreak
   
   \vspace{-6pt}
   
   some real numbers ${\llist{\l}{1}{n} \in [0 , {+\infty})}$ 
   which satisfy ${\disp \,\! \sum_{i = 1}^{n} \l_{i} = 1}$. 
\end{definition}

\bigskip

We obviously have ${\Conv{S} \inc \Aff{S}}$. 

\bigskip

\begin{remark} \label{rem:cvx+star-shaped+cone} \ 

\begin{enumerate}[1)]
   \item Any convex subset $C$ of $V \!$ is star-shaped if and only if it contains the origin 
   since the convexity of $C$ is equivalent to saying that $C$ is star-shaped about any of its points. 
   
   \noindent For example, any convex subset $C$ of $V \!$ for which we can find $x \in V \!$ 
   such that $C$ absorbs $x$ and ${-x}$ contains the origin. 
   
   \noindent Indeed, there exist real numbers ${\l > 0}$ and ${\m > 0}$ satisfying 
   ${\l x \in C}$ and ${{-\m} x \in C \.}$, and hence we get 
   
   %\smallskip
   \vspace{-10pt}
   
   \centerline{
   $\disp 0 \ = \ \frac{\m}{\l + \m}(\l x) + \frac{\l}{\l + \m}({-\m} x) \in C$
   } 
   
   \smallskip
   
   \noindent since $C$ is convex and since we have 
   ${\disp \frac{\m}{\l + \m} \geq 0}$, ${\disp \frac{\l}{\l + \m} \geq 0}$ 
   and ${\disp \frac{\m}{\l + \m} + \frac{\l}{\l + \m} = 1}$. 
   
   \medskip
   
   \item The convex hull of a cone $C$ in $V \!$ is still a cone in $V \!\.$. 
   
   \noindent Indeed, for any real number ${\l > 0}$, 
   the multiplication by $\l$ is an affine mapping from $V \!$ to $V \!\!$, 
   and hence satisfies $\l \Conv{C} \inc \Conv{\l C \.}$, 
   which yields $\l \Conv{C} \inc \Conv{C}$ since we have $\l C \inc C \.$. 
\end{enumerate} 
\end{remark}

\bigskip

\begin{proposition} \label{prop:sshc} 
   For any convex subset $C \.$ of a real vector space, 
   its star-shaped hull $\widehat{C} \.$ is also convex. 
\end{proposition}

\bigskip

\begin{proof} 
We refer to~\cite[Theorem~7.2.3, page~336]{Web94} for a proof of a general result 
which implies Proposition~\ref{prop:sshc} when we take $A \as C$ and $B \as \{ 0 \}$. 
\end{proof}

\bigskip

\begin{remark} \label{rem:cvx-cone} 
The pointed conic hull $\Cone{C}$ of a convex subset $C$ of a real vector space $V \!$ is convex. 

Indeed, assuming that $C$ is not empty (since this property is obvious otherwise), 
we know that for each real number ${\l > 0}$ the subset ${\l \widehat{C}}$ of $V \!$ is convex 
by Proposition~\ref{prop:sshc} and since the multiplication by $\l$ is an affine mapping 
from $V \!$ to itself. 

Moreover, we have $\l \widehat{C} = [0 , \l] C \.$. 

Therefore, the family $(\l \widehat{C})_{\l > 0}$ is non-decreasing for the inclusion, 
and hence its union, which is equal to $\Cone{C}$, 
is a convex subset of $V \!$ by Remark~\ref{rem:cvx-op}. 
\end{remark}

\bigskip

\begin{proposition} \label{prop:cone-+} 
   Given a cone $C \.$ in a real vector space $V \!\!$, we have the following equivalence: 
   \[
   C \ \mbox{is convex} 
   \qquad \iff \qquad 
   C \ \mbox{is stable with respect to} \ +~.
   \] 
\end{proposition}

\bigskip

\begin{proof} 
We refer to~\cite[Theorem~2.6, page~14]{Roc70} for a proof of this result. 
\end{proof}

\bigskip

Let us now switch to functions by recalling the definition of a (strictly) convex function 
on a convex subset of a real vector space. 

\bigskip

\begin{definition} \label{def:f-cvx+strict-cvx} 
   Given a convex subset $C$ of a real vector space, a function ${\. f \. : C \to \RR}$ is said to be 
   
   \begin{enumerate}
      \item \emph{convex} if we have 
      ${\. f \. ( \. (1 - t) x + t y) \leq (1 - t) f \. (x) + t f \. (y)}$ 
      for any points ${x , y \in C}$ and any number ${t \in (0 , 1)}$, 
      
      \smallskip
      
      \item \emph{strictly convex} if we have 
      ${\. f \. ( \. (1 - t) x + t y) < (1 - t) f \. (x) + t f \. (y)}$ 
      for any distinct points ${x , y \in C}$ and any number ${t \in (0 , 1)}$. 
   \end{enumerate} 
\end{definition}

\bigskip

It is to be noticed that both convexity and strict convexity of functions are mere affine notions. 

\bigskip

\begin{remark} \label{rem:sc} \ 

\begin{enumerate}[1)]
   \item A strictly convex function is of course convex. 
   
   \medskip
   
   \item Given a scalar product $\scal{\cdot \,}{\cdot}$ on a real vector space $V \!$ 
   and a real number ${\a > 1}$, the norm $\norm{\cdot}$ associated with $\scal{\cdot \,}{\cdot}$ 
   is such that the function ${\. f \. \as \norm{\cdot}^{\. \a} \!}$ is strictly convex. 
   
   \noindent Indeed, pick $x$ and $y$ in $V \!$ which satisfy ${x \neq y}$, 
   and let us fix ${t \in (0 , 1)}$. 
   
   \noindent First of all, if we have ${\norm{x} \neq \norm{y}}$, then one can write 
   \[
   f \. ( \. (1 - t) x + t y) 
   \ \leq \ 
   ( \. (1 - t) \norm{x} + t \norm{y})^{\. \a} \. 
   \ < \ 
   (1 - t) f \. (x) + t f \. (y) 
   \] 
   since the function ${\f : [0 , {+\infty}) \to \RR}$ defined by ${\f(s) \. \as s^{\a} \!}$ 
   is non-decreasing and strictly convex. 
   
   \noindent Now, assume that the equality ${\norm{x} = \norm{y}}$ holds. 
   
   \noindent If we had ${\scal{x}{y} = \norm{x} \norm{y}}$, 
   the vectors $x$ and $y$ would be collinear by the equality case in the Cauchy-Schwarz inequality, 
   which means that there would exist ${\l \in \RR}$ satisfying ${y = \l x}$. 
   
   \noindent Therefore, we would obtain ${\l \norm{x}^{\. 2} \! = \scal{x}{y} = \norm{x}^{\. 2} \!}$, 
   which yields $\l = 1$ since $x$ is not equal to the zero vector 
   (otherwise we would have ${x = y = 0}$ which is not possible), and hence ${y = x}$, 
   which is not possible. 
   
   \noindent So we have ${\scal{x}{y} < \norm{x} \norm{y}}$, which implies 
   %%%%%%%%%%
   \begin{eqnarray*} % ===> HAL ne gère pas \\ avec l'environnement multline !!!
      \norm{\. (1 - t) x + t y}^{\. 2} \! \! 
      & = & \! 
      (1 - t)^{\. 2} \norm{x}^{\. 2} \! + 2 t (1 - t) \scal{x}{y} + t^{2} \norm{y}^{\. 2} \\ 
      & = & \! 
      2 t (1 - t) \. \big[ \. \scal{x}{y} - \norm{x} \norm{y} \big] + \norm{x}^{\. 2} \. 
      \ < \ \norm{x}^{\. 2} \!~, 
   \end{eqnarray*} 
   %%%%%%%%%%
   
   \vspace{-6pt}
   
   \noindent and hence we get 
   \[
   f \. ( \. (1 - t) x + t y) 
   \ = \ 
   (\norm{\. (1 - t) x + t y}^{\. 2})^{\a \. / \. 2} \. 
   \ < \ (\norm{x}^{\. 2})^{\a \. / \. 2} \! 
   \ = \ 
   \norm{x}^{\. \a} \! 
   \ = \ 
   (1 - t) f \. (x) + t f \. (y) 
   \] 
   since the function ${\f : [0 , {+\infty}) \to \RR}$ defined by 
   ${\f(s) \. \as s^{\a \. / \.2} \!}$ is increasing. 
\end{enumerate} 
\end{remark}

\bigskip

The next property gives a way for constructing new (strictly) convex functions from old ones. 

\bigskip

\begin{proposition} \label{prop:ac} 
   Let $C \.$ be a subset of a real vector space, ${\. f \. : C \to \RR}$ a non-negative function 
   and ${\a \geq 1}$ a real number. Then we have the implication 
   \[
   f \. \ \mbox{is (strictly) convex} 
   \qquad \imp \qquad \. 
   f^{\a} \! \ \mbox{is (strictly) convex}~.
   \] 
\end{proposition}

\bigskip

\begin{proof} 
This is straightforward since the function ${\f : [0 , {+\infty}) \to \RR}$ defined by 
${\f(t) \. \as t^{\a} \!}$ is both convex and increasing. 
\end{proof}

\bigskip

\begin{remark*} 
Of course, the converse of the implication in Proposition~\ref{prop:ac} is not true 
as we can easily see with ${\a \as 2}$ and ${\. f \. : [0 , {+\infty}) \to \RR}$ defined by 
${\disp \. f \. (t) \. \as \sqrt{t}}$ (for convexity) 
or by ${\. f \. (t) \. \as {t}}$ (for strict convexity). 
\end{remark*}

%%%%%%%%%%%%%%%%%%%%

\subsection{Topological aspects of convexity}

We shall now deal with the topological notion of strict convexity 
for subsets of a general topological real vector space. 

\bigskip

For this purpose, we will first need to introduce the relative interior, closure and boundary 
of an arbitrary subset of a topological real vector space.

\bigskip

\begin{definition} \label{def:ri+rc+rb} 
   Let $S$ be a subset of a topological real vector space $V \!\.$. 
   
   \begin{enumerate}
      \item The \emph{relative interior} $\ri{S}$ of $S$ is the interior of $S$ 
      with respect to the relative topology of $\Aff{S}$ in $V \!\.$. 
      
      \smallskip
      
      \item The \emph{relative closure} $\rc{S}$ of $S$ is the closure of $S$ 
      with respect to the relative topology of $\Aff{S}$ in $V \!\.$. 
      
      \smallskip
      
      \item The \emph{relative boundary} $\rb{S}$ of $S$ is the boundary of $S$ 
      with respect to the relative topology of $\Aff{S}$ in $V \!\.$ 
      (so we have $\rb{S} = \rc{S} \setmin \ri{S}$). 
   \end{enumerate} 
\end{definition}

\bigskip

We obviously have $\intr{S} \inc \ri{S}$ and $\rc{S} \inc \clos{S} \.$, 
which yields $\rb{S} \inc \bd{S} \.$. 

\bigskip

\begin{proposition} \label{prop:ri-rc-rb} 
   Let $V \!\.$ be a topological real vector space. 
   
   \begin{enumerate}
      \item For any subsets $A$ and $B \.$ of $V \!\!$, we have the implication 
      $A \inc B \imp \rc{A} \inc \rc{B}$. 
      
      \smallskip
      
      \item For any subset $S \.$ of $V \!\!$, we have 
      
      \smallskip
      
      \begin{enumerate}[(a)]
         \item $\Aff{\rc{S} \.} = \Aff{S}$, and 
         
         \smallskip
         
         \item $\ri{S} \neq \Oset$ $\iff$ ($S \neq \Oset$ and $\Aff{\ri{S} \.} = \Aff{S} \.$). 
      \end{enumerate} 
   \end{enumerate} 
\end{proposition}

\bigskip

We refer to~\cite[Proposition~2.7, page~805]{SimVer18} for a proof of this result. 

\bigskip

Before we give the definition of a strictly convex set---which is a central notion 
in the present work---let us first recall some useful topological results about convex subsets 
of a general topological real vector space. 

\bigskip

\begin{proposition}[see~{\cite[Chapitre~II, pages~14 and~15]{BouEVT81}}] \label{prop:conv-top} 
   For any convex subset $C \.$ of a topological real vector space, 
   we have the following properties: 
   
   \begin{enumerate}
      \item The closure $\clos{C} \.$ of $C \.$ is also convex. 
      
      %\smallskip
      
      \item For any $x \in \intr{C} \.$ and $y \in \clos{C} \.$, we have ${]x , y[} \inc \intr{C} \.$. 
   \end{enumerate} 
\end{proposition}

\bigskip

\begin{remark} \label{rem:cvx-open-segment} 
In particular, Point~2 in Proposition~\ref{prop:conv-top} 
implies that for any ${x , y \in \clos{C}}$ we have the implication 
\[
{]x , y[} \cap \intr{C} \ \neq \ \Oset 
\qquad \imp \qquad 
{]x , y[} \inc \intr{C} \.~.
\] 
\end{remark}

\bigskip

\begin{definition} 
   A subset $C$ of a topological real vector space $V \!$ 
   is said to be \emph{strictly convex} if
   for any two distinct points $x , y \in \rc{C}$ we have ${]x , y[} \inc \ri{C \.}$. 
\end{definition}

\bigskip

\begin{remark}\label{rem:strict-cvx} \ 

\begin{enumerate}[1)]
   \item This definition coincides with the usual one 
   when $V \!$ is the canonical topological real vector space $\Rn{n}$ 
   (see for example~\cite[page~2]{Gru07} and~\cite[page~87]{Sch14}) 
   since in this case the closeness of $\Aff{C}$ in $V \!$ yields ${\rc{C} = \clos{C \.}}$. 
   
   \medskip
   
   \item A strictly convex subset of $V \!$ is of course convex. 
   
   \medskip
   
   \item It is to be noticed that the property of being strictly convex for $C$ 
   involves the topology of $V \!$ whereas the property of just being convex does not. 
   
   \medskip
   
   \item A subset $C$ of $V \!$ is strictly convex if and only if it is convex 
   and for any two distinct points ${x , y \in \rb{C}}$ we have ${{]x , y[} \inc \ri{C \.}}$. 
   Indeed, the necessary condition is an easy consequence of the very definition 
   of the strict convexity of $C$ combined with Point~2 above. 
   And the sufficient condition is obtained by Point~2 in Proposition~\ref{prop:conv-top}. 
   
   \noindent According to the common geometric intuition, this means that $C$ is convex 
   and that there is no non-trivial segment in the relative boundary of $C \.$. 
\end{enumerate} 
\end{remark}

\bigskip

\begin{proposition} \label{prop:ri-strict-cvx} 
   For any strictly convex subset $C \.$ of a topological real vector space, we have the implication 
   \[
   C \ \neq \ \Oset 
   \qquad \imp \qquad 
   \ri{C} \ \neq \ \Oset \.~.
   \] 
\end{proposition}

\bigskip

We refer to~\cite[Proposition~2.9, page~807]{SimVer18} for a proof of this result. 

\bigskip

\begin{remark} 
When dealing with a single strictly convex subset $C$ 
of a general topological real vector space $V \!\!$, we will always assume in the hypotheses 
that $C$ has a non-empty interior in $V \!$ in order to insure $\Aff{C} = \. V \!\!$, 
and this makes sense by Proposition~\ref{prop:ri-strict-cvx} 
and Point~2.b in Proposition~\ref{prop:ri-rc-rb}. 
\end{remark}

\bigskip

\begin{proposition} \label{prop:C1-C2-strictly-cvx} 
   Let $C_{1 \. \!}$ and $C_{2 \!}$ be subsets of a topological real vector space 
   
   which satisfy ${\ri{C_{1}} \inc C_{2} \inc \rc{C_{1}}}$. Then we have the implication
   \[
   C_{1 \. \!} \ \mbox{is strictly convex} 
   \qquad \imp \qquad 
   C_{2 \!} \ \mbox{is strictly convex}~.
   \] 
\end{proposition}

\bigskip

\begin{proof} 
Assume that $C_{1 \. \!}$ is strictly convex. 

\medskip

\noindent \textasteriskcentered \ 
In case when $C_{1 \. \!}$ is empty, the sets $\ri{C_{1}}$ and $\rc{C_{1}}$ are empty too, 
which implies that $C_{2 \!}$ is empty, and hence strictly convex.  

\medskip

\noindent \textasteriskcentered \ 
In case when $C_{1 \. \!}$ is not empty, we have ${\ri{C_{1}} \neq \Oset}$ 
according to Proposition~\ref{prop:ri-strict-cvx}. 

\smallskip

Therefore, Points~2.b and~2.a in Proposition~\ref{prop:ri-rc-rb} 
yield ${\Aff{\ri{C_{1}} \.} = \Aff{C_{1}} = \Aff{\rc{C_{1}} \.}}$, 
and hence we get ${\Aff{C_{2}} = \Aff{C_{1}}}$ 
owing to the hypothesis ${\ri{C_{1}} \inc C_{2} \inc \rc{C_{1}}}$. 

\smallskip

As a consequence, one obtains ${\ri{C_{1}} \inc \ri{C_{2}}}$. 

\smallskip

Now, any two points ${x \neq y}$ in ${\rc{C_{2}} \inc \rc{\rc{C_{1}} \.} = \rc{C_{1}}}$ 
satisfy ${{]x , y[} \inc \ri{C_{1}}}$ since $C_{1}$ is strictly convex, 
and hence we get ${{]x , y[} \inc \ri{C_{2}}}$ 
from the obvious inclusion ${\ri{C_{1}} \inc \ri{C_{2}}}$. 

\smallskip

This proves that $C_{2 \!}$ is strictly convex. 
\end{proof}

\bigskip

\begin{proposition} \label{prop:strict-cvx-cone} 
   Let $C \.$ be a strictly convex cone in a topological real vector space $V \!\.$ 
   whose interior is not empty and which is not equal to the whole space $V \!\!$. 
   
   \smallskip
   
   Then $V \!\.$ is one-dimensional and $C \cup \{ 0 \}$ is a ray. 
\end{proposition}

\bigskip

\begin{proof} \ 

\noindent \textasteriskcentered \ 
Assume first that $C$ is pointed. 

\smallskip

Then the boundary $\bd{C}$ of $C$ is also a cone by Point~3 in Remark~\ref{rem:coneop}. 

\smallskip

Since $C$ is not equal to the whole space $V \!\!$, 
we have ${0 \in C \,\. \setmin \:\! \intr{C} \inc \bd{C}}$ 
by Point~2 in Remark~\ref{rem:0-tvs}, and hence the boundary $\bd{C}$ is reduced to the origin 
(indeed, if there were a point ${x \neq 0}$ in $\bd{C \.}$, 
then the ray ${[0 , {+\infty}) x}$ would lie in the pointed cone $\bd{C \.}$, 
and this is not possible since $C$ is strictly convex). 

\smallskip

As a consequence, we get ${C \inc \clos{C} = \intr{C} \cup \bd{C} = \intr{C} \cup \{ 0 \} \inc C \.}$, 
that is, ${C = \clos{C} = \intr{C} \cup \{ 0 \}}$. 

\smallskip

We therefore have ${C \,\. \setmin \{ 0 \} = \clos{C} \cap (V \setmin \{ 0 \} \. )}$, 
which proves that ${C \,\. \setmin \{ 0 \}}$ is closed in ${V \setmin \{ 0 \}}$. 

\smallskip

On the other hand, one has ${C \,\. \setmin \{ 0 \} = \intr{C} \inc V \setmin \{ 0 \}}$, 
which shows that ${C \,\. \setmin \{ 0 \}}$ is open in ${V \setmin \{ 0 \}}$. 

\smallskip

Since ${C \,\. \setmin \{ 0 \} = \intr{C}}$ is not empty 
and different from ${V \setmin \{ 0 \}}$ by hypothesis, 
this implies that ${V \setmin \{ 0 \}}$ is not connected. 

\smallskip

Therefore, $V \!$ is one-dimensional since any topological real vector space of dimension 
greater than one is arcwise connected. 

\smallskip

Finally, since $V \!$ is the union of two distinct rays, 
the pointed cone $C$ is equal to one of these rays by Point~1 in Remark~\ref{rem:coneop} 
since we have ${C \neq V \!}$ and ${C \neq \{ 0 \}}$. 

\medskip

\noindent \textasteriskcentered \ 
Assume now that $C$ is blunt. 

\smallskip

Let us consider the pointed cone ${D \as C \cup\{ 0 \}}$ 
whose interior contains ${\intr{C} \neq \Oset \.}$. 

\smallskip

For any ${x \in C \.}$, the sequence $\seq{x \. / \. n}{n}{1}$ is in the cone $C$ 
and converges to $0$, which yields ${0 \in \clos{C \.}}$. 

\smallskip

This proves the inclusion ${D \inc \clos{C \.}}$, which implies that $D$ is strictly convex 
according to Proposition~\ref{prop:C1-C2-strictly-cvx} with ${C_{1 \!} \as C}$ and ${C_{2} \as D \.}$. 

\smallskip

Therefore, the previous point implies that the pointed cone $D$ is equal to either $V \!$ or a ray. 

\smallskip

But the first case is not possible since we would get 
${C = D \,\. \setmin \{ 0 \} = V \setmin \{ 0 \}}$, which is not a convex set 
(notice that we have ${\Oset \neq \intr{C} \inc C \inc V \setmin \{ 0 \} \.}$, 
and hence ${V \neq \{ 0 \}}$). 
\end{proof}

\bigskip

Let us now end this section by recalling some useful results 
about the continuity of convex functions that we will need in the sequel. 

\bigskip

\begin{proposition} \label{prop:gauge-lsc} 
   For any closed convex subset $S \.$ of a topological real vector space which contains the origin, 
   the gauge function $p_{S}$ of $S \.$ is lower semi-continuous. 
\end{proposition}

\bigskip

\begin{proof} 
This is a straightforward consequence of Point~3 in Remark~\ref{rem:gauge-top}. 
\end{proof}

\bigskip

\begin{remark*} 
This result can be found for example in Point~a in \cite[Proposition~3.4.5, page~132]{NicPer18}. 
\end{remark*}

\bigskip

\begin{theorem}[see~{\cite[Chapitre~II, Proposition~21, page~20]{BouEVT81}}] \label{thm:conv-cont} 
   Given a non-empty open convex subset $C \.$ of a topological real vector space $V \!\!$, 
   a convex function $\. f \. : C \to \RR$ is continuous if and only if 
   there exists a non-empty open subset $U \!\.$ of $V \!\.$ which satisfies $U \inc C \.$ 
   and such that $\. f \.$ is bounded from above on $U \!\.$. 
\end{theorem}

\bigskip

\begin{corollary}[see~{\cite[Chapitre~II, Corollaire, page~20]{BouEVT81}}] \label{cor:conv-cont-Rn} 
   Given an open convex subset $C \.$ of the canonical topological real vector space $\Rn{n}$, 
   any convex function from $C \.$ to $\RR$ is continuous. 
\end{corollary}

\bigskip

%%%%%%%%%%%%%%%%%%%%

\section{About sub-additive functions} \label{sec:About-sub-additive-functions} 

In this section, we give some relationships between the notions of sub-additivity 
and convexity for positively homogeneous functions. 

\bigskip

Let us begin with the definition of sub-additivity for an arbitrary function 
defined on a subset of a real vector space. 

\bigskip

\begin{definition} 
   Given a subset $S$ of a real vector space which is stable with respect to $+$, 
   a function ${\. f \. : S \to \RR}$ is said to be \emph{sub-additive} if we have 
   ${\. f \. (x + y) \leq \. f \. (x) + \. f \. (y)}$ for any ${x , y \in S \.}$. 
\end{definition}

\bigskip

\begin{remark*} 
For example, we may take a convex cone for $S$ according to Proposition~\ref{prop:cone-+}. 
\end{remark*}

\bigskip

\begin{proposition}[see~{\cite[Theorem~4.7, page~30]{Roc70}}] \label{prop:phs-c} 
   Let $C \.$ be a convex cone in a real vector space 
   and ${\. f \. : C \to \RR}$ a positively homogeneous function. 
   Then we have the equivalence 
   \[
   f \. \ \mbox{is convex} 
   \qquad \iff \qquad \. 
   f \. \ \mbox{is sub-additive}~.
   \] 
\end{proposition}

\bigskip

\begin{corollary} \label{cor:gauge-cvx} 
   For any convex subset $S \.$ of $V \!\.$ which contains the origin, 
   its pointed cone ${C \as \Cone{S}}$ and its gauge function $p_{S}$ satisfy the following properties: 
   
   \begin{enumerate}
      \item The set $C \.$ is convex and we have $C = \inv{p_{S}}(\RR)$. 
      
      \smallskip
      
      \item The function $h : C \to \RR$ defined by $h(x) \. \as p_{S}(x)$ is convex. 
   \end{enumerate} 
\end{corollary}

\bigskip

\begin{proof} \ 

\textsf{Point~1.} 
Since $S$ is convex, the same holds for $C$ according to Remark~\ref{rem:cvx-cone}. 

\smallskip

On the other hand, we have ${C = \inv{p_{S}}(\RR)}$ by Point~1 in Proposition~\ref{prop:gauge-vect}. 

\medskip

\textsf{Point~2.} 
Let us now notice that $h$ is positively homogeneous by Point~3 in Proposition~\ref{prop:gauge-vect}. 

\smallskip

Next, fix ${x , y \in C \.}$, and let $\l$ and $\m$ be positive real numbers 
satisfying ${h(x) < \l}$ and ${h(y) < \m}$. 

\smallskip 

This writes ${h(x \. / \. \l) < 1}$ and ${h(y \. / \. \m) < 1}$, 
which implies that $x \. / \. \l$ and $y \. / \. \m$ are in $S$ 
since we have ${\inv{h}([0 , 1) \. ) = \inv{p_{S}}([0 , 1) \. ) \inc S}$ 
by Point~5 in Proposition~\ref{prop:gauge-vect} 
knowing that $S$ is star-shaped by Point~1 in Remark~\ref{rem:cvx+star-shaped+cone}. 

\smallskip

Therefore, the convexity of $S$ yields 
\[
\frac{x + y}{\l + \m} 
\ = \ 
\Big( \frac{\l}{\l + \m} \Big) \frac{x}{\l} 
\, + \, 
\Big( \frac{\m}{\l + \m} \Big) \frac{y}{\m} \in S~,
\] 
and hence we get 
${h( \. (x + y) \:\!\! / \:\!\! (\l + \m) \. ) 
= p_{S}( \. (x + y) \:\!\! / \:\!\! (\l + \m) \. ) \leq 1}$ 
owing to the obvious inclusion ${S \inc \Cone{S} = C}$ 
and since one has ${S \inc \inv{p_{S}}([0 , 1])}$ by Point~5 in Proposition~\ref{prop:gauge-vect}. 

\smallskip

So, we get ${h(x + y) \leq \l + \m}$, which implies ${h(x + y) \leq h(x) + h(y)}$ 
since $\l$ and $\m$ are arbitrary. 

\smallskip

This proves that $h$ is sub-additive, and hence convex by Proposition~\ref{prop:phs-c}. 
\end{proof}

\bigskip

\begin{remark*} 
It is to be noticed that the function $h$ in Corollary~\ref{cor:gauge-cvx} may be convex 
even though $S$ is not convex. 

\smallskip

Indeed, the subset ${S \as ( \. {- \. 1} , 1)^{2} \cup \{ \. (1 , 1) , (1 , {- \. 1}) \. \} \.}$ 
of $\Rn{2}$ is not convex and we obviously have ${C \as \Cone{S} = \Rn{2}}$. 
On the other hand, we can easily check that the subset ${T \. = [{- \. 1} , 1]^{2} \!}$ of $\Rn{2}$ 
satisfies $\Cone{T} = C = \Rn{2}$ 
and ${h(x , y) \. \as p_{S}(x , y) = p_{\,\. T}(x , y)}$ for any ${(x , y) \in \Rn{2}}$. 
Therefore, since $T$ is star-shapped and convex, 
the function $h$ is convex by Point~2 in Corollary~\ref{cor:gauge-cvx}. 
\end{remark*}

\bigskip

\begin{proposition} \label{prop:cvx-gauge-cont} 
   Let $S \.$ be a non-empty convex subset of a topological real vector space $V \!\.$ 
   and $p_{S}$ the gauge function of $S \.$. Then we have the equivalence 
   \[
   p_{S} \ \mbox{is continuous} 
   \qquad \iff \qquad 
   S \. \ \mbox{is a neighborhood of the origin in} \ V \!\!~.
   \] 
\end{proposition}

\smallskip %\bigskip

\begin{proof} \ 

$(\imp)$. 
Assume that $p_{S}$ is continuous. 

\smallskip

Since it is upper semi-continuous, $S$ must be a neighborhood of the origin in $V \!$ 
by Point~2.b in Proposition~\ref{prop:gauge-top}. 

\medskip

$(\con)$. 
Assume that $S$ is a neighborhood of the origin in $V \!\.$. 

\smallskip

Then we have $C \as \Cone{S} = V \!$ by Point~1.b in Remark~\ref{rem:0-tvs}, 
and hence Corollary~\ref{cor:gauge-cvx} implies that 
$p_{S}$ is real-valued and that the function $h : C \to \RR$ defined by $h(x) \. \as p_{S}(x)$ is convex. 

\smallskip

On the other hand, the non-empty open set $\intr{S}$ lies in $\inv{p_{S}}([0 , 1])$ 
by Point~1.a in Proposition~\ref{prop:gauge-top}, which proves that $h$ is bounded from above 
on $\intr{S} \.$. 

\smallskip

Therefore, $h$ is continuous owing to Theorem~\ref{thm:conv-cont}, 
and hence the same holds for ${p_{S} = j \comp h}$, where $j$ denotes the canonical inclusion of $\RR$ 
into $\clos{\RR}$. 
\end{proof}

\bigskip

\begin{remark*} 
It is to be noticed that Proposition~\ref{prop:cvx-gauge-cont} is a generalization to arbitrary 
topological vector spaces of a result given in Point~c in \cite[Proposition~3.4.5, page~132]{NicPer18} 
for normed vector spaces. 
\end{remark*}

\bigskip

%%%%%%%%%%%%%%%%%%%%

\section{About sub-convex functions} \label{sec:About-sub-convex-functions}

In this section, we deal with sub-convex functions, and then define the key notion 
of \emph{strict} sub-convexity, which we shall link up with strict convexity 
for positively homogeneous functions. 

%%%%%%%%%%%%%%%%%%%%

\subsection{Geometric aspects of sub-convexity}

We first introduce (strictly) quasi-convex functions, 
whose class is wider than that of (strictly) convex functions, 
but whose properties are nevertheless very close to (strict) convexity 
(we may refer to~\cite{GrePie71} and~\cite[Part~I, page~3]{HKS05} for an overview). 

\begin{definition} 
   Given a convex subset $C$ of a real vector space, a function ${\. f \. : C \to \RR}$ is said to be 
   
   \begin{enumerate}
      \item \emph{quasi-convex} if we have 
      ${\. f \. ( \. (1 - t) x + t y) \leq \max{\! \{ f \. (x) , f \. (y) \. \}}}$ 
      for any points ${x , y \in C}$ and any number ${t \in (0 , 1)}$. 
      
      \smallskip
      
      \item \emph{strictly quasi-convex} if we have 
      ${\. f \. ( \. (1 - t) x + t y) < \max{\! \{ f \. (x) , f \. (y) \. \}}}$ 
      for any two distinct points ${x , y \in C}$ and any number ${t \in (0 , 1)}$. 
   \end{enumerate} 
\end{definition}

\bigskip

It is to be noticed that both quasi-convexity and strict quasi-convexity of functions 
are mere affine notions. 

\bigskip

\begin{remark} \label{rem:sqc} \ 

\begin{enumerate}[1)]
   \item A strictly quasi-convex function is of course quasi-convex. 
   
   \medskip
   
   \item A convex function is of course quasi-convex, but the converse is obviously not true 
   as we can check with the function ${\. f \. : \RR \to \RR}$ defined by 
   ${\disp \. f \. (x) \. \as \sqrt{|x|}}$. 
   
   \medskip
   
   \item On the other hand, a strictly convex function is strictly quasi-convex, 
   but the converse is clearly false as we can see with the absolute value function on the real line. 
   
   \medskip
   
   \item Given a convex cone $C$ in a real vector space $V \!$ 
   and a positively homogeneous function ${\. f \. : C \to \RR}$ which is a strictly quasi-convex, 
   we have ${\. \inv{f}(0) \inc \{ 0 \}}$. 
   
   \noindent Indeed, any vector $x \neq 0$ in $V \!$ 
   belongs to the open line segment $]x \. / \. 2 \, , \, 2 x[$, which yields 
   
   \smallskip
   
   \centerline{
   $f \. (x) \ < \ \max{\! \{ f \. (x \. / \. 2) \, , \, f \. (2x) \. \}} \. 
   \ = \ \max{\! \{ f \. (x) \. / \. 2 \, , \, 2 f \. (x) \. \}}$
   } 
   
   \smallskip
   
   \noindent owing to the strict quasi-convexity and the positive homogeneity of $\. f \!$, 
   and hence we cannot have $\. f \. (x) = 0$. 
\end{enumerate} 
\end{remark}

\bigskip

\begin{proposition} \label{prop:sqc} 
   Let $C \.$ be a convex subset of a real vector space, 
   ${\. f \. : C \to \RR}$ a non-negative function and ${\a > 0}$ a real number. 
   Then we have the equivalence 
   \[
   f \. \ \mbox{is (strictly) quasi-convex} 
   \qquad \iff \qquad \. 
   f^{\a} \! \ \mbox{is (strictly) quasi-convex}~.
   \] 
\end{proposition}

\bigskip

\begin{proof} 
This is straightforward since the function ${\f : [0 , {+\infty}) \to \RR}$ 
defined by ${\f(t) \. \as t^{\a} \!}$ is increasing. 
\end{proof}

\bigskip

\begin{proposition} \label{prop:strict-c-strict-qc} 
   Given a convex subset $C \.$ of a real vector space, 
   any non-negative function ${\. f \. : C \to \RR}$ such that $\. f^{\a} \!$ is (strictly) convex 
   for some real number ${\a > 0}$ is (strictly) quasi-convex. 
\end{proposition}

\bigskip

\begin{proof} 
Assume that $\. f^{\a} \!$ is (strictly) convex for some real number $\a > 0$. 

\smallskip

According to Points~2 and~3 in Remark~\ref{rem:sqc}, 
the function $\. f^{\a} \!$ is (strictly) quasi-convex, 
and hence $\. f \.$ is (strictly) quasi-convex by Proposition~\ref{prop:sqc}. 
\end{proof}

\bigskip

We shall now consider the class of sub-convex functions 
and give its relationships with both convexity and quasi-convexity. 

\bigskip

\begin{definition} 
   Given a subset $C$ of a real vector space, 
   a function $\. f \. : C \to \RR$ is said to be \emph{sub-convex} 
   if the sublevel set $S_{r}(f)$ is convex for any $r \in \RR$. 
   \end{definition}

\bigskip

\begin{remark} \label{rem:sub-cvx} \ 

\begin{enumerate}[1)]
   \item The domain $C$ of such a function $\. f \.$ is actually convex 
   since the non-decreasing family $(S_{r}(f) \. )_{r \in \RR}$ covers $C \.$. 
   Therefore, when dealing with sub-convexity, 
   we will always consider functions defined on convex domains. 
   
   \medskip
   
   \item In case when $C$ is a cone and $\. f \. $ is non-negative and positively homogeneous, 
   then Point~1 in Proposition~\ref{prop:ph-sl} with ${\a \as 1}$ 
   and Point~2 in Remark~\ref{rem:sublevel-basic} with ${r \as 0}$ 
   show that the sub-convexity of $\. f \.$ is equivalent to the convexity of $S_{1 \.}(f)$ 
   since both the image of a convex set by a homothety 
   and the intersection of a family of convex sets are convex sets. 
\end{enumerate} 
\end{remark}

\bigskip

\begin{proposition} \label{prop:sub-cvx-strict-sublevel} 
   Given a convex subset $C \.$ of a real vector space, 
   a function $\. f \. : C \to \RR$ is sub-convex if and only if 
   $\. \inv{f}( \. ({-\infty} , r) \. )$ is convex for any $r \in \RR$. 
\end{proposition}

\bigskip

\begin{proof} 
We refer to~\cite[Lemma~1.27, page~32]{HKS05} for a proof of this result. 
\end{proof}

\bigskip

\begin{proposition} \label{prop:cvx-scvx} 
   For any convex subset $C \.$ of a real vector space and any function $\. f \. : C \to \RR$, 
   we have the following implication: 
   \[
   f \. \ \mbox{is convex} 
   \qquad \imp \qquad \. 
   f \. \ \mbox{is sub-convex}~.
   \] 
\end{proposition}

\bigskip

\begin{proof} 
Since the preimage of a convex set by a convex function is itself convex, the implication is proved. 
\end{proof}

\bigskip

The converse of the implication in Proposition~\ref{prop:cvx-scvx} is of course not true 
as one can check by considering the function $\. f \. : \RR \to \RR$ defined by 
$\. f \. (t) \. \as 0$ for $t < 0$ and ${\. f \. (t) \. \as {-t}}$ for $t \geq 0$. 

\bigskip

Let us now turn our attention to the relationship between sub-convexity and quasi-convexity. 

\bigskip

\begin{proposition} \label{prop:sc-qc} 
   For any convex subset $C \.$ of a real vector space and any function $\. f \. : C \to \RR$, 
   the following equivalence holds: 
   \[
   f \. \ \mbox{is sub-convex} 
   \qquad \iff \qquad \. 
   f \. \ \mbox{is quasi-convex}~.
   \] 
\end{proposition}

\bigskip

\begin{proof} 
We refer to~\cite[Point~50, page~69 (reformatted by Border)]{Fen53} for a proof of this result. 
\end{proof}

\bigskip

\begin{remark} 
This equivalence shows that there is no need to be cautious 
when mixing sub-convexity and quasi-convexity, 
and this is actually what many authors do in the literature (see for example~\cite[Chapter~1]{HKS05}). 
On the contrary, this is no longer the case when dealing with strict quasi-convexity 
and strict sub-convexity as we shall see in the next subsection. 
\end{remark}

\bigskip

The next property gives a way for constructing new sub-convex functions from old ones. 

\bigskip

\begin{proposition} \label{prop:sc-comp} 
   Let $C \.$ be a convex subset of a real vector space, 
   ${\. f \. : C \to \RR}$ a sub-convex function, $D$ a subset of $\RR$ which contains $\. f \. (C)$ 
   and ${\f : D \to \RR}$ a non-decreasing function. 
   Then the function $g : C \to \RR$ defined by $g(x) \. \as \f[f \. (x) \. ]$ is sub-convex. 
\end{proposition}

\bigskip

\begin{proof} 
For any ${x , y \in C}$ and ${t \in [0 , 1]}$, we have 
${\. f \. ( \. (1 - t) x + t y) \leq \max{\! \{ f \. (x) , f \. (y) \. \}} \.}$ 
since the sub-convexity of $\. f \.$ is equivalent to its quasi-convexity 
by Proposition~\ref{prop:sc-qc}, which yields 

\smallskip

\centerline{
$g( \. (1 - t) x + t y) 
\ \leq \ 
\f(\max{\! \{ f \. (x) , f \. (y) \. \}} \. ) 
\ = \ 
\max{\! \{ g(x) , g(y) \. \}} \.$~.
} 

\smallskip

This proves that $g$ is quasi-convex, and hence sub-convex by Proposition~\ref{prop:sc-qc}. 
\end{proof}

\bigskip

If we pick a real number ${\a > 0}$ and apply Proposition~\ref{prop:sc-comp} 
to the function ${\f : [0 , {+\infty}) \to \RR}$ defined by ${\f(t) \. \as t^{\a} \!}$, 
then we obtain the following result. 

\bigskip

\begin{corollary} \label{cor:asc} 
   Let $C \.$ be a convex subset of a real vector space, 
   ${\. f \. : C \to \RR}$ a non-negative function and ${\a > 0}$ a real number. 
   Then we have the equivalence 
   \[
   f^{\a} \! \ \mbox{is sub-convex} 
   \qquad \iff \qquad \. 
   f \. \ \mbox{is sub-convex}~.
   \] 
\end{corollary}

\bigskip

Let us now improve Proposition~\ref{prop:cvx-scvx} 
for non-negative and positively homogeneous functions. 

\bigskip

\begin{proposition} \label{prop:sc-c} 
   Let $C \.$ be a convex cone in a real vector space 
   and ${\. f \. : C \to \RR}$ a non-negative function which is positively homogeneous. 
   Then we have the equivalence 
   \[
   f \. \ \mbox{is sub-convex} 
   \qquad \iff \qquad \. 
   f \. \ \mbox{is convex}~.
   \] 
\end{proposition}

\bigskip

\begin{proof}
First of all, we may assume that $C$ is not empty 
since in this case the equivalence to be proved is trivial. 

\smallskip

As we only have to prove the implication $\! \imp \!$ according to Proposition~\ref{prop:cvx-scvx}, 
let us assume that $\. f \.$ is sub-convex. 

\smallskip

So, $S_{1 \.}(f)$ is convex, and hence the set $S \as \widehat{S_{1 \.}(f)}$ is also convex 
by Proposition~\ref{prop:sshc}. 

\smallskip

Moreover, according to Point~3 in Proposition~\ref{prop:ph-sl} with $r \as 1 > 0$, 
we have $S = S_{1 \.}(f) \cup \{ 0 \}$. 

\smallskip

Now, since $S$ is a star-shaped subset of $C \cup \{ 0 \}$ which contains $\. \inv{f}([0 , 1) \. )$, 
we have $\rest{{(p_{S}) \.\.}}{C} = j \comp \. f$ % no \. at the end of the line
by Corollary~\ref{cor:f=p}, where $j$ denotes the canonical inclusion of $\RR$ into $\bar{\RR}$. 

\smallskip

On the other hand, we have $\Cone{S} = \Cone{S_{1 \.}(f) \.} = C \cup \{ 0 \}$ 
by Point~1 in Remark~\ref{rem:cone} and Point~1 in Corollary~\ref{cor:1-sublevel-f-span-C}. 

\smallskip

Finally, since the set $S$ is non-empty, star-shaped and convex, 
this implies that the function $h : C \cup \{ 0 \} \to \RR$ defined by $h(x) \. \as p_{S}(x)$ is convex 
by Corollary~\ref{cor:gauge-cvx}, and hence $\. f \. = \rest{h}{C}$ is convex too. 
\end{proof}

\bigskip

\begin{remark*} \ 

\begin{enumerate}[1)]
   \item According to Point~2 in Remark~\ref{rem:sub-cvx}, it is to be noticed 
   that Proposition~\ref{prop:sc-c} corresponds to the result given 
   in~\cite[Lemma~3.4.2, page~130]{NicPer18}. 
   
   \medskip
   
   \item On the other hand, Proposition~\ref{prop:sc-c} has also been proved 
   in~\cite[Theorem~3, page~208]{Ber63} in the particular case where $V \!$ is equal to $\Rn{n}$ 
   and where $\. f \.$ is positive outside the origin. 
   
   \medskip
   
   \item Proposition~\ref{prop:sc-c} is useful to avoid long computations in differential calculus. 
   For example, if we consider the convex cone 
   ${C \as (0 , {+\infty}) \cart \RR}$ in $\Rn{2}$, the function ${\. f \. : C \to \RR}$ defined by 
   ${\. f \. (x , y) \. \as (x^{2} \! + y^{2}) \:\!\! / \:\!\! (2 x)}$ is convex 
   since it is non-negative, positively homogeneous and sub-convex (indeed, $S_{0}(f)$ is empty 
   and $S_{1 \.}(f)$ is the closed disk in $\Rn{2}$ about $(1 , 0)$ with radius~$1$ less the origin). 
\end{enumerate} 
\end{remark*}

\bigskip

\begin{corollary} \label{cor:hsc-c} 
   For any convex cone $C \.$ in a real vector space 
   and any non-negative function $\. f \. : C \to \RR$ 
   which is positively homogeneous of degree $\a \geq 1$, we have the equivalence 
   \[
   f \. \ \mbox{is sub-convex} 
   \qquad \iff \qquad \. 
   f \. \ \mbox{is convex}~.
   \] 
\end{corollary}

\bigskip

\begin{proof} 
As we only have to prove the implication $\! \imp \!$ according to Proposition~\ref{prop:cvx-scvx}, 
let us assume that $\. f \.$ is sub-convex. 

\smallskip

Therefore, the function $\. f^{1 \. / \. \a} \!$ is sub-convex by Corollary~\ref{cor:asc}, 
and hence Proposition~\ref{prop:sc-c} implies that it is convex since it is positively homogeneous. 

\smallskip

Finally, using Proposition~\ref{prop:ac}, 
we get that $\. f \. = (f^{1 \. / \. \a})^{\. \a} \!$ is convex. 
\end{proof}

\bigskip

\begin{remark} 
It is to be noticed that the implication $\! \imp \!$ is no longer true if we have ${\a < 1}$ 
since the function ${\. f \. : \RR \to \RR}$ defined by 
${\disp \. f \. (t) \. \as \sqrt{|t|}}$ is not convex 
even though it is sub-convex and positively homogeneous of degree $1 \. / \. 2$. 
\end{remark}

%%%%%%%%%%%%%%%%%%%%

\subsection{Topological aspects of sub-convexity}

We shall now deal with the topological notion of strict sub-convexity for functions 
defined on a general topological real vector space, and show how it is related to continuity 
and strict convexity when these functions have the extra property of being positively homogeneous. 

\bigskip

\begin{definition} \label{def:strict-sub-cvx} 
   Given a subset $C$ of a topological real vector space, 
   a function $\. f \. : C \to \RR$ is said to be \emph{strictly sub-convex} 
   if the sublevel set $S_{r}(f)$ is strictly convex for any $r \in \RR$. 
\end{definition}

\bigskip

\begin{remark} \label{rem:strict-sub-cvx} \ 

\begin{enumerate}[1)]
   \item According to Point~2 in Remark~\ref{rem:strict-cvx}, 
   any strictly sub-convex function $\. f \. : C \to \RR$ defined on a subset $C$ 
   of a topological real vector space is sub-convex, 
   and hence its domain $C$ is necessarily convex by Point~1 in Remark~\ref{rem:sub-cvx}. 
   
   \medskip
   
   \item A function defined on an interval of $\RR$ is of course strictly sub-convex 
   if and only if it is sub-convex. 
   
   \medskip
   
   \item A norm $\norm{\cdot}$ on a real vector space $V \!$ is strictly sub-convex 
   with respect to the topology associated with $\norm{\cdot}$ if and only if 
   the normed vector space $(V , \norm{\cdot})$ 
   is ``strictly convex'' in the sense given in~\cite[page~108]{Car04} and \cite[page~30]{JohLin01}. 
   This is a mere consequence of the equivalence Point~1$\iff$Point~4 
   in Proposition~\ref{prop:Minkowski-strict-sub-additivity} 
   since in $(V , \norm{\cdot})$ the topological boundary of the unit closed ball 
   is exactly the unit sphere. 
   
   \medskip
   
   \item It is to be noticed that Proposition~\ref{prop:sub-cvx-strict-sublevel} does not hold 
   when sub-convexity is replaced by strict sub-convexity. 
   In other words, the strict sub-convexity of a function ${\. f \. : C \to \RR}$ 
   defined on a subset $C$ of a topological real vector space cannot be characterized by saying 
   that $\. \inv{f}( \. ({-\infty} , r) \. )$ is strictly convex for any ${r \in \RR}$. 
   
   \smallskip
   
   \noindent \textasteriskcentered \ 
   Indeed, if we consider the set ${C \as \Rn{2}}$ 
   and the function $\. f \. : C \to \RR$ defined by 
   
   \smallskip
   
   \centerline{
   $\disp f \. (x , y) \. \as \. \max{\! \{ 0 \, , \, \sqrt{2 (x^{2} \. + y^{2})} - 2 y \}} \.$~,
   } 
   
   \smallskip
   
   \noindent the sublevel set ${S_{0}(f) = \{ (x , y) \in \Rn{2} \st y \geq |x| \}}$ 
   is not strictly convex, which shows that $\. f \.$ is not strictly sub-convex. 
   
   \noindent Nevertheless, for any ${r \in \RR}$, we have 
   
   \smallskip
   
   \centerline{
   $\inv{f}( \. ({-\infty} , r) \. ) 
   \ = \ 
   \{ (x , y) \in \Rn{2} \st (y + r)^{\. 2} \! - x^{2} \. > r^{2} \!\. / \. 2 \} 
   \cap [\RR \cart \. ({-r} \. / \. 2 , {+\infty}) \. ]$
   } 
   
   \smallskip
   
   \noindent if one has ${r > 0}$ 
   and ${\. \inv{f}( \. ({-\infty} , r) \. ) = \Oset}$ if one has ${r \leq 0}$, 
   which shows that $\. \inv{f}( \. ({-\infty} , r) \. )$ is strictly convex. 
   
   \smallskip
   
   \noindent \textasteriskcentered \ 
   On the other hand, if we consider the set ${C \as \Rn{2}}$ 
   and the function ${\. f \. : C \to \RR}$ defined by 
   
   \smallskip
   
   \centerline{
   $\disp f \. (x , y) \. \as \. \min{\! \{ 0 \, , \, \sqrt{2 (x^{2} \. + y^{2})} - 2 y \}} \.$~,
   } 
   
   \smallskip
   
   \noindent we get ${S_{r}(f) = \Rn{2}}$ for any ${r \in [0 , {+\infty})}$ and 
   
   \smallskip
   
   \centerline{
   $S_{r}(f) 
   \ = \ 
   \{ (x , y) \in \Rn{2} \st (y + r)^{\. 2} \! - x^{2} \. \geq r^{2} \!\. / \. 2 \} 
   \cap [\RR \cart \. [{-r} \. / \. 2 , {+\infty}) \. ]$
   } 
   
   \smallskip
   
   \noindent for any ${r \in ({-\infty} , 0)}$, which shows that $\. f \.$ is strictly sub-convex. 
   
   \noindent Nevertheless, the set 
   ${\. \inv{f}( \. ({-\infty} , 0) \. ) = \{ (x , y) \in \Rn{2} \st y > |x| \}}$ 
   is not strictly convex. 
\end{enumerate} 
\end{remark}

\bigskip

Let us now give the relationship between strict sub-convexity and strict quasi-convexity. 

\bigskip

\begin{proposition} \label{prop:sqc+c-ssc} 
   For any strictly convex subset $C \.$ of a topological real vector space $V \!\.$ 
   and any continuous function $\. f \. : C \to \RR$, we have the following implication: 
   \[
   f \. \ \mbox{is strictly quasi-convex} 
   \qquad \imp \qquad \. 
   f \. \ \mbox{is strictly sub-convex}~.
   \] 
\end{proposition}

\bigskip

\begin{proof} 
Assume that $\. f \.$ is strictly quasi-convex, fix a number ${r \in \RR}$, 
and consider two points $x \neq y$ in $\rc{S_{r}(f) \.} \inc \rc{C}$ 
(see Point~1 in Proposition~\ref{prop:ri-rc-rb}). 

\smallskip

Then we have ${{]x , y[} \inc \ri{C}}$ since $C$ is strictly convex. 

\smallskip

Moreover, the quasi-convexity of $\. f \.$ insures that $S_{r}(f)$ is convex 
by Proposition~\ref{prop:sc-qc}, which implies that $\rc{S_{r}(f) \.}$ is also convex 
owing to Point~1 in Proposition~\ref{prop:conv-top}, 
and hence we obtain the inclusion ${]x , y[} \inc \rc{S_{r}(f) \.}$. 

\smallskip

On the other hand, we compute 
%%%%%%%%%%
\begin{eqnarray*} 
   \rc{S_{r}(f) \.} \cap \ri{C} \! 
   & = & \! 
   \clos{S_{r}(f)} \cap \Aff{S_{r}(f) \.} \cap \ri{C} \\ 
   & = & \! 
   \clos{S_{r}(f)} \cap C \cap \Aff{S_{r}(f) \.} \cap \ri{C} \\ 
   & & \quad 
   \mbox{(since one has $\ri{C} \inc C$)} \\ 
   & = & \! 
   S_{r}(f) \cap \Aff{S_{r}(f) \.} \cap \ri{C} \\ 
   & & \quad 
   \mbox{(since $S_{r}(f)$ is closed in $C$ by lower semi-continuity of $\. f$)} \\ 
   & = & \! 
   S_{r}(f) \cap \ri{C} \\ 
   & & \quad 
   \mbox{(since the inclusion $S_{r}(f) \inc \Aff{S_{r}(f) \.}$ holds)}~, 
\end{eqnarray*} 
%%%%%%%%%%
which proves ${]x , y[} \inc S_{r}(f) \cap \ri{C \.}$. 

\smallskip

Now, given an arbitrary ${\a \in (0 , 1)}$, fix ${s \in (0 , \a)}$ and ${t \in (\a , 1)}$, 
and define ${x' \! \as (1 - s) x + s y}$ and ${y' \! \as (1 - t) x + t y}$. 

\smallskip

Then the distinct points $x' \!$ and $y' \!$ are in ${]x , y[}$, 
and hence in $S_{r}(f) \cap \ri{C}$ as shown above. 

\smallskip

Moreover, the point ${z \as (1 - \a) x + \a y}$ belongs to the open line segment $]x' \! , y'[$, 
and hence the strict quasi-convexity of $\. f \.$ yields 
${\. f \. (z) < \max{\! \{ f \. (x') , f \. (y') \. \}} \leq r}$, 
which implies ${z \in \:\!\! \inv{f}( \. ({-\infty} , r) \. )}$. 

\smallskip

Finally, since $\ri{C}$ is open in $\Aff{C}$ and since $\. \inv{f}( \. ({-\infty} , r) \. )$ 
is open in $C$ by the upper semi-continuity of $\. f \!$, 
there exist open sets $U \!$ and $W \!$ in $V \!$ such that one can write 
${\ri{C} = U \. \cap \Aff{C}}$ and ${\. \inv{f}( \. ({-\infty} , r) \. ) = W \! \cap C \.}$, 
which yields 
%%%%%%%%%%
\begin{eqnarray*} 
   z \in \:\!\! \inv{f}( \. ({-\infty} , r) \. ) \cap {]x , y[} \! 
   & \inc & \! 
   \inv{f}( \. ({-\infty} , r) \. ) \cap \ri{C} \\ 
   & & \quad 
   \mbox{(since we proved ${]x , y[} \inc \ri{C}$ above)} \\ 
   & \inc & \! 
   \inv{f}( \. ({-\infty} , r) \. ) \cap \Aff{S_{r}(f) \.} \cap \ri{C} \\ 
   & & \quad 
   \mbox{(since we have 
   $\. \inv{f}( \. ({-\infty} , r) \. ) \inc S_{r}(f) \inc \Aff{S_{r}(f) \.} \.$)} \\ 
   & = & \! 
   (W \! \cap C) \cap \Aff{S_{r}(f) \.} \cap \ri{C} 
   \ = \ 
   W \! \cap \Aff{S_{r}(f) \.} \cap \ri{C} \\ 
   & & \quad 
   \mbox{(since the inclusion $\ri{C} \inc C$ holds)} \\ 
   & = & \! 
   W \! \cap \Aff{S_{r}(f) \.} \cap (U \. \cap \Aff{C}) \\ 
   & = & \! 
   \Om \as (U \. \cap W) \cap \Aff{S_{r}(f) \.} \\ 
   & & \quad 
   \mbox{(since we have $\Aff{S_{r}(f) \.} \inc \Aff{C}$ from $S_{r}(f) \inc C$)}~. 
\end{eqnarray*} 
%%%%%%%%%%

\smallskip

But, using again the second line of the above computations, we have 

\smallskip

\centerline{
$\Om \ = \ \:\!\! \inv{f}( \. ({-\infty} , r) \. ) \cap \Aff{S_{r}(f) \.} \cap \ri{C} 
\ \inc \ 
\inv{f}( \. ({-\infty} , r) \. ) \ \inc \ S_{r}(f)$~,
} 

\smallskip

and this yields the inclusion $\Om \inc \ri{S_{r}(f) \.}$ 
since $\Om$ is an open set in $\Aff{S_{r}(f) \.}$, which finally leads to $z \in \ri{S_{r}(f) \.}$. 

\smallskip

Conclusion: since the points $x \neq y$ in $\rc{S_{r}(f) \.}$ and $\a \in (0 , 1)$ 
have been chosen arbitrarily, the sublevel set $S_{r}(f)$ is therefore strictly convex. 

\smallskip

This proves that $\. f \.$ is strictly sub-convex since $r \in \RR$ is arbitrary. 
\end{proof}

%\bigskip
\pagebreak

\begin{remark} \label{rem:sqc+c-ssc} \ 

\begin{enumerate}[1)]
   \item Let us notice that Proposition~\ref{prop:sqc+c-ssc} does not hold 
   if $C$ is not strictly convex as we can see with the non-strictly convex subset 
   ${C \as [{- \. 1} , 1] \cart [{- \. 1} , 1]}$ of $\Rn{2}$ 
   and the strictly convex function ${\. f \. : C \to \RR}$ 
   defined by ${\. f \. (x , y) \. \as x^{2} \! + y^{2} \!}$ for in this case $\. f \.$ is continuous 
   and strictly quasi-convex (see Point~3 in Remark~\ref{rem:sqc}) 
   but the sublevel set ${S_{2}(f) = C}$ is not strictly convex. 
   
   \medskip
   
   \item Moreover, Proposition~\ref{prop:sqc+c-ssc} is false if $\. f \.$ is not continuous. 
   
   \noindent Indeed, let us consider the convex subset 
   ${\disp C \as \{ u \in \Rn{2} \st \norm{u} \leq \sqrt{2} \}}$ of $\Rn{2}$, 
   where $\norm{\cdot}$ stands for the canonical Euclidean norm on $\Rn{2}$, 
   and the function ${\. f \. : C \to \RR}$ defined by 
   ${\. f \. (u) \. \as \norm{u}^{\. 2}}$ % no \! at the end of the line
   for ${u \in [{- \. 1} , 1] \cart [{- \. 1} , 1]}$ 
   and ${\. f \. (u) \. \as \norm{u}^{\. 2} \! + 2}$ 
   for ${u \in C \setmin ([{- \. 1} , 1] \cart [{- \. 1} , 1])}$, 
   which is obviously not upper semi-continuous. 
   
   \noindent Now, given two points $v$ and $w$ in $C \.$, there are three cases to be considered. 
   
   \smallskip
   
   \noindent \textasteriskcentered \ 
   If $v$ and $w$ both belong to ${[{- \. 1} , 1] \cart [{- \. 1} , 1]}$ 
   or to one of the four connected components of ${C \setmin ([{- \. 1} , 1] \cart [{- \. 1} , 1])}$, 
   we have ${\. f \. ( \. (1 - t) v + t w) < \max{\! \{ f \. (v) , f \. (w) \. \}}}$ 
   for all ${t \in (0 , 1)}$ since $\. f \.$ is strictly convex 
   on each of these five convex sets---and hence strictly quasi-convex 
   owing to Point~3 in Remark~\ref{rem:sqc}. 
   
   \smallskip
   
   \noindent \textasteriskcentered \ 
   If $v$ is in ${[{- \. 1} , 1] \cart [{- \. 1} , 1]}$ 
   and $w$ belongs to one of the four connected components 
   of ${C \setmin ([{- \. 1} , 1] \cart [{- \. 1} , 1])}$, 
   we have ${\. f \. ( \. (1 - t) v + t w) \leq 2 < f \. (w)}$ for all ${t \in [0 , 1)}$, 
   and this yields ${\. f \. ( \. (1 - t) v + t w) < \max{\! \{ f \. (v) , f \. (w) \. \}}}$ 
   since the first inequality with $t \as 0$ writes $f \. (v) \leq f \. (w)$. 
   
   \smallskip
   
   \noindent \textasteriskcentered \ 
   If $v$ and $w$ belong to different connected components 
   of ${C \setmin ([{- \. 1} , 1] \cart [{- \. 1} , 1])}$, 
   we have ${\. f \. (u) \leq \norm{u}^{\. 2} \! + 2 < \max{\! \{ f \. (v) , f \. (w) \. \}}}$ 
   for all ${u \in {]v , w[}}$ since the second inequality is given by the strict quasi-convexity 
   of the strictly convex function ${\norm{\cdot}^{\. 2} \! + 2}$ 
   (see Point~3 in Remark~\ref{rem:sqc}). 
   
   \smallskip
   
   \noindent Summing up, this proves that $\. f \.$ is strictly quasi-convex. 
   
   \smallskip
   
   \noindent Nevertheless, the function $\. f \.$ is not strictly sub-convex 
   since the sublevel set $S_{2}(f)$ is not strictly convex. 
   
   \noindent It should also be noticed that the function $\. f \.$ is lower semi-continuous 
   (indeed, given any real number $\a$, the set $\{ u \in C \st f \. (u) \leq \a \}$ 
   is obviously closed in $C$). 
   
   \medskip
   
   \item On the other hand, the converse implication $\! \con \!$ in Proposition~\ref{prop:sqc+c-ssc} 
   is not true---even though the function is continuous---as we can see with the zero function 
   defined on the real line. 
   
   \medskip
   
   \item Replacing ``strictly sub-convex'' by ``strictly convex'' 
   in Proposition~\ref{prop:sqc+c-ssc} is not possible 
   as one can check with the absolute value function defined on the real line. 
\end{enumerate} 
\end{remark}

\bigskip

Owing to the first part of Point~3 in Remark~\ref{rem:sqc} 
and Proposition~\ref{prop:sqc+c-ssc}, we get the next result. 

\bigskip

\begin{corollary} \label{cor:sc+c-ssc} 
   For any strictly convex subset $C \.$ of a topological real vector space 
   and any continuous function $\. f \. : C \to \RR$, we have the following implication: 
   \[
   f \. \ \mbox{is strictly convex} 
   \qquad \imp \qquad \. 
   f \. \ \mbox{is strictly sub-convex}~.
   \] 
\end{corollary}

\bigskip

\begin{remark} \label{rem:sc+c-ssc} \ 

\begin{enumerate}[1)]
   \item Point~1 in Remark~\ref{rem:sqc+c-ssc} shows that we cannot drop the strict convexity of $C$ 
   in the hypotheses of Corollary~\ref{cor:sc+c-ssc}. 
   
   \medskip
   
   \item Moreover, Corollary~\ref{cor:sc+c-ssc} is false if $\. f \.$ is not continuous as we can see 
   by considering the strictly convex subset ${C \as \RR}$ of $\RR$ 
   and the strictly convex function ${\. f \. : C \to \RR}$ defined by 
   ${\. f \. (x) \. \as x^{2} \!}$ since in this case $\. f \.$ is not continuous 
   and the sublevel set ${S_{1 \.}(f) = [{- \. 1} , 1]}$ is not strictly convex. 
   
   \medskip
   
   \item The converse of Corollary~\ref{cor:sc+c-ssc} is not true since otherwise 
   this would lead to a contradiction by combining Proposition~\ref{prop:sqc+c-ssc} 
   and the second part of Point~3 of Remark~\ref{rem:sqc}. 
\end{enumerate} 
\end{remark}

\bigskip

The next property gives a way for constructing new strictly sub-convex functions from old ones. 

\medskip %\bigskip

\begin{proposition} \label{prop:ssc-comp} 
   Let $C \.$ be a strictly convex subset of a topological real vector space, 
   ${\. f \. : C \to \RR}$ a strictly sub-convex function, $D$ a subset of $\RR$ 
   which contains $\. f \. (C)$ and ${\f : D \to \RR}$ a non-decreasing function 
   which is lower semi-continuous. 
   Then the function $g : C \to \RR$ defined by $g(x) \. \as \f[f \. (x) \. ]$ is strictly sub-convex. 
\end{proposition}

\smallskip %\bigskip

\begin{proof} 
Let us fix $r \in \RR$. 

\smallskip

If we have $S_{r}(\f) = D \.$, then $S_{r}(g)$ is reduced to $C \.$, which is strictly convex. 

\smallskip

Otherwise, there exists ${t_{0} \in D}$ which satisfies ${\f(t_{0}) > r}$, and hence we get 

\smallskip

\centerline{
$S_{r}(\f) \ \inc \ \RR \setmin [t_{0} , {+\infty}) 
\ = \ 
({-\infty} , t_{0}) \ \inc \ ({-\infty} , t_{0}]$
} 

%\smallskip

since $\f$ is non-decreasing. 

\smallskip

So, since $S_{r}(\f)$ is bounded from above, let us consider its supremum $a \in \RR$. 

\smallskip

We first have ${S_{r}(\f) \inc ({-\infty} , a]}$, which yields ${S_{r}(g) \inc S_{a}(f)}$ 
by using the obvious equality ${S_{r}(g) = \:\!\! \inv{f}(S_{r}(\f) \. )}$. 

\smallskip

Now, in order to prove the reverse inclusion, 
there are two cases to be considered depending on whether $a$ belongs to $D$ or not. 

\medskip

\noindent \textasteriskcentered \ 
Assume that we have $a \in D \.$. 

\smallskip

Then, since one has ${a \in \clos{S_{r}(\f)}^{\RR} \!\!}$, 
we obtain ${a \in \clos{S_{r}(\f)}^{D} \!\!}$, which writes ${a \in S_{r}(\f)}$ 
since $S_{r}(\f)$ is closed in $D$ by the lower semi-continuity of $\f$. 

\smallskip

So we have ${\f(a) \leq r}$, which implies ${g(x) = \f[f \. (x) \. ] \leq \f(a) \leq r}$ 
for any ${x \in S_{a}(f)}$ since $\f$ is non-decreasing, and this proves ${S_{a}(f) \inc S_{r}(g)}$. 

\medskip

\noindent \textasteriskcentered \ 
Assume that we have $a \not \in D \.$. 

\smallskip

Then, for each ${x \in S_{a}(f)}$, one gets ${\. f \. (x) < a}$, 
and hence there exists ${b \in S_{r}(\f)}$ 
which satisfies ${\. f \. (x) \leq b \leq a}$ by the very definition of $a$, 
and this yields ${g(x) = \f[f \. (x) \. ] \leq \f(b) \leq r}$. 

\smallskip

This proves the inclusion $S_{a}(f) \inc S_{r}(g)$. 

\medskip

Conclusion: in both cases, we proved that $S_{r}(g)$ is equal to the strictly convex set $S_{a}(f)$. 

\smallskip

The function $g$ is therefore strictly sub-convex since $r \in \RR$ is arbitrary. 
\end{proof}

\medskip %\bigskip

\begin{remark} \ 

\begin{enumerate}[1)]
   \item It is to be noticed that the strict convexity of $C$ cannot be dropped 
   in the hypotheses of Proposition~\ref{prop:ssc-comp}. 
   
   \noindent Indeed, let us consider the open convex subset 
   ${C \as ({- \. 1} , 1) \cart ({- \. 1} , 1)}$ of $\Rn{2}$, 
   the smooth function ${\. f \. : C \to \RR}$ defined by 
   ${\. f \. (x , y) \. \as 1 \. / \. [ \. (1 - x^{2}) \. (1 - y^{2}) \. ]}$, 
   the subset ${D \as [0 , {+\infty})}$ of $\RR$ and the function ${\f : D \to \RR}$ defined by 
   ${\f(t) \. \as t \. / \! (t + 1)}$. 
   
   \noindent Next, let us fix any real number ${r > 1}$, 
   define ${\l \as 1 \. / \. r \in (0 , 1)}$ and ${\disp a \as \sqrt{1 - \l} > 0}$, 
   and introduce the function ${\s : [{-a} , a] \to \RR}$ defined by 
   ${\disp \s(x) \. \as \sqrt{1 - \l \. / \. (1 - x^{2})}}$. 
   
   \noindent Then, for any ${x \in ({-a} , a)}$, we compute 
   ${\s' \. (x) = {- \. \l} x \. / \. [ \. (1 - x^{2})^{\. 2} \s(x) \. ]}$, 
   which shows that the derivative of $\s$ is decreasing on $({-a} , a)$ 
   (use the fact that $\s$ is a symmetric function which is positive and decreasing on $(0 , a) \.$), 
   and hence that $\s$ is strictly concave. 
   
   \noindent Therefore, the sublevel set 
   ${S_{r}(f) = \{ (x , y) \st |y| \leq \s(x) \. \}}$ is strictly convex. 
   
   \noindent On the other hand, for any ${r \in ({-\infty} , 1)}$, 
   the sublevel set ${S_{r}(f) = \Oset}$ is also strictly convex. 
   
   \smallskip
   
   \noindent Finally, since ${S_{1 \.}(f) = \{ (0 , 0) \. \}}$ is strictly convex too, 
   the function $\. f \.$ is strictly sub-convex. 
   
   \smallskip
   
   \noindent Nevertheless, even though $\f$ is non-decreasing 
   and lower semi-continuous (since it is continuous), 
   the function ${g : C \to \RR}$ defined by ${g(x , y) \. \as \f[f \. (x , y) \. ]}$ 
   is \emph{not} strictly sub-convex. 
   
   \noindent This is because we have ${\f(D) \inc ({-\infty} , 1]}$, 
   and hence the sublevel set $S_{1 \.}(g)$ is equal to ${\. \inv{f}(D) = C \.}$, 
   which is not strictly convex. 
   
   %\medskip
   \pagebreak
   
   \item Moreover, Proposition~\ref{prop:ssc-comp} is false 
   if $\f$ is not semi-continous as we can check by considering 
   the second function $\. f \.$ that we used in Point~4 
   in Remark~\ref{rem:strict-sub-cvx} and the non-decreasing function 
   ${\f \as \! - \!\. \1{({-\infty}, 0)}}$ defined on $\RR$. 
   
   \noindent Indeed, we then get that $\f$ is not semi-continuous at the origin 
   and ${g \as \f \comp \. f}$ is not strictly sub-convex since 
   ${S_{- \. 1}(g) = \:\!\! \inv{f}{( \. ({-\infty},0) \. )} = \{ (x , y) \in \Rn{2} \st y > |x| \}}$ 
   is not strictly convex. 
\end{enumerate} 
\end{remark}

\bigskip

From now on, let us focus on functions which are non-negative and positively homogeneous. 

\bigskip

\begin{proposition} \label{prop:qc-cont} 
   For any convex cone $C \.$ in $\Rn{n}$ and any non-negative function ${\. f \. : C \to \RR}$ 
   which is positively homogeneous of degree $\a > 0$, we have the implication 
   \[
   f \. \ \mbox{is sub-convex} 
   \qquad \imp \qquad \. 
   f \. \ \mbox{continuous on} \ \ri{C}~.
   \] 
\end{proposition}

\bigskip

\begin{proof} 
Since $\. f^{1 \. / \. \a} \!$ is sub-convex by Corollary~\ref{cor:asc}, 
it is convex by Corollary~\ref{cor:hsc-c} 
(indeed, the function $\. f^{1 \. / \. \a} \!$ is positively homogeneous of degree one). 

\smallskip

Now, using Corollary~\ref{cor:conv-cont-Rn} in $\Aff{C} \inc \Rn{n}$, 
we get that $\. f^{1 \. / \. \a} \!$ is continuous on $\ri{C \.}$, and the same holds for $\. f \!$. 
\end{proof}

\bigskip

\begin{remark} \label{rem:qc-cont} \ 

\begin{enumerate}[1)]
   \item Proposition~\ref{prop:qc-cont} may be false if $\ri{C}$ is replaced by $C \.$. 
   
   \noindent Indeed, if we consider the convex cone 
   ${C \as ( \. (0 , {+\infty}) \cart \RR) \cup \{ (0 , 0) \. \}}$ in $\Rn{2}$ 
   and the function ${\. f \. : C \to \RR}$ defined by ${\. f \. (0 , 0) \. \as 0}$ 
   and ${\. f \. (x , y) \. \as (x^{2} \! + y^{2}) \:\!\! / \:\!\! (2 x)}$ 
   in case when one has ${x > 0}$, 
   then $\. f \.$ is non-negative, positively homogeneous and sub-convex 
   (and even strictly sub-convex since we have ${S_{0}(f) = \{ 0 \}}$ 
   and since $S_{1 \.}(f)$ is the closed disk in $\Rn{2}$ about $(1 , 0)$ of radius $1$), 
   but it is not continuous at $(0 , 0)$ since we have 
   ${\disp \. f \. (1 \. / \. n \, , \, 1 \. / \! \sqrt{n}) \to 1 \. / \. 2 \neq \. f \. (0 , 0)}$ 
   and ${\disp (1 \. / \. n \, , \, 1 \. / \! \sqrt{n}) \to (0 , 0)}$ as ${n \to {+\infty}}$. 
   
   \medskip
   
   \item On the other hand, Proposition~\ref{prop:qc-cont} is no longer true 
   if $\Rn{n}$ is replaced by an arbitrary topological real vector space. 
   
   \noindent Indeed, if we consider the vector space 
   ${V \! \as \Cl{0}(\RR , \RR) \cap \cLL{2}(\RR , \RR) \inc \Rn{\RR}}$ 
   endowed with the topology $\cT \.$ of pointwise convergence 
   (this is nothing else than the product topology) 
   and the function ${\. f \. : V \! \to \RR}$ defined by ${\. f \. (u) \. \as \twonorm{u}^{\. 2}}$, 
   then $\. f \.$ is non-negative, positively homogeneous of degree ${\a \as 2}$ and sub-convex 
   (it is even strictly convex since for any ${u \in V \!}$ its Hessian at $u$ 
   with respect to the norm $\twonorm{\cdot} \.$ on $V \!$ is equal to ${2 \scal{\cdot \,}{\cdot}}$, 
   which is positive definite). 
   
   \noindent Nevertheless, $\. f \.$ is not continuous with respect to $\cT \.$ 
   for it is not even upper semi-continuous at the origin with respect to $\cT \.$ 
   (since the sequence $\seq{u_{n}}{n}{1}$ in $V \!$ defined in~\cite[Example~1.1, page~797]{SimVer18} 
   converges to zero with respect to $\cT \.$ 
   and satisfies ${\. f \. (u_{n}) \geq 4 > 0 = \. f \. (0)}$ for any integer ${n \geq 1}$). 
\end{enumerate} 
\end{remark}

\bigskip

We shall now prove that the situation described in Point~2 of Remark~\ref{rem:qc-cont} 
does not appear when sub-convexity is replaced by \emph{strict} sub-convexity. 

\bigskip

In order to do this, the following easy-to-prove result will be useful. 

\bigskip

\begin{lemma} \label{lem:assc} 
   Let $C \.$ be a convex subset of a topological real vector space, 
   ${\. f \. : C \to \RR}$ a non-negative function and $\a > 0$ a real number. 
   Then we have the equivalence 
   \[
   f^{\a} \! \ \mbox{is strictly sub-convex} 
   \qquad \iff \qquad \. 
   f \. \ \mbox{is strictly sub-convex}~.
   \] 
\end{lemma}

\bigskip

\begin{proposition} \label{prop:strict-sc+0-level-f-cont-riC} 
   Given a convex cone $C \.$ in a topological real vector space 
   and a non-negative function ${\. f \. : C \to \RR}$ 
   which is positively homogeneous of degree $\a > 0$, we have the implication 
   \[
   f \. \ \mbox{is strictly sub-convex} 
   \qquad \imp \qquad \. 
   f \. \ \mbox{is continuous on} \ \ri{C}~.
   \] 
\end{proposition}

\bigskip

\begin{proof} 
Assume that $\ri{C}$ is not empty (since otherwise there is nothing to prove) 
and that $\. f$ % no \. at the end of the line
is strictly sub-convex. 

\smallskip

First of all, notice that we have ${\Aff{C} = \Aff{\ri{C} \.}}$ 
by Point~2.b in Proposition~\ref{prop:ri-rc-rb} with ${S \as C \.}$. 

\smallskip

Now, in case when $C$ is not equal to $\{ 0 \}$, we have ${\Bcone{C} = C \,\. \setmin \{ 0 \}}$ 
and ${\Cone{C} = C \cup \{ 0 \}}$ by Point~4 in Remark~\ref{rem:cone} with ${A \as \. V \!\!}$, 
which yields the equalities ${\Aff{C \,\. \setmin \{ 0 \} \.} = \. \Vect{C}}$ 
and ${\Aff{C \cup \{ 0 \} \.} = \. \Vect{C}}$ 
according to Points~1 and~2 in Proposition~\ref{prop:abs-cone} with ${S \as C \.}$. 

\smallskip

Therefore, the inclusions ${C \,\. \setmin \{ 0 \} \inc C \inc C \cup \{ 0 \}}$ 
imply ${\Aff{C} = W \! \as \. \Vect{C}}$, 
and this latter equality still holds in case when we have ${C = \{ 0 \}}$. 

\smallskip

Since $\. f \.$ is sub-convex by Point~1 in Remark~\ref{rem:strict-sub-cvx}, 
the positively homogeneous function ${g \as \. f^{1 \. / \. \a}}$ % no \! at the end of the line
is also sub-convex by Corollary~\ref{cor:asc}, 
and hence it is convex owing to Proposition~\ref{prop:sc-c}. 

\smallskip

On the other hand, since $C$ is not empty, we have ${\Aff{S_{1 \.}(g) \.} = W \.}$ 
by Point~3 in Corollary~\ref{cor:1-sublevel-f-span-C}, 
and hence the interior of $S_{1 \.}(g)$ in $W$ is nothing else than ${U \as \ri{S_{1 \.}(g) \.}}$. 

\smallskip

Now, according to Lemma~\ref{lem:assc}, the function $g$ is strictly sub-convex, 
and hence the sublevel set $S_{1 \.}(g)$ is strictly convex. 

\smallskip

Therefore, since $S_{1 \.}(g)$ is not empty by Point~3 in Proposition~\ref{prop:ph-sl}, 
we get ${U \neq \Oset}$ by Proposition~\ref{prop:ri-strict-cvx} 
applied with the strictly convex subset $S_{1 \.}(g)$ of $W \!\!$. 

\smallskip

Finally, since $g$ is bounded from above (by the constant $\. 1$) 
on the non-empty open subset $U \!$ of the open convex subset $\ri{C}$ of $W \!\!$, 
it is continuous on $\ri{C}$ by Theorem~\ref{thm:conv-cont}, 
and this proves that $\. f \.$ is continuous on $\ri{C}$ too. 
\end{proof}

\bigskip

\begin{remark*} 
It is to be noticed that Proposition~\ref{prop:strict-sc+0-level-f-cont-riC} 
is no longer true if $\ri{C}$ is replaced by $C$ as Point~1 in Remark~\ref{rem:qc-cont} shows. 
\end{remark*}

\bigskip

A straightforward consequence of Proposition~\ref{prop:strict-sc+0-level-f-cont-riC} 
when $C$ is the whole space is the following result. 

\bigskip

\begin{corollary} \label{cor:strict-sc+0-level-f-cont} 
   Given a topological real vector space $V \!\.$ 
   and a non-negative function ${\. f \. : V \! \to \RR}$ 
   which is positively homogeneous of degree $\a > 0$, we have the implication 
   \[
   f \. \ \mbox{is strictly sub-convex} 
   \qquad \imp \qquad \. 
   f \. \ \mbox{is continuous}~.
   \] 
\end{corollary}

\bigskip

Let us now give the relationship between strict sub-convexity 
and strict convexity for positively homogeneous functions. 

\bigskip

\begin{proposition} \label{prop:strict-sc+0-level+cone-cvx-f-alpha-f-strict-c} 
   Given a convex cone $C \.$ in a topological real vector space 
   and a non-negative function $\. f \. : C \to \RR$ which is positively homogeneous, 
   we have the implication 
   %%%%%%%%%%
   \begin{eqnarray*} 
      f \. \ \mbox{is strictly sub-convex and} \ \. \inv{f}(0) \inc \{ 0 \} 
      & & 
      \\ 
      & & 
      \hspace{-30pt} \imp \. f^{\a} \! \ \mbox{is strictly convex for any real number} \ \a > 1~. 
   \end{eqnarray*} 
   %%%%%%%%%%
\end{proposition}

\bigskip

\begin{proof} 
Assume that $C$ is not empty (since otherwise the implication to be proved is trivial) 
and that $\. f \.$ is strictly sub-convex and satisfies ${\. \inv{f}(0) \inc \{ 0 \}}$, 
fix a real number ${\a > 1}$, and pick two distinct points ${x , y \in C \.}$. 

\smallskip

There are two cases to be considered 
depending on whether $\. f \. (x)$ and $\. f \. (y)$ are equal or not. 

\medskip

\noindent \textasteriskcentered \ Case when we have $\. f \. (x) \neq \. f \. (y)$. 

\smallskip

According to Point~1 in Remark~\ref{rem:strict-sub-cvx} and Corollary~\ref{cor:hsc-c}, 
the function $\. f \.$ is convex, which writes 
${\. f \. ( \. (1 - t) x + t y) \leq (1 - t) f \. (x) + t f \. (y)}$ 
for any ${t \in (0 , 1)}$, and hence we get 

\smallskip

\centerline{
$f^{\a} \. ( \. (1 - t) x + t y) 
\ \leq \ 
[ \. (1 - t) f \. (x) + t f \. (y) \. ]^{\. \a} \. 
\ < \ 
(1 - t) f^{\a} \. (x) + t f^{\a} \. (y)$
} 

\smallskip

since the function ${\f : [0 , {+\infty}) \to \RR}$ defined by ${\f(s) \. \as s^{\a} \!}$ 
is non-decreasing and strictly convex. 

\medskip

\noindent \textasteriskcentered \ Case when we have $\. f \. (x) = \. f \. (y)$. 

\smallskip

The property ${\. \inv{f}(0) \inc \{ 0 \}}$ together with ${x \neq y}$ imply ${r \as \. f \. (x) > 0}$, 
and hence we can write ${x , y \in S_{r}(f) \inc \rc{S_{r}(f) \.}}$, 
which yields ${(1 - t) x + t y \in \ri{S_{r}(f) \.}}$ for any ${t \in (0 , 1)}$ 
since $\. f \.$ is strictly sub-convex. 

\smallskip

Now, we have ${\widehat{S_{1 \.}(f)} = S_{1 \.}(f) \cup \{ 0 \}}$ 
by Point~3 in Proposition~\ref{prop:ph-sl}, 
and hence ${\rest{(p_{\widehat{S_{1 \.}(f)}}){\.}}{C}  = j \comp \. f}$ % no \. at the end of the line
owing to Corollary~\ref{cor:f=p} with ${S \as \widehat{S_{1 \.}(f)}}$, 
which yields ${\rest{(p_{S_{1 \.}(f)}){\.}}{C} = j \comp \. f \.}$ 
since one has ${p_{\widehat{S_{1 \.}(f)}} = p_{S_{1 \.}(f)}}$ 
according to Point~4 in Proposition~\ref{prop:gauge-vect} 
with ${S \as S_{1 \.}(f)}$ and ${V \! \as \. \Vect{C}}$. 

\smallskip

Therefore, we get the equalities ${\inv{p_{S_{1 \.}(f)}}([0 , 1]) \cap C = S_{1 \.}(f)}$ 
and ${\inv{p_{S_{1 \.}(f)}}([0 , 1) \. ) \cap C = \:\!\! \inv{f}([0 , 1) \. )}$. 

\smallskip

As a consequence, the first equality implies ${S_{1 \.}(f) \inc \inv{p_{S_{1 \.}(f)}}{([0 , 1])}}$, 
and hence we obtain the inclusion 
${\Aff{S_{1 \.}(f) \.} \inc \Aff{\inv{p_{S_{1 \.}(f)}}{([0 , 1])} \!}}$, 
which leads to ${\Vect{C} \inc \Aff{\inv{p_{S_{1 \.}(f)}}{([0 , 1])} \!}}$ 
since one has ${\Aff{S_{1 \.}(f) \.} = \. \Vect{C}}$ 
by Point~3 in Corollary~\ref{cor:1-sublevel-f-span-C}. 

\smallskip

On the other hand, using Point~1 in Proposition~\ref{prop:gauge-vect} with ${S \as S_{1 \.}(f)}$, 
Point~2 in Remark~\ref{rem:cone} and Proposition~\ref{prop:vect-vs-aff-bis} with ${S \as C \.}$, 
one can write 
%%%%%%%%%%
\begin{eqnarray*} 
   \inv{p_{S_{1 \.}(f)}}{([0 , 1])} 
   \ \inc \ 
   \inv{p_{S_{1 \.}(f)}}{([0 , {+\infty}) \. )} \! 
   & \inc & \! 
   \Bcone{S_{1 \.}(f) \.} \cup \{ 0 \} \\ 
   & = & \! 
   \Cone{S_{1 \.}(f) \.} \ \inc \ \Cone{C} 
   \ \inc \ 
   \Vect{{\Cone{C}} \.} \ = \ \. \Vect{C} 
\end{eqnarray*} 
%%%%%%%%%%
since we have ${S_{1 \.}(f) \inc C \.}$, and this yields 

\smallskip

\centerline{
$\Aff{\inv{p_{S_{1 \.}(f)}}{([0 , 1])} \!} 
\ \inc \ 
\Aff{\Vect{C} \.} 
\ \inc \ 
\Vect{\Vect{C} \.} \ = \ \. \Vect{C}$~.
} 

\smallskip

Therefore, we have obtained $\Aff{\inv{p_{S_{1 \.}(f)}}{([0 , 1])} \!} = \. \Vect{C \.}$. 

\smallskip

The inclusion ${S_{1 \.}(f) \inc \inv{p_{S_{1 \.}(f)}}{([0 , 1])}}$ 
then implies ${\ri{S_{1 \.}(f) \.} \inc \ri{\inv{p_{S_{1 \.}(f)}}{([0 , 1])} \!} \!\!}$ 
(notice that both interiors are computed in $\. \Vect{C} \.$), 
and hence ${\ri{S_{1 \.}(f) \.} \inc \ri{\inv{p_{S_{1 \.}(f)}}{([0 , 1])} \!} \cap C}$ 
since we have ${\ri{S_{1 \.}(f) \.} \inc S_{1 \.}(f) \inc C \.}$. 

\smallskip

Now, according to Point~2 in Lemma~\ref{lem:gauge-top} 
with ${S \as \widehat{S_{1 \.}(f)}}$ and ${V \! \as \. \Vect{C}}$,
we can write ${\ri{\inv{p_{S_{1 \.}(f)}}{([0 , 1])} \!} \inc \inv{p_{S_{1 \.}(f)}}{([0 , 1) \. )}}$ 
since one has ${p_{\widehat{S_{1 \.}(f)}} = p_{S_{1 \.}(f)}}$, which implies 

\smallskip

\centerline{
$\ri{S_{1 \.}(f) \.} 
\ \inc \ 
\inv{p_{S_{1 \.}(f)}}{([0 , 1) \. )} \cap C 
\ = \ \:\!\! 
\inv{f}([0 , 1) \. )$~.
} 

\smallskip

The positive homogeneity of $\. f \.$ and Point~3 in Proposition~\ref{prop:ap} then yield 

\smallskip

\centerline{
$\ri{S_{r}(f) \.} 
\ = \ 
r \act \ri{S_{1 \.}(f) \.} 
\ \inc \ 
r \act \inv{f}([0 , 1) \. ) 
\ = \ \:\!\! 
\inv{f}([0 , r) \. )$
} 

\smallskip

given that the homothety with ratio $r$ is a topological vector space automorphism 
which preserves $\. \Vect{C \.}$.

\smallskip

Finally, one obtains
${\. f^{\a} \. ( \. (1 - t) x + t y) < r^{\a} \! 
= \. f^{\a} \. (x) = (1 - t) f^{\a} \. (x) + t f^{\a} \. (y)}$ for any ${t \in (0 , 1)}$
since the function ${\f : [0 , {+\infty}) \to \RR}$ defined by 
${\f(s) \. \as s^{\a} \!}$ is increasing. 
\end{proof}

%\bigskip
\pagebreak

\begin{remark*} \ 

\begin{enumerate}[1)]
   \item Proposition~\ref{prop:strict-sc+0-level+cone-cvx-f-alpha-f-strict-c} 
   is not true without the assumption ${\. \inv{f}(0) \inc \{ 0 \}}$ 
   as we can easily check when considering the zero function defined on the real line. 
   
   \medskip
   
   \item It is to be noticed that the converse implication $\! \con \!$ 
   in Proposition~\ref{prop:strict-sc+0-level+cone-cvx-f-alpha-f-strict-c} is not true. 
   
   \noindent Indeed, let us consider the vector space 
   ${V \! \as \Cl{0}(\RR , \RR) \cap \cLL{2}(\RR , \RR) \inc \Rn{\RR}}$ 
   endowed with its natural scalar product $\scal{\cdot \,}{\cdot}$ 
   whose associated norm is denoted by $\twonorm{\cdot}$. 
   
   \noindent Then the function ${\. f \. : V \! \to \RR}$ defined by 
   ${\. f \. (u) \. \as \twonorm{u} \.}$ is non-negative and positively homogeneous. 
   
   \noindent Moreover, $\. f^{\a} \!$ is strictly convex for any real number ${\a > 1}$ 
   according to Point~2 in Remark~\ref{rem:sc}. 
   
   \noindent Nevertheless, $\. f \.$ is not strictly sub-convex 
   when $V \!$ is endowed with the topology $\cT \.$ of pointwise convergence 
   (this is nothing else than the product topology) 
   since if it were the case the function $\. f \.$ would be continuous with respect to $\cT \.$ 
   by Corollary~\ref{cor:strict-sc+0-level-f-cont}. 
   
   \noindent But $\. f \.$ is not even upper semi-continuous at the origin with respect to $\cT \.$ 
   since the sequence $\seq{u_{n}}{n}{1}$ in $V \!$ defined in~\cite[Example~1.1, page~797]{SimVer18} 
   converges to zero with respect to $\cT \.$ 
   and satisfies ${\. f \. (u_{n}) \geq 4 > 0 = \. f \. (0)}$ for any integer ${n \geq 1}$. 
\end{enumerate} 
\end{remark*}

\bigskip

Finally, if we sum up all the previous results, we eventually get the following useful property. 

\bigskip

\begin{theorem} \label{thm:strict-cvx-strict-qc-strict-subc} 
   Let $V \!\.$ be a topological real vector space 
   and ${\. f \. : V \! \to \RR}$ a non-negative function which is positively homogeneous. 
   Then for any real number ${\a > 1}$, the following conditions are equivalent: 
   
   \begin{enumerate}
      \item $\. f \.$ is continuous and $\. f^{\a} \!$ is strictly convex. 
      
      \smallskip
      
      \item $\. f \.$ is continuous and strictly quasi-convex. 
      
      \smallskip
      
      \item $\. f \.$ is strictly sub-convex and we have $\. \inv{f}(0) = \{ 0 \}$. 
   \end{enumerate} 
\end{theorem}

\bigskip

\begin{proof} \ 

\textsf{Point~1 $\imp$ Point~2.} 
Assume that Point~1 is satisfied. 

\smallskip

Then Point~2 is a mere consequence of Proposition~\ref{prop:strict-c-strict-qc} 
with $C \. \as \. V \!\.$. 

\medskip

\textsf{Point~2 $\imp$ Point~3.} 
Assume that Point~2 is satisfied. 

\smallskip

Then the function $\. f \.$ is strictly sub-convex 
by Proposition~\ref{prop:sqc+c-ssc} with $C \. \as \. V \!\.$. 

\smallskip

On the other hand, Remark~\ref{rem:ph-z} with ${C \. \as \. V \!}$ and Point~4 in Remark~\ref{rem:sqc} 
yield $\. \inv{f}(0) = \{ 0 \}$. 

\medskip

\textsf{Point~3 $\imp$ Point~1.} 
Assume that Point~3 is satisfied. 

\smallskip

Then Corollary~\ref{cor:strict-sc+0-level-f-cont} implies that the function $\. f \.$ is continuous 
since it is positively homogeneous of degree one. 

\smallskip

Moreover, Proposition~\ref{prop:strict-sc+0-level+cone-cvx-f-alpha-f-strict-c} 
with ${C \. \as \. V \!}$ shows that $\. f^{\a} \!$ is strictly convex. 
\end{proof}

\bigskip

\begin{remark} \ 

\begin{enumerate}[1)]
   \item It is to be pointed out that if Point~1 in Theorem~\ref{thm:strict-cvx-strict-qc-strict-subc} 
   is satisfied for some real number ${\a > 1}$ (for instance ${\a \as 2}$), 
   then it is satisfied for any real number ${\a > 1}$. 
   
   \medskip
   
   \item In the particular case where $V \!$ is equal to $\Rn{n}$ equipped with the usual topology, 
   the continuity in Point~1 in Theorem~\ref{thm:strict-cvx-strict-qc-strict-subc} can be dropped 
   since the strict convexity of $\. f^{\a} \!$ implies its convexity 
   by Point~1 in Remark~\ref{rem:sc}, 
   and hence that $\. f \.$ is continuous according to Corollary~\ref{cor:conv-cont-Rn}. 
   
   \medskip
   
   \item On the other hand, when $V \!$ is as in the previous point, 
   the continuity in Point~2 in Theorem~\ref{thm:strict-cvx-strict-qc-strict-subc} is useless 
   since the strict quasi-convexity of $\. f \.$ implies its quasi-convexity 
   by Point~1 in Remark~\ref{rem:sqc}, and hence its sub-convexity by Proposition~\ref{prop:sc-qc}. 
   Therefore, the positive homogeneity of $\. f \.$ and Proposition~\ref{prop:sc-c} 
   insure that $\. f \.$ is convex, and hence continuous according to Corollary~\ref{cor:conv-cont-Rn}. 
\end{enumerate} 
\end{remark}

\bigskip

Now, a natural question is to know whether Theorem~\ref{thm:strict-cvx-strict-qc-strict-subc} holds 
in the more general framework when the function $\. f \.$ is defined on a non-empty convex cone $C$ 
in $V \!$ and when the condition ${\. \inv{f}(0) = \{ 0 \}}$ in Point~3 
is replaced by ${\. \inv{f}(0) \inc \{ 0 \}}$. 

\noindent Unfortunately, this is actually not always true as we can check when considering the open convex cone 
${C \as \{ (x , y) \in \Rn{2} \st y > |x| \}}$ in ${V \! \as \Rn{2}}$ 
and the function ${\. f \. : C \to \RR}$ defined by 
${\disp \. f \. (x , y) \. \as \sqrt{x^{2} \. + y^{2}} \.}$ 
which satisfies ${\. \inv{f}(0) = \Oset \inc \{ 0 \}}$. 

\noindent Indeed, $\. f^{2} \!$ is strictly convex since its Hessian is positive definite, 
and hence $\. f \.$ is strictly quasi-convex by Proposition~\ref{prop:strict-c-strict-qc}. 
Moreover, $\. f \.$ is non-negative, positively homogeneous and continuous. 

\noindent Nevertheless, since the boundary of 
${S_{1 \.}(f) = C \cap \{ (x , y) \in \Rn{2} \st x^{2} \! + y^{2} \. \leq 1 \}}$ in $\Rn{2}$ 
contains the line segment ${[ \. (0 , 0) \,\. , \,\. (1 , 1) \. ]}$, 
the function $\. f \.$ is not strictly sub-convex. 

\noindent This shows that even if the condition ${\. \inv{f}(0) = \{ 0 \}}$ in Point~3 
in Theorem~\ref{thm:strict-cvx-strict-qc-strict-subc} is replaced by ${\. \inv{f}(0) \inc \{ 0 \}}$, 
the implication Point~2 $\imp$ Point~3 is not satisfied by $\. f \!$. 

\bigskip

\noindent However, in the case where the convexity of $C$ is replaced by the \emph{strict} convexity, 
Theorem~\ref{thm:strict-cvx-strict-qc-strict-subc} still holds. 

\bigskip

\begin{corollary} \label{cor:strict-cvx-strict-qc-strict-subc-epi} 
   Let $C \.$ be a non-empty strictly convex cone in a topological real vector space $V$ 
   % no \!\. at the end of the line
   and ${\. f \. : C \to \RR}$ a non-negative function which is positively homogeneous. 
   Then for any real number ${\a > 1}$, the following conditions are equivalent: 
   
   \begin{enumerate}
   \item $\. f \.$ is continuous and $\. f^{\a} \!$ is strictly convex. 
   
   \smallskip
   
   \item $\. f \.$ is continuous and strictly quasi-convex. 
   
   \smallskip
   
   \item $\. f \.$ is strictly sub-convex and we have $\. \inv{f}(0) \inc \{ 0 \}$. 
   \end{enumerate} 
\end{corollary}

\bigskip

\begin{proof} 
Let us consider the vector subspace $W \! \as \. \Vect{C}$ of $V \!\.$. 

\smallskip

Owing to Proposition~\ref{prop:ri-strict-cvx}, $\ri{C}$ is not empty, 
and hence Proposition~\ref{prop:strict-cvx-cone} implies that either $C$ is equal to $W \!$ 
or ${C \cup \{ 0 \}}$ is a ray and $W \!$ is one-dimensional. 

\medskip

\noindent \textasteriskcentered \ 
In case when $C$ is equal to $W \!\!$, 
we may apply Theorem~\ref{thm:strict-cvx-strict-qc-strict-subc} 
with $W \!$ instead of $V \!\!$, and then obtain the equivalence of Points~1, 2 and~3 
since the inclusion ${\{ 0 \} \inc \inv{f}(0)}$ already holds 
by Remark~\ref{rem:ph-z} with ${C \. \as \. V \!\.}$. 

\medskip

\noindent \textasteriskcentered \ 
Now, in case when ${C \cup \{ 0 \}}$ is a ray and $W \!$ is one-dimensional, 
the situation is the same as if we had ${W \! = \RR}$ 
together with ${C \cup \{ 0 \} = [0 , {+\infty})}$ 
and if the function ${\. f \. : C \to \RR}$ were defined by $\. f \. (t) \. \as a t$, 
where $a$ is a non-negative real constant. 

\smallskip

On the other hand, since the open set $\ri{C}$ in $W \!$ is not empty 
and since we have ${\ri{C} \inc C \neq W \!\!}$, the topology on $W \!$ is not trivial, 
and hence it is equal to the usual topology on $\RR$ according to Point~4 in Remark~\ref{rem:0-tvs}. 

\smallskip

Moreover, if $a$ is positive, then Points~1, 2 and~3 are all satisfied by $\. f \!$, 
and if we have ${a = 0}$ then none of these points are satisfied by $\. f \!$. 

\smallskip

This proves that Points~1, 2 and~3 are equivalent. 
\end{proof}

\bigskip

%%%%%%%%%%%%%%%%%%%%

\section{Application to Minkowski norms} \label{sec:Application-to-Minkowski-norms}

So far, we have been dealing with positively homogeneous functions, 
and we obtained accurate relationships between strict convexity and strict sub-convexity, 
that is, between geometric and topological aspects of these functions. 

\smallskip

We shall now apply all these properties to the particular case of Minkowski norms 
in order to characterize those ones which are strictly convex and to give a generalization 
of a result that Carothers proved for normed vector spaces (see~\cite[Theorem~11.1, page~110]{Car04}). 

\bigskip

The first characterization of strict convexity for Minkowski norms 
on arbitrary topological real vector spaces is stated as follows. 

\medskip %\bigskip

\begin{theorem} \label{thm:Carothers-generalized} 
   Given a topological real vector space $V \!\.$ and a Minkowski norm $N \!$ on $V \!\!$, 
   the following properties are equivalent: 
   
   \begin{enumerate}
   \item $N \!$ is strictly sub-convex. 
   
   \smallskip
   
   \item $N \!$ is continuous, and for any two vectors ${x \neq y}$ in $V \!\.$ 
   which satisfy ${N \. (x) = N \. (y) = 1}$, 
   we have the inequality ${N \. ( \. (x + y) \! / \. 2) < 1}$. 
   \end{enumerate} 
\end{theorem}

\smallskip %\bigskip

\begin{proof} 
First of all, ${S \as \inv{N}([0 , 1]) = S_{1 \.}(N)}$ is a star-shaped subset of $V \!$ 
by Point~2 in Proposition~\ref{prop:ph-sl} with ${C \as V \!}$ and ${\. f \as N \.}$ 
since $S_{1 \.}(N)$ contains the origin, and hence the gauge function $p_{S}$ of $S$ 
satisfies ${p_{S}(x) = N \. (x)}$ for any ${x \in V \!}$ 
owing to Proposition~\ref{prop:f=p} with ${C \as V \!\.}$. 

\smallskip

Moreover, notice that we have ${\Aff{S} = \. V \!}$ by Proposition~\ref{prop:1-sublevel-f-span-V} 
with ${\. f \. \as \. N \!}$, which yields ${\ri{S} = \intr{S} \.}$, ${\rc{S} = \clos{S}}$ 
and ${\rb{S} = \clos{S} \setmin \intr{S} = \bd{S} \.}$. 

\medskip

\textsf{Point~1}$\imp$\textsf{Point~2.} 
Assume that Point~1 is satisfied, and let ${x \neq y}$ be two vectors in $V \!$ 
such that we have ${N \. (x) = N \. (y) = 1}$. 

\smallskip

Then Corollary~\ref{cor:strict-sc+0-level-f-cont} with ${\a \as 1}$ insures that $N \.$ is continuous, 
which implies that we have both ${\intr{S} = \inv{N}([0 , 1) \. )}$ and ${\bd{S} = \inv{N}(1)}$ 
by Points~2.b and~2.c in Proposition~\ref{prop:gauge-top}. 

\smallskip

Since we have ${x , y \in \inv{N}(1) = \bd{S} \inc \clos{S} = \rc{S}}$ 
and since $S$ is strictly convex, 
the open line segment ${]x , y[}$ lies in ${\ri{S} = \intr{S} = \inv{N}([0 , 1) \. )}$. 

\smallskip

Hence, we get in particular ${N \. ( \. (x + y) \! / \. 2) < 1}$ 
since the condition ${x \neq y}$ implies ${(x + y) \! / \. 2 \in {]x , y[}}$. 

\medskip

\textsf{Point~2}$\imp$\textsf{Point~1.} 
Suppose that Point~2 is true, and let us first prove that $S$ is strictly convex. 

%\smallskip

Since $N \.$ is continuous, we have both ${\intr{S} = \inv{N}([0 , 1) \. )}$ 
and ${\bd{S} = \inv{N}(1)}$ by Points~2.b and~2.c in Proposition~\ref{prop:gauge-top}. 

\smallskip 

Then, given ${x , y \in \bd{S} \.}$, one has ${N \. (x) = N \. (y) = 1}$, 
which yields ${N \. ( \. (x + y) \! / \. 2) < 1}$. 
So this reads ${(x + y) \! / \. 2 \in \inv{N}([0 , 1) \. ) = \intr{S} \.}$, 
and hence we get ${{]x , y[} \cap \intr{S} \neq \Oset \.}$. 

\smallskip

On the other hand, since $N \.$ is positively homogeneous and sub-additive, 
it is convex by Proposition~\ref{prop:phs-c} with ${C \as V \!\!}$, 
which implies that $S$ is a convex subset of $V \!$ as the pre-image by $N \.$ 
of the convex subset $[0 , 1]$ of $\RR$. 

%\smallskip

Therefore, using the implication in Remark~\ref{rem:cvx-open-segment}, 
we get ${{]x , y[} \inc \intr{S} =  \ri{S}}$. 

\smallskip

So, owing to Point~4 in Remark~\ref{rem:strict-cvx}, we have proved that $S$ is strictly convex. 

\smallskip

Now, for any ${r \in (0 , {+\infty})}$, we have ${S_{r}(N) = r S}$ 
by Point~3 in Proposition~\ref{prop:ap} with ${C \as V \!\!}$, 
${\. f \. \as \. N \.}$ and ${\a \as 1}$, which shows that $S_{r}(N)$ is also strictly convex 
since the homothety of $V$ % no \! at the end of the line
with ratio $r$ is an affine homeomorphism. 

\smallskip

Moreover, since we have ${\inv{N}(0) = \{ 0 \}}$, 
the sublevel set $S_{0}(N)$ reduces to $\{ 0 \}$, and hence is strictly convex. 

\smallskip

Finally, for any ${r \in ({-\infty} , 0)}$, the sublevel set $S_{r}(N)$ is empty, 
and hence is strictly convex too. 

\smallskip

This proves that $N \.$ is strictly sub-convex. 
\end{proof}

%\bigskip
\pagebreak

\begin{remark} \label{rem:Carothers-generalized} \ 

\begin{enumerate}[1)] 
   \item When the Minkowski norm $N \.$ in Theorem~\ref{thm:Carothers-generalized} is reduced to a norm 
   and when $V \!$ is endowed with the topology associated with $N \!$, then we obtain a generalization 
   of the characterization of ``strictly convex'' normed vector spaces (see~\cite[page~108]{Car04} 
   and~\cite[page~30]{JohLin01} for the definition) given by Point~i 
   in~\cite[Theorem~11.1, page~110]{Car04}. This is indeed a consequence 
   of Point~3 in Remark~\ref{rem:strict-sub-cvx}. 
   
   \medskip
   
   \item It is to be noticed that if we drop continuity 
   in Point~2 in Theorem~\ref{thm:Carothers-generalized}, 
   the implication Point~2$\imp$Point~1 is no longer true. 
   
   \noindent Indeed, if we consider the vector space 
   ${V \! \as \Cl{0}(\RR , \RR) \cap \cLL{2}(\RR , \RR) \inc \Rn{\RR}}$ 
   endowed with the topology $\cT \.$ of pointwise convergence 
   (this is nothing else than the product topology) 
   and the function ${\. f \. : V \! \to \RR}$ defined by ${\. f \. (u) \. \as \. \twonorm{u}^{\. 2}}$, 
   then $\. f \.$ is not continuous with respect to $\cT$ % no \. at the end of the line
   for it is not even upper semi-continuous at the origin with respect to $\cT \.$ 
   (since the sequence $\seq{u_{n}}{n}{1}$ in $V \!$ defined in~\cite[Example~1.1, page~797]{SimVer18} 
   converges to zero with respect to $\cT \.$ and satisfies 
   ${\. f \. (u_{n}) \geq 4 > 0 = \. f \. (0)}$ for any integer ${n \geq 1}$). 
   
   \noindent On the other hand, $\. f \.$ is strictly convex since for any ${u \in V \!}$ 
   its Hessian at $u$ with respect to the norm $\twonorm{\cdot} \.$ on $V \!$ 
   is equal to ${2 \scal{\cdot \,}{\cdot}}$, which is positive definite. 
   
   \noindent Now, let us consider the Minkowski norm $N \.$ on $V \!$ defined by 
   $N \. (u) \. \as \twonorm{u}$. 
   
   \noindent Then, for any two vectors ${u \neq v}$ in $V \!$ 
   which satisfy ${N \. (u) = N \. (v) = 1}$, we get 
   
   \smallskip
   
   \centerline{
   $f \. ( \. (u + v) \! / \. 2) \ < \ [f \. (u) + \. f \. (v) \. ] \. / \. 2 \ = \ 1$
   } 
   
   \smallskip
   
   \noindent by the strict convexity of $\. f \!$, 
   and hence ${N \. ( \. (u + v) \! / \. 2) < 1}$ holds since we have ${\. f = N^{2} \!}$. 
   
   \noindent However, if $N \.$ were strictly sub-convex, then it would be continuous 
   by Proposition~\ref{prop:strict-sc+0-level-f-cont-riC} with ${C \as V \!\!}$, 
   ${\. f \. \as \. N \.}$ and ${\a \as 1}$, 
   and hence the function $N^{2} \! = \twonorm{\cdot}^{\. 2}$ would be continuous too, 
   which is not the case as shown above. 
   
   \medskip
   
   \item Even though all the norms on a real vector space are not strictly sub-convex, 
   it is well known that any norm which gives rise to a separable Banach space is actually equivalent 
   to a smooth and strictly sub-convex norm (see for example~\cite[page~33]{JohLin01}). 
   
   \noindent Therefore, since there are many separable Banach spaces 
   (especially in functional analysis), 
   this shows that there is a lot of Banach spaces whose norm is strictly sub-convex. 
   This is a good reason for which strictly sub-convex norms are worth being studied, 
   not to mention the fact that it is more convenient to deal with such norms. 
\end{enumerate} 
\end{remark}

\bigskip

The second characterization of strict convexity for Minkowski norms 
on arbitrary topological real vector spaces is given by the following result. 

\bigskip

\begin{theorem} \label{thm:Minkowski-norm-strict-sub-cvx} 
   Let $V \!\.$ be a topological real vector space and ${N : V \! \to \RR}$ a function. 
   Then for any real number ${\a > 1}$, the following conditions are equivalent: 
   
   \begin{enumerate}
   \item $N \!$ is strictly sub-convex and is a Minkowski norm on $V \!\!$. 
   
   \smallskip
   
   \item $N \!$ is non-negative, positively homogeneous, 
   continuous and $N^{\a} \!$ is strictly convex. 
   \end{enumerate} 
\end{theorem}

% ===> HAL ne gère pas \\ avec l'environnement multline !!!
%
%
%\begin{theorem} \label{thm:Minkowski-norm-strict-sub-cvx} 
%   Let $V \!\.$ be a topological real vector space and ${N : V \! \to \RR}$ a function. 
%   Then for any real number ${\a > 1}$, the following equivalence holds: 
%   %%%%%%%%%%
%   \begin{multline*} 
%      N \. \ \mbox{is strictly sub-convex and is a Minkowski norm on} \ V \\ 
%      \iff \ N \. \ \mbox{is non-negative, positively homogeneous,} \\ 
%      \mbox{continuous and} \ N^{\a} \! \ \mbox{is strictly convex}~. 
%   \end{multline*} 
%   %%%%%%%%%%
%\end{theorem}

\bigskip

Of course, once $N \.$ satisfies Point~2, it automatically satisfies both Points~2 and~3 
in Theorem~\ref{thm:strict-cvx-strict-qc-strict-subc}. 

\bigskip

\begin{proof} 
Let us fix a real number $\a > 1$. 

\medskip

\textsf{Point~1}$\imp$\textsf{Point~2.} 
Assume that Point~1 is satisfied. 

\smallskip

Since we have ${\inv{N}(0) = \{ 0 \}}$, the function $N \.$ is continuous 
and $N^{\a} \!$ is strictly convex owing to the implication Point~3$\imp$Point~1 
in Theorem~\ref{thm:strict-cvx-strict-qc-strict-subc}. 

%\medskip
\pagebreak

\textsf{Point~2}$\imp$\textsf{Point~1.} 
Assume that Point~1 is satisfied. 

\smallskip

Then $N \.$ is sub-convex by Point~1 in Remark~\ref{rem:strict-sub-cvx} with ${C \as V \!\!}$, 
and hence it is convex owing to Proposition~\ref{prop:sc-c} with ${C \as V \!\.}$. 

\smallskip

Therefore, Proposition~\ref{prop:phs-c} with ${C \as V \!}$ implies that $N \.$ is sub-additive. 

\smallskip

On the other hand, the strict sub-convexity of $N \.$ and the equality ${\inv{N}(0) = \{ 0 \}}$ 
are a consequence of the implication Point~1$\imp$Point~3 
in Theorem~\ref{thm:strict-cvx-strict-qc-strict-subc}. 
\end{proof}

\bigskip

\begin{remark} \ 

\begin{enumerate}[1)]
   \item When the function $N \.$ in Theorem~\ref{thm:Minkowski-norm-strict-sub-cvx} 
   is reduced to a norm and when $V \!$ is endowed with the topology associated with $N \!$, 
   then we obtain a generalization of the characterization of ``strictly convex'' normed vector spaces 
   (see~\cite[page~108]{Car04} and~\cite[page~30]{JohLin01} for the definition) given by Point~ii 
   in~\cite[Theorem~11.1, page~110]{Car04}. This is indeed a consequence 
   of Point~3 in Remark~\ref{rem:strict-sub-cvx}. 
   
   \medskip
   
   \item It is to be noticed that a function ${N \. : V \! \to \RR}$ 
   which satisfies the equivalent conditions 
   in Theorem~\ref{thm:Minkowski-norm-strict-sub-cvx} is a continuous Minkowski norm on $V \!\!$, 
   which shows that the topology $\cT \.$ defined by $N \.$ (see Introduction) 
   is coarser than the vector space topology on $V \!\.$. 
\end{enumerate} 
\end{remark}

\bigskip

%%%%%%%%%%%%%%%%
% Bibliography %
%%%%%%%%%%%%%%%%

\bibliographystyle{acm}
\bibliography{math-biblio-convexity}

% Document ends here
\end{document}